\documentclass{amsart}
\usepackage{todonotes}
\usepackage{amsthm}
\usepackage{amssymb}
\usetikzlibrary{quotes,angles}
\usepackage{amsmath}
\usepackage{amsfonts}
\usepackage[none]{hyphenat}

\usepackage{textcomp}
\usepackage[T1]{fontenc}
\usepackage[utf8x]{inputenc}
\usepackage{textcomp}
\usepackage{wasysym}
\usepackage{stmaryrd}
\usepackage{esint}
\usepackage[all]{xy}
\usepackage{graphicx}
\usepackage{bbm}
\usepackage{xcolor,lineno}
\usepackage{ulem}
\usepackage{comment}

\newtheorem{theorem}{Theorem}[section]
\newtheorem{proposition}[theorem]{Proposition}
\newtheorem{lemma}[theorem]{Lemma}

\newtheorem{example}[theorem]{Example}

\definecolor{wineRed}{rgb}{0.7,0,0.3}
\newcommand{\PR}[1]{{\color{wineRed}#1}}

\theoremstyle{definition}
\newtheorem{definition}[theorem]{Definition}
\newtheorem{remark}[theorem]{Remark}

\newcommand{\authorfootnotes}{\renewcommand\thefootnote{\@fnsymbol\c@footnote}}%

\newcommand{\cC}{\mathcal C}

\newcommand{\cH}{\mathcal H}
\newcommand{\res}{\mathop{\hbox{\vrule height 7pt width .5pt depth 0pt \vrule height .5pt width 6pt depth 0pt}}\nolimits}

\def\dist{\mbox{dist}\,}
\def\conv{\mbox{conv}\,}

\newcommand{\bN}{\mathbb N}
\newcommand{\R}{\mathbb R}
\newcommand{\bR}{\mathbb R}

\newcommand{\bbT}{{\bf T}}

\newcommand{\lb}{\langle}
\newcommand{\rb}{\rangle}
\newcommand{\LC}{\hbox{\Large$\llcorner$}}

\newcommand{\eps}{\varepsilon}

\DeclareMathOperator{\spt}{spt}
\DeclareMathOperator{\Lip}{Lip}

\DeclareMathOperator{\diam}{diam}

\usepackage{pgf,tikz}
\usetikzlibrary{arrows}
\definecolor{zzttqq}{rgb}{0.6,0.2,0}
\definecolor{uququq}{rgb}{0.25,0.25,0.25}
\definecolor{qqqqff}{rgb}{0,0,1}
\definecolor{xdxdff}{rgb}{0.49,0.49,1}
\usepackage{pgf,tikz}
\usetikzlibrary{arrows}
\definecolor{ffffqq}{rgb}{1,1,0}
\definecolor{fftttt}{rgb}{1,0.2,0.2}
\definecolor{ffqqqq}{rgb}{1,0,0}
\definecolor{ffqqtt}{rgb}{1,0,0.2}
\definecolor{ffttqq}{rgb}{1,0.2,0}
\definecolor{xdxdff}{rgb}{0.49,0.49,1}
\definecolor{qqqqff}{rgb}{0,0,1}
\definecolor{ffttww}{rgb}{1,0.2,0.4}
\definecolor{ttqqcc}{rgb}{0.2,0,0.8}
\definecolor{ffqqww}{rgb}{1,0,0.4}
\definecolor{xdxdff}{rgb}{0.49,0.49,1}
\definecolor{qqqqff}{rgb}{0,0,1}
\definecolor{uququq}{rgb}{0.25,0.25,0.25}
\definecolor{ttffww}{rgb}{0.2,1,0.4}
\definecolor{ffqqqq}{rgb}{1,0,0}
\definecolor{qqfftt}{rgb}{0,1,0.2}
\definecolor{zzttqq}{rgb}{0.6,0.2,0}
\definecolor{ffffqq}{rgb}{1,1,0}
\definecolor{qqffww}{rgb}{0,1,0.4}
\definecolor{xdxdff}{rgb}{0.49,0.49,1}
\definecolor{qqqqff}{rgb}{0,0,1}

\definecolor{ffqqqq}{rgb}{1,0,0}
\definecolor{ttqqff}{rgb}{0.2,0,1}
\definecolor{ttqqcc}{rgb}{0.2,0,0.8}
\definecolor{uququq}{rgb}{0.25,0.25,0.25}
\definecolor{xdxdff}{rgb}{0.49,0.49,1}
\definecolor{qqqqff}{rgb}{0,0,1}
\definecolor{ffqqtt}{rgb}{1,0,0.2}

\definecolor{qqffqq}{rgb}{0,1,0}
\definecolor{qqfftt}{rgb}{0,1,0.2}
\definecolor{ffffqq}{rgb}{1,1,0}
\definecolor{ffttcc}{rgb}{1,0.2,0.8}
\definecolor{ffwwqq}{rgb}{1,0.4,0}
\definecolor{ffqqqq}{rgb}{1,0,0}
\definecolor{ffffqq}{rgb}{1,1,0}
\definecolor{qqfftt}{rgb}{0,1,0.2}
\definecolor{ffttcc}{rgb}{1,0.2,0.8}
\definecolor{ffwwqq}{rgb}{1,0.4,0}
\definecolor{ffqqqq}{rgb}{1,0,0}
\usepackage{subfig}
\usetikzlibrary{arrows}
\definecolor{ffqqtt}{rgb}{1,0,0.2}
\definecolor{ffffqq}{rgb}{1,1,0}
\definecolor{ttffqq}{rgb}{0.2,1,0}
\definecolor{wwwwww}{rgb}{0.4,0.4,0.4}
\definecolor{zzqqzz}{rgb}{0.6,0,0,6}
\usepackage{mathtools}

\usepackage[active]{srcltx}
\usepackage{bbm}
\usepackage{amsfonts}
\usepackage{accents,eqnarray}
\usetikzlibrary{arrows}
\usepackage{amssymb}
\usepackage[colorlinks=true,urlcolor=blue,
citecolor=red,linkcolor=blue,linktocpage,pdfpagelabels,bookmarksnumbered,bookmarksopen]{hyperref}
\usepackage{graphicx, tikz}
\usepackage{pgf,tikz}
\usetikzlibrary{arrows}
\definecolor{wwqqww}{rgb}{0.4,0,0,4}
\definecolor{wwffqq}{rgb}{0.4,1,0}
\definecolor{fffftt}{rgb}{1,1,0.2}
\definecolor{ccqqcc}{rgb}{0.8,0,0.8}
\definecolor{ffffww}{rgb}{1,1,0.4}
\definecolor{ttfftt}{rgb}{0.2,1,0.2}
\definecolor{fffftt}{rgb}{1,1,0.2}
\definecolor{ccqqww}{rgb}{0.8,0,0.4}
\definecolor{ffffww}{rgb}{1,1,0.4}
\definecolor{ttfftt}{rgb}{0.2,1,0.2}
\definecolor{uququq}{rgb}{0.25,0.25,0.25}
\definecolor{ffqqtt}{rgb}{1,0,0.2}
\definecolor{ttffff}{rgb}{0.2,1,1}
\definecolor{fffftt}{rgb}{1,1,0.2}
\definecolor{wwffqq}{rgb}{0.4,1,0}
\definecolor{qqqqff}{rgb}{0,0,1}
\definecolor{uququq}{rgb}{0.25,0.25,0.25}
\definecolor{ffqqtt}{rgb}{1,0,0.2}
\definecolor{fffftt}{rgb}{1,1,0.2}
\definecolor{wwffqq}{rgb}{0.4,1,0}
\definecolor{qqqqff}{rgb}{0,0,1}
\definecolor{ttfftt}{rgb}{0.2,1,0.2}
\definecolor{qqqqff}{rgb}{0,0,1}
\definecolor{ffqqtt}{rgb}{1,0,0.2}
\definecolor{fffftt}{rgb}{1,1,0.2}
\definecolor{ffffqq}{rgb}{1,1,0}
\definecolor{qqfftt}{rgb}{0,1,0.2}
\tikzset { domaine/.style 2 args={domain=#1:#2} }
\tikzset{
xmin/.store in=\xmin, xmin/.default=-3, xmin=-3,
xmax/.store in=\xmax, xmax/.default=3, xmax=3,
ymin/.store in=\ymin, ymin/.default=-3, ymin=-3,
ymax/.store in=\ymax, ymax/.default=3, ymax=3,
}
\definecolor{ffffqq}{rgb}{1,1,0}
\definecolor{ttfftt}{rgb}{0.2,1,0.2}
\definecolor{qqffcc}{rgb}{0,1,0.8}
\definecolor{ffwwqq}{rgb}{1,0.4,0}
\definecolor{ffqqqq}{rgb}{1,0,0}
\definecolor{ttffqq}{rgb}{0.2,1,0}
\definecolor{qqqqff}{rgb}{0,0,1}
\definecolor{ffqqtt}{rgb}{1,0,0.2}
\definecolor{ffffqq}{rgb}{1,1,0}
\definecolor{qqfftt}{rgb}{0,1,0.2}
\usepackage[english]{babel}
\usepackage{marginnote}
\usepackage[left=2.95cm,right=2.95cm,top=2.8cm,bottom=2.8cm]{geometry}
\numberwithin{equation}{section}
\usepackage{mathrsfs}
\usepackage{color}
\usepackage{graphicx}%
\providecommand{\U}[1]{\protect\rule{.1in}{.1in}}
\definecolor{linkcolor}{rgb}{0.00,0.50,0.00}
\providecommand{\U}[1]{\protect\rule{.1in}{.1in}}
\definecolor{ffffqq}{rgb}{1,1,0}
\definecolor{qqffqq}{rgb}{0,1,0}
\definecolor{ttttff}{rgb}{0.2,0.2,1}
\definecolor{ffqqtt}{rgb}{1,0,0.2}

\def\XXint#1#2#3{{\setbox0=\hbox{$#1{#2#3}{\int}$}
     \vcenter{\hbox{$#2#3$}}\kern-.5\wd0}}

\usepackage{geometry}
\def\bR{\mathbb{R}}

\def\cN{{\mathcal{N}}}
\def\cH{{\mathcal{H}}}

\def\cM{{\mathcal{M}}}

\def\cR{{\mathcal{R}}}

\usepackage{mathrsfs}
\usetikzlibrary{arrows}
\numberwithin{equation}{section}

\newcommand{\tarc}{\mbox{\large$\frown$}}
\newcommand{\arc}[1]{\stackrel{\tarc}{#1}}

\begin{document}
\title{The non-convex planar Least Gradient Problem}
\author{Samer Dweik} 
\address{Samer Dweik, Department of Mathematics and Statistics, College of Arts and Sciences, Qatar University,
2713, Doha, Qatar}
\email{sdweik@qu.edu.qa}
%\thanks{email 1}

% second author
\author{Piotr Rybka}
\address{Faculty of Mathematics, Informatics and Mechanics, University of Warsaw\\ ul. Banacha 2, 02-097 Warsaw, Poland\\
ORCiD ID \url{https://orcid.org/0000-0002-0694-8201}}
\email{rybka@mimuw.edu.pl}

\author{Ahmad Sabra}
\address{Ahmad Sabra, Department of Mathematics,  Faculty of Arts and Sciences, American University of Beirut, Fisk Hall 305, PO Box 11-0236, Riad El Solh, Beirut 1107 2020 Lebanon, ORCiD ID \url{https://orcid.org/0000-0002-4669-4912}}
\email{asabra@aub.edu.lb}


\maketitle

%\title{The planar Least Gradient problem in convex domains, the case of continuous datum}
%\author{Samer Dweik, Piotr Rybka, Ahmad Sabra}

%opening
%\title{The non-convex planar Least Gradient Problem}
%\author{Samer Dweik \thanks{University of Qatar, Doha Qatar}, Piotr Rybka, Ahmad Sabra}
 
 %\footnote{University of Qatar, Doha Qata, email:sdweik@qu.edu.qa}. Piotr Rybka\footnote{Faculty of Mathematics, Informatics and  Mechanics, the University of Warsaw, Warsaw, Poland, email: rybka@mimuw.edu.pl}, Ahmad Sabra\footnote{Department of mathematics, Center for Advanced Mathematical Sciences,  American University of Beirut, Beirut, Lebanon. email: asabra@aub.edu.lb}

%\author{SD, PR, AS}
\maketitle
%\centerline{\today}

%\maketitle
\begin{abstract}
We study the least gradient problem in bounded regions with Lipschitz boundary in the plane.
We provide a set of conditions for the existence of solutions in non-convex simply connected regions. We assume the boundary data is continuous and in the space of functions of bounded variation, and we are interested in solutions that satisfy the boundary conditions in the trace sense. Our method relies on the equivalence of the least gradient problem and the Beckman problem which allows us to use the tools of the optimal transportation theory. 

%To do so, we approach the problem first in the strictly convex case  
%We state a set of necessary and sufficient conditions for existence of a unique solution in strictly convex regions when the boundary condition is continuous and in the space of functions of
%bounded variation. We are interested in solutions which satisfy the boundary conditions in the trace sense.

%We extend our results to convex domains by means of approximation. Finally then relax the strict convexity condition and provide necessary and sufficient conditions for convex domains as well as for domainsFor merely convex regions we state `almost' sufficient and necessary conditions. Finally, we present a quite set of conditions sufficient conditions for existence in non-convex contractible regions. Our method depends on the equivalence of the least gradient problem and the Beckman problem (also known as Free Material Design). We may apply the tools of the optimal transportation theory to the former one. However, this tool requires the data to be in the $BV$ space.
\end{abstract}

\section{Introduction} %\marginnote{\PR{this is not the\\ final ver. of \S1}}
In this paper, we study  the following least gradient problem in bounded Lipschitz planar domains
\begin{equation} \label{LGP}
   \min\bigg\{ \int_{\Omega} |D u| \,:\,u \in BV(\Omega),\,\,u|_{\partial \Omega}=g \bigg\},
\end{equation}
where $g$ is assumed to be continuous on $\partial \Omega$ %\marginnote{rather $W^{1,1}$} 
and the restriction $u|_{\partial \Omega}$ represents the trace of $u$ on $\partial \Omega$. The interest in \eqref{LGP} stems from pure mathematics, e.g. as a limiting case of the $p$-harmonic functions when $p\to 1$, as well as from science and engineering. It is known that \eqref{LGP} is equivalent to the following Beckman problem also called Free Material Design problem, where one has to distribute the conductive material in an optimal way (see \cite{czle})
\begin{equation}\label{beck}
\min\left\{ \int_{\bar \Omega} | {v}|: \nabla \cdot { v} = 0,\  {v}\cdot \nu = 
\partial_\tau g\,\,\text{ on \,$\partial \Omega$}\right\},
\end{equation}
where $\nu$ is the unit normal and $\tau$ is the unit tangent to $\partial\Omega$ such that the system $(\nu, \tau)$ is positively oriented. This equivalence valid for two-dimensional problems was noted in \cite{grs} for convex domains and later in \cite{DweGor} for simply connected ones. The main idea of the equivalence proof is that $u$ is a solution to the least gradient problem if and only if the rotation of measure $Du$ by $-\pi/2$, is a solution to the Beckman problem. There the assumption that $\Omega$ is simply connected is needed to write every closed form $\omega=v_2\,dx-v_1\, dy$ with $v=(v_1,v_2)\in L^1(\Omega,\mathbb R^2)$ satisfying $\nabla \cdot v=0$ as a potential of a function $u\in W^{1,1}(\Omega)$.

In the seminal paper \cite{sternberg} the authors showed the existence of minimizes to \eqref{LGP} in case the boundary datum $g$ is continuous and the mean curvature of $\partial \Omega$ is non-negative, no part of $\partial\Omega$ is a minimal surface. In case $\Omega$ is contained 
in the plane, these conditions reduce to strict convexity of $\Omega$. This result was later relaxed in \cite{rs1} and \cite{rs2} where the domain is assumed to be a convex polygon. In that case, the polygon $\Omega$ was approximated by a sequence of strictly convex domains $\Omega_n$ with boundary data $g_n$ to ensure the convergence of the solution to \eqref{LGP} on $(\Omega_n,g_n)$ to a solution on $(\Omega,g)$. Using such an argument, sufficient conditions on $\Omega$ and $g$ were obtained to show the existence of solutions to \eqref{LGP}.

%In case $\Omega$ is contained 
%in the plane, these conditions reduce to strict convexity of $\Omega$. 
%Our set of restrictions imposed on the data $(\Omega, g)$ of \eqref{LGP} is also geometric but of a different nature. We surely need the boundary of $\Omega$ to be Lipschitz continuous to construct the trace operator and define the datum $f$ as a tangential derivative of $g$. The most important geometric condition is related to the need to guarantee equivalence of \eqref{LGP} and \eqref{beck}. This is the reason why we assume that $\Omega$ must be simply connected. \sout{This guarantees that a closed form $\omega$ is exact, i.e. there is $f$ such that $df = \omega$.}% \AS{}
%\marginnote{Explain more}

In the present paper, we exploit %The advantage of the equivalence of \eqref{LGP} and \eqref{beck} %is that we may 
%apply tools of the optimal transportation theory to study \eqref{beck}. We 
%exploit 
the equivalence of \eqref{LGP} and \eqref{beck} which permits us to %first 
address %both sufficient and necessary conditions for 
the existence of solutions to \eqref{LGP} when $\Omega$ is a convex domain not necessarily a polygon, but also for some classes of non-convex domains. %We introduce the 
%basic information about the Monge-Kantorovich problem in Section \ref{sec:Prelim}. 
Switching to \eqref{beck} changes the perspective. From this point, we are interested in the construction of an optimal transport map from the data $(\Omega, \partial_\tau g)$. In this process, we pay special attention to open arcs, where $g$ is strictly monotone, Condition (H1). Furthermore, we must ensure that all transport rays lie inside $\Omega$, this is Condition (H2).
%\sout{We must keep in mind that the sum of these arcs is strictly smaller %\marginnote{Not clear} than the support of $|f|$, where $f = \frac{\partial g}{\partial \tau}$ and the derivative is taken with respect to the arc length parameter.}
%\marginnote{Should be introduced before}
The optimal transportation theory also provides a geometric condition on $g$, which is related to the cyclic monotonicity of the optimal transportation plan, Condition (H3). %It is important to emphasize that this condition%\sout{This is what we assume} \sout{It is weaker} 
%is weaker than the admissibility conditions imposed in \cite{rs1} or \cite{rs2}, see Example \ref{Example1} . %In these papers, the domains $\Omega$ were assumed to be polygonal, and only sufficient conditions were present.  This paper considers condition (H3) necessary and sufficient, and $\Omega$ is a general convex domain. Subsequently in the paper, our results are also generalized to non-convex sets.}

%\marginnote{I think we need to explain more what is the difference}

The first of our major 
results are about necessary and
sufficient conditions for 
the existence of solutions 
to \eqref{LGP}, see 
Theorems \ref{main theorem 
convex case} and \ref{th-necessity}. Since the structure of the set where 
$g$ is monotone matters, we 
assume that $g$ is  not only
continuous but is also in 
$BV(\partial\Omega)$. 
%First
We first make a regularity 
assumption on $\Omega$, 
namely, we assume that the 
set of singular points of 
$\partial \Omega$ understood as points without a tangent line to $\Omega$ is negligible with respect to the measure $|f|$ where $f= \partial_\tau g$. This allows us to use the projection map in our approximation. %\dfrac{\partial g}{\partial \tau}$. 
Later in the paper, such a condition is relaxed. %\marginnote{But this assumption is not necessary}
The final assumption on the continuous datum $g$ is that it is piecewise monotone, see Definition \ref{d-ps-mono}, i.e. there is a family of open arcs such that $g$ restricted to each of them is strictly monotone and their complement %of the sum of open arcs, where $g$ is monotone, 
is a null set with respect to the measure $|f|$. In this way, we %ruled out 
excluded from the data continuous functions $g$, which grow on a Cantor set. We prove in Theorem \ref{main theorem 
convex case} that under these assumptions %Then, 
there exists a unique solution. However, as Example \ref{Ex:Cantor} attests, the piecewise monotonicity is not necessary for existence. %\marginnote{In this way, it looks like this assumption is required to get uniqueness}

We also show that the existence of solutions for piecewise monotone $g$, (for which Cantor functions as data is excluded), implies that our geometric conditions are satisfied, see Theorem \ref{th-necessity}. %However, we gain no information on the regularity of $\Omega$. 
In conclusion, at the expense of loss of generality of continuous data, we were able to provide a characterization of solvability of \eqref{LGP} in convex, not necessarily strictly convex regions $\Omega$.

Once we establish our existence results in the convex case, %manage to characterize the solvability of \eqref{LGP} in convex sets, 
we show sufficient conditions for the solvability of \eqref{LGP} in non-convex domains $\Omega$. This is done in multiple stages%In fact this is a multistage process.
In the first attempt, we assume that we can partition $\Omega$ into convex sets $C_i$, on which the set of admissibility conditions (H1-H3) established for
convex $\Omega$ holds. %Of course we will have to define $g$ on $\partial C_i\cap \Omega$ appropriately. 
Then, we can show an existence result for \eqref{LGP}, this is the content of Theorem \ref{Theorem 4.5}. A drawback of this result is that even if $\Omega$ is not convex, then nonetheless the curves $\partial C_i\cap \partial \Omega$ are convex. The main thrust of the proof of existence is to show that the optimal transportation plan does not move any mass across the boundaries of $\partial C_i$, that is no mass is transported between two distinct $C_i$'s. Our most general existence result permits curves with negative curvature.
Note that, in the non-convex case, we provide only a set of sufficient conditions, see Theorems \ref{Theorem 4.7} and \ref{thm: weak A3}.

Let us present our results in a broader context. As we stated earlier, 
the authors of \cite{sternberg} established in 1992 the basic existence for convex regions satisfying additional conditions in case of continuous data $g$. The main point is that the boundary data is attained in the trace sense. Their construction is geometric in nature. The authors of \cite{sternberg} also showed that violation of their sufficient conditions may lead to the non-existence of solutions. In 2014 Mazon et al \cite{mazon} showed the existence of a solution to a relaxed version of \eqref{LGP} for arbitrary domains and arbitrary data $g\in L^1(\partial\Omega)$. Their construction is based on the approximation by the $p$-harmonic variational problems. However, the boundary condition is attained in a very weak sense. The paper \cite{mazon} stirred renewed research on \eqref{LGP}. For example, there the authors were interested in the characterization of the trace space, i.e. the subset of $L^1(\partial\Omega)$ consisting of traces of solutions to (\ref{LGP}). It turns out that this set is strictly smaller than $L^1(\partial\Omega)$, see \cite{tamasan} and the discussion in \cite[Chapter 5]{goma}. It is interesting to note that the sufficient conditions of \cite{sternberg} were relaxed in various ways. For example, the case of unbounded $\Omega$ was studied in \cite{WG1}. Another line of research was related to weakening the convexity assumptions. The authors of \cite{grs} addressed \eqref{LGP} on a rectangular region, but Dirichlet data are specified only on a non-trivial subset. In \cite{rs1}, \eqref{LGP} in polygonal regions for continuous data was studied. The data $(\Omega, g)$ are assumed to satisfy a set of admissible assumptions which are stronger than the current set (H1-H3). This result was extended in \cite{rs2}, where % polygonal domains 
discontinuous data were studied. %The set of admissible conditions is even more complicated than in the case of continuous data, see \cite{rs1}. 
A different approach was presented in \cite{DweGor}, where the authors studied \eqref{LGP} in an annulus %The authors of this paper 
using the equivalence of \eqref{LGP} and \eqref{beck} and the tools of the optimal transportation theory. We also mention that the least gradient problem was studied in its anisotropic form or non-homogeneous. We refer the reader interested in the state of the art to a recent book by G\'orny and Maz\'on, see \cite{goma}.

Here is the organization of the paper. Section \ref{sec:Prelim} is devoted to introducing the Beckmann problem and the tools of the optimal transportation for the associated Monge-Kantorovich problem. In Section \ref{Section Convex case}, we show that the condition (H1-H3) permits to construct an optimal transportation plan for the Monge-Kantorovich problem associated with (\ref{beck}). In particular, we show that conditions H1-H3 are necessary and sufficient for the existence of solutions to (\ref{LGP}). Section \ref{sec:nonconvex} is devoted to establishing sufficient conditions for the existence of solutions in two types of non-convex regions $\Omega$.

%Here, we restrict our attention to the two dimensional case, i.e.,  $\Omega\in \bR^2$.
%
%We present a sufficient and necessary conditions for existence of solutions of the in a convex planar region 
%
%wher is a connected open bounded and convex set and the data $g$ are in $C(\partial \Omega)\cap BV(\partial \Omega)$. Later we extend them to the case of contractible domains. In the next sections, we lift the continuity and the convexity assumptions. We prove the following theorems.

%State the problem and equivalence with the Beckman problem, then state the following:
%\begin{itemize}
%\item Conditions are necessary and sufficient conditions for existence of solutions when $\Omega$ is convex and $g\in C(\partial \Omega)\cap BV(\partial \Omega)$. In this case, the solution is unique
%\item If $\partial \Omega$ is contractible then $g$ satisfies conditions $(C1)$, $(C2)$, and $(C3)$, then a unique solution exists.
%\item We remove the continuity conditions and state condition for which a solution exists when the boundary data is only $BV$
%\end{itemize}

\vspace{0.1 cm}
\paragraph{\bf Notation}
We consider $\Omega$ having Lispchitz continuous boundary with the natural, i.e. positive orientation of its boundary. If  $\beta \subsetneq \partial\Omega$ is an arc, then we may compare its points: for $x_1, x_2\in\beta$ means that $x_1$ precedes $x_2$ in the natural orientation and we write $x_1< x_2$. We will write throughout the paper about monotonicity of $g$ defined on an arc $\beta \subsetneq \partial\Omega$ in the present context, we mean monotonicity with respect to the natural orientation of  $\partial\Omega$. When we pick points $x_1, x_2\in \partial\Omega$ to form an arc $\arc{x_1\ x_2}$ we take care to state which of the two possibilities we have in mind unless such a choice does not matter. By default, we mean open arcs unless stated otherwise. %When we write that an arc $\beta$ is open, we mean it's relatively open, i.e. $\beta = \partial\Omega \cap U$, where $U\subset \bR^2$ is open.

For points $x,y\in \bR^2$ we write $]x,y[$ to denote the open line segment with endpoints $x$ and $y$, by $[x,y]$ we will denote the closed segment. %If $a,b$ are real numbers, then $(a,b)$ is an open interval and $[a,b]$ is a closed interval.

We denote by $\mathcal M(\overline{\Omega}, \mathbb R^2)$ the set of two dimensional Borel vector measures on $\overline \Omega$ and $\mathcal M^+(\overline{\Omega}\times \overline{\Omega})$ the set of nonnegative Borel measures on $\overline{\Omega}\times \overline{\Omega}$. 

$\Pi_x$ and $\Pi_y$ are the projection maps from $\overline{\Omega}\times \overline{\Omega}$ into the first and second component, respectively. For a measurable map $S: X \mapsto Y$ and a nonnegative Borel measure $\mu$, we write the pushforward measure $S_{\#}\mu(A)=\mu(S^{-1}(A))$ for every Borel set $A\subseteq Y$.

%We write  $u\in BV(\Omega)$, bounded variation, if $u\in L^1(\Omega)$ and there exists a Radon measure, denoted by $Du=(D_1u,D_2u)$, such that for every $\varphi\in C_c^{1}(\overline \Omega)$ and $i=1,2$
%$$\int_{\Omega} u\varphi_{x_i}$$

$\Lip_1(\overline{\Omega})$ denotes the space of Lipschitz continuous function on $\overline{\Omega}$ with Lipshitiz constant equal to 1.

For a measure $\mu$, $\mbox{spt}(\mu)$ will
denote the support of $\mu$.

%Finally, we will write $\bN^\star$ for $\bN\setminus \{0\}.$

%if it is not stated otherwise, we mean we choose the arc where $x_1$ precedes $x_2$ in the natural order of $\partial\Omega$ related to the orientation of $\partial\Omega$. 

%In order to make our presentation clear we introduce the following shorthand,
%For any $A, B\subset \bR^2$ we set
%$$
%d(A, B) = \inf\{(x,y) \in A \times B: |x - y|\},\qquad
%d_M(A, B) = \sup\{(x,y) \in A \times B: |x - y|\}.
%$$
%when we writea monotone function 
%when $\Omega\subset \bR^2$ has Lipschitz continuous boundary. In 

\section{Preliminaries on the Beckmann problem and the Optimal Transportation Theory}\label{sec:Prelim}
%{\color{red} Ahmad Add simply connected and replace contractibility and refer why we need it to the paper with Wojtek. Done in the introduction.}
%We present here the B a comment replace contractibeckmale with nn problem and in a separate subsection properties of the transport rays.
%\subsecSition{The Beckmann problem}
The present paper benefits from the method developed in \cite{DweGor}, where the equivalence between the least gradient problem and the Beckmann problem %a problem from the field of optimal transport 
was shown to hold in simply connected regions.
%To be specific we mean the Beckmann problem.
%For this reason we  present here the necessary setting. To be more 
%Our main observation here is the equivalence between a well-known problem which can be treated by the Optimal Transportation theory and the Least Gradient Problem. 
%It turns out that the data of the Beckmann problem are the derivative of the data of the LGP. 
Before entering into the details, we introduce some definitions. First of all, we assume that $\Omega$ is an open, bounded, simply connected and Lipschitz domain of $\R^2$. Let $g$ be a given continuous function on $\partial\Omega$. Since the boundary of our domain $\Omega$ may not be smooth, then we need to define the tangential derivative of $g$.
%In the sequel,
%explain what if the derivative $ g_{\partial \tau}$.
For this purpose, 
we will denote by $\alpha$ any arc length parametrization of $\partial\Omega$ with positive orientation (i.e.,
$\alpha:[0,L)\mapsto \partial \Omega$ with $|\alpha^\prime|=1$ {a.e.}).
%If $\partial\Omega$ is Lipschitz continuous, so is $\alpha$ 
\begin{definition}
%Assume that  is an arc-length parametrization of the boundary of $\Omega$ with positive orientation.
Let $h:\partial\Omega\mapsto \mathbb{R}$ be a Lipschitz function.
%and $x_0 = \alpha(s_0)$, we set
For almost every $s_0 \in [0,L)$, we set 
\begin{equation*}
\partial_\tau h(\alpha(s_0)) = \frac{\mathrm{d}}{\mathrm{d}s} [h(\alpha(s))]_{|s=s_0}.
\end{equation*}
%whenever the right-hand-side (RHS for short) exists.
\end{definition}
%Let us notice that if $h\in C(\partial\Omega)$ satisfies the Lipschitz condition so does the composition $h\circ \alpha$. 
In particular, $\partial_\tau h$ is well-defined $\cH^1$-a.e. on $\partial\Omega$, provided that $h$ is Lipschitz. Now, we will extend this definition to less regular functions on $\partial\Omega$. More precisely, %functional setting,
we define the distributional tangential derivative in an obvious way as a functional over Lipschitz continuous functions. 
%We use it below.

\begin{definition}\label{df-dist}
We say that $g: \partial\Omega\mapsto \mathbb{R}$ belongs to
$BV(\partial\Omega)$ provided that $g \in L^1(\partial\Omega)$ and the distributional tangential derivative of $g$ is a measure with a finite total variation, i.e. there exists a measure denoted by \,$\partial_\tau g$ with finite total variation ($|\partial_\tau g|(\partial\Omega)<\infty$) such that for all functions $h \in \mbox{Lip}(\partial\Omega)$, we have
\begin{equation}\label{r1}
 \int_{\partial \Omega} h \, \mathrm{d}[\partial_\tau g]=-\int_{\partial\Omega} g \cdot \partial_\tau h \, \mathrm{d}\mathcal{H}^1.     
\end{equation}
%Moreover, the total variation $|\partial_\tau g|(\partial\Omega)$ of $\partial_\tau g$ is finite.
\end{definition}

From now on, we always assume that \,$g \in BV(\partial\Omega)\cap C(\partial \Omega)$. So, we may introduce {more carefully }the so-called Beckmann problem, see (\ref{beck}):
\begin{equation} 
\label{Beckmann}\min\bigg\{\int_{\overline{\Omega}} |v|\,:\,v \in \mathcal{M}(\overline{\Omega}, \mathbb{R}^2),\,\nabla \cdot v= f
\ 
 %0\,\,\,\,
\,\mbox{in}\,\ 
 %\,\,\,v \cdot n =f \,\,\mbox{on}\,\,\,\partial
\overline{\Omega}\bigg\}.
\end{equation}
Here, $f:=\partial_\tau g$ is understood in accordance with Definition \ref{df-dist} and the divergence is taken in the distributional sense in $\bR^2$, i.e. $-\int_{\overline{\Omega}} \nabla \varphi \cdot \mathrm{d}v=\int_{\partial\Omega} \varphi\,\mathrm{d}f$, for all $\varphi \in C^1(\overline{\Omega})$. We note that the definition of $f$ implies that this  measure on $\partial\Omega$ is such that 
\begin{equation}\label{zeromass}
f(\partial\Omega)=0.
\end{equation} 

The equivalence we announced at the beginning of this section reads as follows, see \cite[Theorem 3.4]{DweGor} for the proof.

\begin{proposition}\label{t-equiv} %\marginnote{\PR{SOURCE!}}
Suppose that $\Omega\subset \bR^2$ is simply connected and with Lipschitz boundary. If $g\in  C(\partial\Omega)\cap BV(\partial\Omega)$ and $f = \partial_\tau g,$ then Problems (\ref{LGP}) and (\ref{Beckmann}) are equivalent
in the following sense:\\
(1) The values of the infima are equal,  i.e. (\ref{LGP}) = (\ref{Beckmann}).\\
(2) Given a solution $u \in BV (\Omega)$ of (\ref{LGP}), we can construct $v \in \cM(\Omega, \bR^2 )$ a solution of (\ref{Beckmann}) by %Moreover, 
$v = R_{\frac\pi 2} Du$, where $R_{\frac\pi2}$ is the rotation of the plane by \,$\frac\pi2$.\\
(3) Given a solution  $v \in \cM(\overline{\Omega}, \bR^2 )$ of (\ref{Beckmann}) with $|v|(\partial\Omega) =0$, we can construct  $u \in BV (\Omega)$ a solution of (\ref{LGP}) with $v = R_{\frac\pi 2} Du$.
\end{proposition}

%Then, by \cite{DweGor}*{Theorem 3.4} and thanks to the fact that the domain $\Omega$ is assumed to be contractible, we know that Problem \eqref{LGP} is completely equivalent to the following minimal field formulation, which is also To be more precise, we have $\inf\eqref{LGP}=\inf\eqref{Beckmann}$. In addition, if $u$ is a solution for Problem \eqref{LGP} then $v := R_{\frac{\pi}{2}} Du$ (where $R_{\frac{\pi}{2}}$ denotes the rotation by $\frac\pi2$) turns out to be a solution for Problem \eqref{Beckmann}. On the other hand, if \,$v$\, is a solution of Problem \eqref{Beckmann} with $|v|(\partial\Omega) = 0$, then there is a function $u$ such that $v = R_{\frac{\pi}{2}} Du$\, and $u$ is a solution of Problem \eqref{LGP}. 

Hence, in order to prove the existence of a solution to Problem \eqref{LGP}, we just need to show the existence of a solution $v$ to the Beckmann problem \eqref{Beckmann}, which gives zero mass to the boundary. 

Let us comment on the solvability of (\ref{Beckmann}). It is well known, see e.g. \cite[Chapter 4]{Santambrogio},
that for convex domains $\Omega$ the Beckmann problem \eqref{Beckmann} is equivalent to the following Monge-Kantorovich problem: 
\begin{equation}\label{Kantorovich}\min\bigg\{\int_{\overline{\Omega} \times \overline{\Omega}} |x-y|\,\mathrm{d}\gamma\,:\,\gamma \in  \mathcal{M}^+(\overline{\Omega} \times \overline{\Omega}),\,(\Pi_x)_{\#}\gamma=f^{+} \,\,\mbox{and}\,\,(\Pi_y)_{\#} \gamma =f^{-}\bigg\},
\end{equation}
where $f^+$ and $f^-$ are the positive and negative parts of $f$. Moreover,  it is known, see \cite{Santambrogio}, that for any bounded domain $\Omega$, not necessarily convex, the optimal transport problem \eqref{Kantorovich} has  a dual formulation: %which is given by the following:
\begin{equation} \label{dual}
    \sup\bigg\{\int_{\overline{\Omega}} \phi\,\mathrm{d}(f^+ - f^-)\,:\,\phi\in \mbox{Lip}_1(\overline{\Omega})\bigg\}.
\end{equation}
Due to the duality \,$\eqref{Kantorovich}=\eqref{dual}$, we get that if $\gamma$ is an optimal transport plan in Problem \eqref{Kantorovich} and $\phi$ is a Kantorovich potential , i.e. a maximizer for Problem \eqref{dual}, then from  \cite[Chapter 3]{Santambrogio} we have the following equality: %\marginnote{\PR{SOURCE!}}
\begin{equation} \label{transport ray}
\phi(x)- \phi(y)=|x-y|,\,\,\,\,\,\mbox{for all}\,\,\,\,(x,y) \in \mbox{spt}(\gamma).
\end{equation}
Any maximal line segment $[x,y]$ that satisfies the equality \eqref{transport ray} will be called a {\it{transport ray}}. In other words, any optimal transport plan $\gamma$ has to move the mass $f^+$ to the destination $f^-$ along these transportation rays. Due to \eqref{transport ray}, one can show that if $\phi$ is a Kantorovich potential then $\phi$ is differentiable in the interior of any transport ray $]x,y[$ with $\nabla \phi(z)=\frac{x-y}{|x-y|}$, for all $z \in ]x,y[$, see \cite{Santambrogio}. In particular, two different transport rays cannot intersect at an interior point of one of them.
%\newpage 

Coming back to Problem \eqref{Beckmann}, %it is not difficult to
we see that we always have the following inequality $\eqref{dual} \leq \eqref{Beckmann}$ even if the domain $\Omega$ is not convex. Indeed, if $v$ is admissible in Problem \eqref{Beckmann} and $\phi$ is a 1-Lip smooth function on $\overline{\Omega}$, then we must have
$$\int_{\overline{\Omega}} \phi\,\mathrm{d}(f^+ - f^-)=-\int_{\overline{\Omega}} \nabla\phi\cdot\mathrm{d}v \leq \int_{\overline{\Omega}}|v|.$$
Now, let us assume that the domain $\Omega$ is convex. Then, one can show that $\eqref{Beckmann}=\eqref{Kantorovich}$, see for instance, \cite[Chapter 4]{Santambrogio}. 
In fact, from an optimal transport plan $\gamma$ of Problem \eqref{Kantorovich}, one can construct a minimal vector field for Problem \eqref{Beckmann} by considering the vector field $v_\gamma$ which is defined as follows:
\begin{equation}\label{flow}
\lb v_\gamma,\xi\rb :=\int_{\overline{\Omega}\times\overline{\Omega}}\int_0^1 \xi((1-t)x+t y)\cdot (y-x)\,\mathrm{d}t\,\mathrm{d}\gamma(x,y),\,\,\mbox{for all}\,\,\,\,\xi \in C(\overline{\Omega},\mathbb{R}^2).
\end{equation}
Moreover, %it is classical to
we associate a scalar measure $\sigma_\gamma$, called {\it{transport density}}, with  vector measure $v_\gamma$. Measure $\sigma_\gamma$ represents the amount of transport taking place in each region of $\overline{\Omega}$. This measure %\sigma_\gamma$ 
is defined as follows:
\begin{equation}\label{Transport density definition}
\lb \sigma_\gamma,\varphi \rb:=\int_{\overline{\Omega}\times\overline{\Omega}}\int_0^1 \varphi((1-t)x+t y) |x-y|\,\mathrm{d}t\,\mathrm{d}\gamma(x,y),\,\,\mbox{for all}\,\,\varphi \in C(\overline{\Omega}).
\end{equation}
It is clear that %we need 
convexity of the domain $\Omega$ {suffices for}  the transport density $\sigma_\gamma$ (or the vector measure $v_\gamma$) {to be} well defined. In fact, one can see easily that this vector field $v_\gamma$ is admissible in Problem \eqref{Beckmann}. In addition, if $\phi$ is a Kantorovich potential between $f^+$ and $f^-$ then we have
%$\phi$ is differentiable in the interior of any transport ray (see \cite{Santambrogio}) with $\nabla \phi(z)=\frac{x-y}{|x-y|}$, for all $z \in ]x,y[$ and all $(x,y) \in \spt(\gamma)$. In particular, we see that 
$v_\gamma=-\sigma_\gamma \nabla \phi$. In particular, we get that
%and we have 
$\int_{\overline{\Omega}}|v_\gamma|=\int_{\overline{\Omega}\times\overline{\Omega}} |x-y|\,\mathrm{d}\gamma(x,y)=\eqref{Kantorovich}$ and so, $v_\gamma$ minimizes Problem \eqref{Beckmann}. Moreover, one can show that any optimal vector field $v$ of Problem \eqref{Beckmann} is of the form $v=v_\gamma$, for some optimal transport plan $\gamma$, see \cite[Theorem 4.13]{Santambrogio}.
%It is also classical to associate with this vector measure $v_\gamma$, a scalar measure $\sigma_\gamma$ (called {\it{transport density}}), which represents the amount of transport taking place in each region of $\Omega$. This measure $\sigma_\gamma$ is defined as follows:
%\begin{equation}\label{Transport density definition}
%\lb \sigma_\gamma,\varphi \rb:=\int_{\overline{\Omega}\times\overline{\Omega}}\int_0^1 \varphi((1-t)x+t y) |x-y|\,\mathrm{d}t\,\mathrm{d}\gamma(x,y),\,\,\mbox{for all}\,\,\varphi \in C(\overline{\Omega}).
%\end{equation}

Hence, the question of the existence of 
a solution $u$ for Problem \eqref{LGP} %exists 
becomes equivalent to whether 
%{(\ref{zeromass}) is satisfied by 
 %$f:=\sigma_\gamma$ i.e. 
 the transport density $\sigma_\gamma$ gives zero mass to the boundary $\partial\Omega$ or not. 
If $\Omega$ is strictly convex, then
we can easily deduce from (\ref{Transport density definition}) that
$$
\sigma_\gamma(\partial\Omega) = \int_{\partial\Omega\times \partial\Omega}
\cH^1([x,y]\cap \partial\Omega)\, d\gamma(x,y) =0.
$$
Consequently, Problem \eqref{LGP} has a solution $u$ as soon as $\Omega$ is strictly convex. 
%In the case 
%While 
%\marginnote{can we remove the word merely, what does it mean} 
However, when $\Omega$ is only assumed to be convex but not necessarily strictly convex, 
it is not true in general that $\sigma_\gamma(\partial\Omega)=0$. As a result, a solution $u$ for Problem \eqref{LGP} may not exist, for example this happens when $\Omega:=[0,1]^2$, $f^+:=\chi_{[0,\frac{1}{2}] \times \{0\}}\LC\cH^1$ and $f^-:=\chi_{[\frac{1}{2},1] \times \{0\}}\LC\cH^1$. \\

%In the sequel, 
%After recalling the necessary ingredients of the Optimal Transportation Theory we make a general remark on the boundary datum $g$.  Namely, 
%we will write throughout the paper  about monotonicity of
%In order to make our presentation clear we introduce the following shorthand,
%For any $A, B\subset \bR^2$ we set
%$$
%d(A, B) = \inf\{(x,y) \in A \times B: |x - y|\},\qquad
%d_M(A, B) = \sup\{(x,y) \in A \times B: |x - y|\}.
%$$
%when we writea monotone function 
%$g$ defined on an arc $\beta \subsetneq \partial\Omega$ and
%when $\Omega\subset \bR^2$ has Lipschitz continuous boundary. In 
%in the present context, we mean monotonicity with respect to the natural orientation of  $\partial\Omega$.
%when $\beta$ is  a strict subset of $\partial\Omega$.
%\subsection{Transportation rays}

At this point, we state an observation, which we will frequently use in the sequel.

\begin{lemma}\label{l-conv}
Let us suppose that $\{ f^+_n\}_{n \geq 1}$ and \,$\{ f^-_n\}_{n \geq 1}$ are two sequences of data for the Monge-Kan\-to\-ro\-vich problem (\ref{Kantorovich}) and, $\{\gamma_n\}_{n \geq 1}$ (resp. $\{\phi_n\}_{n \geq 1}$)  is a sequence of corresponding optimal transportation plans (resp. Kantorovich potentials such that \,$\phi_n(x_0)=0$\, for a fixed point \,$x_0\in \Omega$). Then, we have the following: \\ 
(1) \,If \,$f^+_n \rightharpoonup f^+$ and $f^-_n \rightharpoonup f^-$, then there exists a subsequence such that $\gamma_{n_k} \rightharpoonup \gamma$ weakly and $\phi_{n_k}\to \phi$ uniformly. Moreover, $\gamma$ is an optimal transportation plan for $f^+$, $f^-$ and $\phi$ is the corresponding Kantorovich potential.\\ 
(2) \,If \,$[x_n,y_n]$ is a transportation ray between $f_n^+$ and $f_n^-$ such that \,$x_n\to x$ and $y_n\to y$, then $[x,y]$ is a transportation ray between $f^+$, $f^-$.
\end{lemma}
\begin{proof}
Since measures $\gamma_n$ are uniformly bounded, we can select a subsequence (not relabeled) converging weakly to $\gamma \in \cM^+(\overline{\Omega}\times \overline{\Omega})$. Moreover, the marginals are preserved,
$$
\lb f^+, \varphi\rb  = \lim_{n\to\infty}\lb f^+_n, \varphi\rb = 
\lim_{n\to\infty}\int_{\overline{\Omega}\times\overline{\Omega}} \varphi(x)\, d\gamma_n(x,y) 
= \int_{\overline{\Omega}\times\overline{\Omega}} \varphi(x)\, d\gamma(x,y)
$$
and the same argument is also valid for $f^-$. Namely, we integrate the map $(x,y)\mapsto\varphi(y)$
with respect to $\gamma$. 

%On the other hand, 
Since  $\phi_n(x_0)=0$, then $\phi_n$, up to a subsequence, converges uniformly to a function $\phi \in \Lip_1(\overline{\Omega})$. Due to the duality $\eqref{Kantorovich}=\eqref{dual}$, we have
$$\int_{\overline{\Omega} \times \overline{\Omega}}|x-y|\,\mathrm{d}\gamma_n=\int_{\overline{\Omega}} \phi_n\,\mathrm{d}(f^+_n - f^-_n).$$
Letting $n \to \infty$, we get that $\gamma$ is an optimal transport plan between $f^+$ and $f^-$. Moreover, $\phi$ is the corresponding Kantorovich potential, because
$$
\eqref{Kantorovich} 
\leq \int_{\overline{\Omega} \times \overline{\Omega}} |x-y|\,\mathrm{d}\gamma
=\int_{\overline{\Omega}} \phi\,\mathrm{d}(f^+-f^-) \leq \eqref{dual}.
$$
%\end{proof}
%\begin{proof}
This concludes the proof of Part (1). Part (2) follows immediately from Part (1) and \eqref{transport ray}, because of the following
$$\phi(x) - \phi(y)=\lim_{n\to \infty} \phi_n(x_n) - \phi_n(y_n)=\lim_{n\to \infty} |x_n - y_n|=|x-y|. \qedhere
$$ 
\end{proof}

Finally, we conclude this section by a few remarks about the endpoints of transportation rays. Of course, they all belong to the support of $f = \partial_\tau g$, which is contained in $\partial\Omega$. It may happen that $x\in \partial \Omega$ is an endpoint of more than one transportation ray. We shall call such points the {\it multiple points}.
%Their set $Z$ is negligible. 
So, we have the following:
\begin{lemma}\label{l-neg-N}
Let us suppose that \,$\Omega$ is convex, then the set of multiple points \,$\mathcal{N}$ is at most countable. In particular, this set \,$\mathcal{N}$ is $f^+-$negligible (i.e. $f^+(\mathcal{N})=0$) as soon as $f^+$ is nonatomic. %%\marginnote{what does negligible mean?}
\end{lemma}
% \todo{S.D.: Assume that $\Omega$ is convex. Otherwise, we need to use that $\Omega$ is contractible; since the statement is not true in general.} 
\begin{proof}
Suppose that $x_0$ is a multiple point, then there are two distinct points $x_1$ and $x_2$ in $\partial \Omega$ such that $[x_0, x_1]$ and $[x_0, x_2]$ are transportation rays. In particular, the arc $\arc{x_1 x_2}$ has positive $\cH^1$ measure (here, we assume that $x_0 \notin \,\arc{x_1 x_2}$). Moreover, if $x_0 \neq \tilde{x}_0$ are multiple points then their corresponding arcs $\arc{x_1 x_2}$ and $\arc{\Tilde{x}_1 \Tilde{x}_2}$ must be disjoint. But on $\partial\Omega$, we can fit at most a countable number of such disjoint open arcs corresponding to multiple points. Hence, the set of multiple points \,$\mathcal{N}$ is at most countable. $\qedhere$
\end{proof}

Throughout the paper, we will assume that $f^+$ is nonatomic, in fact, this follows immediately from the continuity of $g$. 
\begin{definition}
Let $\gamma$ be an optimal transport plan in Problem \eqref{Kantorovich}, we associate to $\gamma$ the multivalued map $R_\cN:\spt f^+\mapsto \mathcal P(\partial \Omega)$ as follows
$$R_{\cN}(x):=\{y\in \partial \Omega:[x,y] \hbox{ is a transportation ray}\}.$$   
In virtue of Lemma \ref{l-conv} the graph of $R_{\cN}$ is closed, and so from \cite[Chapter 18]{Aliprantis}, $R_{\cN}$ admits a Borel selector which we will denote by $R$ and call it a {\it  transportation map}. 
\end{definition}   

We make the following observation about $R$.

%\PR{%Once we have an optimal transportation plan $\gamma$ we will define a multivalued transportation map $R_\cN:\partial\Omega\to P(\partial\Omega)$. We set
%$$ 
%R_\mathcal{N}(x) := \{y\in \partial \Omega: [x,y]\hbox{ is a transportation ray}\}.
%$$
%We stress that this definition is valid for the multiple points too making $R_\cN$ set valued. We are mainly interested in the properties of $R_\cN$ restricted to the support of $f^+$ with the singular points removed,
%$$
%R = R_\cN |_{\spt f^+ \setminus \cN}.
%$$
%We collect basic observation of $R$ and $R_\cN$ is  Lemmas below.}

\begin{lemma}\label{l-cR} 
Suppose that %$x_1, x_2\in \cN$ and the
an open arc 
\,$G \subset \spt f^+$ does not intersect \,$\cN.$ Then, $R$ is continuous on $G$ and $R(G)$ is an arc. Moreover, $R(G)\cap \cN =\emptyset.$
%(b) There exists a Borel measurable selection of the multivalued mapping $R_\cN |_{\spt f^+}$.
\end{lemma}
\begin{proof}
Let us suppose that $x_1, x_2\in G$ with $x_1<x_2$, then $R(x_1)> R(x_2)$, otherwise the transportation rays $[x_1, R(x_1)]$ and $[x_2, R(x_2)]$ would intersect, thanks to the convexity of $\Omega$. Since $R$ is monotone, then it is continuous except at a countable set of jump points. Suppose that $x_0\in G$ is a point, where $R$ has a jump. In this case Lemma \ref{l-conv} implies that $x_0$ is a multiple point, contrary to our assumption. Hence, continuity follows and $R(G)$ is an arc.

Suppose now that $y\in R(G)\cap \,\cN,$ then there exist $w_1$, $w_2$ such that $[w_1,y]$, $[w_2,y]$ are transportation rays and $w_1\in  %\arc{x_1x_2}\PR{\subset
G$. This means that $[w_2, y]$ must intersect some transportation rays with endpoints near $y$. Thus, we reached a contradiction, again we use the convexity of $\Omega$. 
%(b) We have to find a selection at $\spt f^+ \cap \cN.$ For $x\in \spt f^+ \cap \cN$ we may set
%$$
%R(x) = y, 
%$$
%where $y$ is the smallest element of $R_\cN(x)$ is the lexicographic ordering of the plane. Since the set $\cN$ is at most countable and $R$ is continuous on $\spt f^+ \setminus \cN$, then we deduce that the selection defined above is Borel.
\end{proof}
We end this section with the following Lemma.
\begin{lemma}\label{l-2.7}
Let us suppose that \,$\Omega$\, is strictly convex, $f^\pm$ are as defined above and \,$\gamma$ is an optimal transportation plan, then $\gamma$ is concentrated on the graph of $R$. Moreover, $\gamma$ is unique. %, in particular $\gamma$ is concentrated on the graph of $R$ and in particular we have $R_{\#}f^+=f^-$.%,\\
%(a)  $R_\#f^+ = f^-$.\\
%(b) $\gamma$ is a unique transportation plan.}
\end{lemma}

\begin{proof}
Thanks to Lemma \ref{l-neg-N}, we see that for $\gamma-$almost every couple $(x,y)\in(\partial\Omega)^2$, the point $x$ does not belong to $\mathcal{N}$. As a result, there is a unique transport ray starting at $x$ and going to the support of $f^-$. 
 %But $\Omega$ is strictly convex, 
Since \,$\Omega \subset \mathbb{R}^2$\, is strictly convex, then this ray must intersect $\spt(f^-)$ at exactly one point $R(x)$. This yields that $\gamma$ is concentrated on the graph of the map $R$. In particular, we have 
\begin{equation}\label{r-Rfp}
R_{\#}f^+=f^-
\end{equation}
because for all $\varphi \in C(\overline{\Omega})$, one has%\marginnote{not clear}
$$\int_{\overline{\Omega}} \varphi(y)\,\mathrm{d}f^-(y)=\int_{\overline{\Omega} \times \overline{\Omega}} \varphi(y)\,\mathrm{d}\gamma(x,y)=\int_{\spt \gamma} \varphi(R(x))\,\mathrm{d}\gamma(x,y)=\int_{\overline{\Omega}} \varphi(R(x))\,\mathrm{d}f^+(x).$$
%\PR{Take a Borel set $A\subset \spt f^-$. Since $f^\pm$ are marginals of $\gamma$, then we have
%$$
%f^-(A) = f^-(A \setminus \cN) =
%\gamma(\partial\Omega\times (A \setminus \cN)) = \gamma\left(\bigcup_{y\in A \setminus \cN} [R^{-1}y, y]\right).
%$$
%The last equality follows from the fact that $\gamma$ is supported on the transportation rays. Moreover, we made sure that $R^{-1}(y)$ is uniquely defined. Thus, we can see
%$$
%\gamma\left(\bigcup_{y\in A \setminus \cN} [R^{-1}y, y]\right)
%= \gamma\left(\bigcup_{x\in R^{-1}(A \setminus \cN)} [x, R(x)]\right)
%= \gamma( R^{-1}(A \setminus \cN)\times\partial\Omega) = f^+(R^{-1}(A \setminus \cN)).
%$$
%Since $R^{-1}(A \setminus \cN) = R^{-1}(A) \setminus R^{-1}(\cN)$ it suffices to show that $f^+(R^{-1}(\cN)) =0.$ We can further reduce this task to showing that $f^+(R^{-1}(y_0)) =0 $ for any $y_0\in \cN$ because $\cN$ is  countable. 

%Indeed, there is a closed arc $\alpha\supset R^{-1}(y_0)$  that is minimal with this property and we may write $\alpha = \arc{x_1 x_2}$. In particular $y_0\in R_\cN(x_i)$, $i=1,2$. Hence,
%\begin{align*}
%f^+(R^{-1}(y_0))& = \gamma( R^{-1}(y_0)\times \partial\Omega)\\& =
%\gamma( R^{-1}(y_0)\times \{y_0\} \cup \{x_1\}\times R_\cN(x_1)  \cup \{x_2\}\times R_\cN(x_2) )  \\
%&\le \gamma(\partial\Omega\times \{y_0\}) +
%\gamma(\{x_1\}\times \partial\Omega) +
%\gamma(\{x_1\}\times \partial\Omega)= 0,
%\end{align*}
%where we used that both $f^+$ and $f^-$ have no atoms. This shows that $R_\# f^+ = f^-.$}

%We next show 
At the same time, we infer 
that the optimal transport plan $\gamma$ is unique, because it is uniquely determined by the map $R$ which depends only on the Kantorovich potential $\phi$ and it is a Borel selector of $R_\mathcal{N}$ which is unique up to the negligible set of multiple points,  $\mathcal{N}$. $\qedhere$
%Let us %By contradiction, 
%assume this is not the case. Then, let $\gamma_1$ and $\gamma_2$ be two different optimal transport plans, $R_1$ and $R_2$ the corresponding optimal transport maps, that is $\gamma_1=(Id,R_1)_{\#}f^+$ and $\gamma_2=(Id,R_2)_{\#}f^+$. Yet, it is easy to check that $\gamma:=\frac{\gamma_1+\gamma_2}{2}$ is also a transport plan. Moreover, it is optimal because one has
%$$\int_{\overline{\Omega}\times\overline{\Omega}}|x-y|\,\mathrm{d}\gamma=\frac{1}{2}\int_{\overline{\Omega}\times\overline{\Omega}}|x-y|\,\mathrm{d}\gamma_1+\frac{1}{2}\int_{\overline{\Omega}\times\overline{\Omega}}|x-y|\,\mathrm{d}\gamma_2=\min\eqref{Kantorovich}.$$\\
%However, we have already proved that any optimal transport plan is concentrated on the graph of a function $R$. But, that is not possible for the optimal transport plan $\gamma$, because $R_1 \neq R_2$.
\end{proof}

\section{Necessary and sufficient conditions for existence and uniqueness in the convex case}
\label{Section Convex case}
%\section{The sufficient and necessary conditions for existence in case of bounded, open and convex $\Omega$}

In this section, we assume that $\Omega$ is an open bounded and convex (but not necessarily strictly convex) domain in $\R^2$ and the trace $g$ is in $C(\partial \Omega)\cap BV(\partial \Omega)$. Under these conditions, we aim to find necessary and sufficient conditions for the existence of a solution $u$ to the least gradient problem \eqref{LGP}. We also show that this solution $u$ is in fact unique.

\subsection{{Admissibility conditions}}
First of all, we introduce our admissibility conditions (H1), (H2), and (H3) on the boundary datum $g$. Let us start with the first condition: \\

 $\bullet$ {\bf{Condition (H1).}} 
%The main idea is to find a decomposition of $\partial \Omega$ depending on the monotonicity and total variation of $g$ such that the transport ray do not intersect. More precisely, 
%Let us
Suppose that there are three (possibly infinite) index sets \,$I_\Gamma,\,I_\chi,\,I_F\subseteq \mathbb N$ such that the boundary $\partial \Omega$ can be decomposed, up to a $|f|-$negligible set,  into disjoint open arcs $\Gamma_i^{\pm}=\arc{a_i^{\pm}b_i^{\pm}}$ ($i\in I_\Gamma$), \,$\chi_i^{\pm}=\arc{c_i c_i^{\pm}}$ ($i\in I_\chi$) and 
%\marginnote{\PR{the use of $x^\pm_i$\\ is a bit\\ misleading}}
$F_i$ ($i\in I_F$) such that: \\
\begin{enumerate}
\item For every $i\in I_\Gamma$, we have \,$\dist(\Gamma_i^+,\Gamma_i^-){:= \inf\{|x-y|:\ x\in \Gamma_i^+,\ y\in \Gamma_i^-\}}>0$, $g$ is strictly increasing on $\Gamma_i^+$ and strictly decreasing on $\Gamma_i^-$  with $$TV(g|_{\Gamma_i^+}) =  TV(g|_{\Gamma_i^-}).$$
For the sake of convenience, we assume that $g(a_i^+)=g(a_i^-)<g(b_i^+)=g(b_i^-)$.\\
%and in particular the positive distance between $\Gamma_i^+$ and $\Gamma_i^-$ implies that $[a_i^+,a_i^-]\cap [b_i^+,b_i^-]=\emptyset,$
%\todo{S.D.: These two statements are equivalent.}
\item For every $i\in I_\chi$, we have {$\overline{\chi_i^+}\cap\overline{\chi_i^-} =\{c_i\}$},
%$\dist(\chi_i^+,\chi_i^-)=0$,
$g$ is strictly increasing on $\chi_i^+$ and strictly decreasing on $\chi_i^-$ with
$$TV(g|_{\chi_i^+}) =  TV(g|_{\chi_i^-}).$$
\item For every $i\in I_F$, %\PR{set $F_i$ is connected, $\overline{\hbox{int}\,F_i}=F_i$ and}
the boundary datum $g$ is constant  on $F_i$. Moreover, each $F_i$ is maximal with this property, i.e. if an open arc $\alpha\supset F_i$ is such that $g|_\alpha$ is constant, then $\alpha = F_i$. We shall say that $F_i$ is a flat part.\\
%\item For all $i\in I_\Gamma$ and $j\in I_\chi$, we have $\Gamma^\pm_i \cap \chi^\pm_j = \emptyset.$
\item  For every $i\in I_\Gamma \cup I_{\chi}$, we denote by \,$T_i$\, the convex hull of \,$\Gamma_i^+$ and $\Gamma_i^-$ and $D_i$ the convex hull of $\chi_i^+$ and $\chi_i^-$. Then, we assume that the sets $\{T_i,\,D_i:i\in I_\Gamma \cup I_\chi\}$ are mutually disjoint.\\
%\marginnote{Not needed}
%\item The set $N = \partial \Omega\setminus(\bigcup_{i\in I_\Gamma} (\Gamma_i^+\cup\Gamma_i^-)\cup \bigcup_{i\in I_\chi} (\chi^+_i\cup \chi_i^-)\cup \bigcup_{i\in I_F}F_i)$ is negligible in the sense that
%$$
%\cH^1(N) +
%\int_N \left| \partial_\tau g\right| =0.
%$$
%\marginpar{ here not \\ the same i\\ no? }
%for indices $i\neq j$ from $I_\Gamma$ (resp. $I_\chi$), 
%$$T_i\cap T_j=\emptyset,\qquad (\hbox{resp. }D_i\cap D_j=\emptyset)$$
%and for any $i\in I_\Gamma$, $j\in I_\chi$
%$$
%T_i\cap D_j=\emptyset.$$
\end{enumerate}
%\begin{remark}
 %   This set $N$ in (5) may contain not only the endpoints of the arcs $\chi_i^\pm$, $\Gamma_i^\pm$ and $F_i$ but also the Cantor set in the distributional derivative of $g$. %In other words, we may assume instead of (5) that $g \in W^{1,1}(\partial\Omega)$ since we already know that the jump set is empty as $g \in C(\partial\Omega)$.
%\end{remark}

%After specifying the arcs of the decomposition of $\partial\Omega$ we may state condition (H1):
%\begin{definition}[Condition (H1)]
%We shall say that the boundary $\partial\Omega$ of an open bounded, connected convex set satisfies condition (H1) provided that there exist  open arcs $\Gamma_i^{\pm}$, $i\in I_\Gamma,$ $\chi_i^{\pm},$ $i\in I_\chi$, and $F_i$, $i\in I_F$, with the properties stated above and such that
%the set 
%$$
%\partial \Omega \setminus \left( 
%\bigcup_{i\in I_\Gamma}(\Gamma_i^+\cup \Gamma_i^-)\cup
%\bigcup_{i\in I_\chi}(\chi_i^+\cup \chi_i^-)\cup
%\bigcup_{i\in I_F} F_i
%\right)
%$$
%is at most countable.
%\end{definition}
%\paragraph{\bf Notation} 
{Keeping (H1)} %With the above condition 
in mind, we introduce the following notation:
$\Gamma_i=\Gamma_i^+\cup \Gamma_i^-$ ($i \in I_\Gamma$), $\chi_i=\chi_i^+\cup \{c_i\} \cup\chi_i^-$ ($i \in I_\chi$), $\Gamma^{\pm}=\bigcup_{i\in I_\Gamma} \Gamma_i^{\pm}$, $\chi^{\pm}=\bigcup_{i\in I_\chi}\chi_i^{\pm}$, $\Gamma=\Gamma^+ \cup \Gamma^-$ and $\chi=\chi^+ \cup \chi^-$. \\

%\PR{$\chi_i$ IS AN ARC, BECAUSE WE NEED IT, WHILE $\chi$ ISN'T THE SUM OF $\chi_i$'S. POSSIBLE TROUBLE AHEAD. }\\
In order to introduce the other assumptions (H2) and (H3) on the boundary datum $g$, we need to define a transport map
${\bf T}:\Gamma^+ \cup \chi^+\mapsto \Gamma^- \cup \chi^-$  
%\PR{HERE, $\bbT$ ISN'T DEFINED AT $c_i$ AND THIS IS OUR INTENTION.}
which will be a good candidate to be the optimal transport map between $f^+$ and $f^-$. So, we proceed as follows:\\
\begin{equation}\label{df-T}{\bf{T}}(x^+)=\begin{cases}
x^-\in \chi_i^- & \mbox{if}\,\,\,\,x^+ \in \chi_i^+\,\,\,\,\,\mbox{and}\,\,\,\,\,TV(g|_{\arc{\,\,c_i\,x^+}})=TV(g|_{\arc{\,\,c_i\,x^-}}),\,\,\,\,\,i \in I_\chi,\\
x^-\in \Gamma_i^- & \mbox{if}\,\,\,\,x^+ \in \Gamma_i^+\,\,\,\,\,\mbox{and}\,\,\,\,\,TV(g|_{\arc{\,\,a_i^+\,x^+}})=TV(g|_{\arc{\,\,a_i^-\,x^-}}),\,\,\,\,i \in I_\Gamma.
\end{cases}
\end{equation}
%\PR{ACCORDING TO THIS DEFINITION $\bbT$ IS NOT SET FOR $x = c_i$}
Here, we take the arcs $\arc{\,c_i\,x^\pm} \subset \chi_i^\pm$\, and $\arc{\,a_i^\pm\,x^\pm} \subset \Gamma_i^\pm$. Due to the fact that $g$ is strictly increasing on $\chi_i^+$ (resp. $\Gamma_i^+$) and strictly decreasing on $\chi_i^-$ (resp. $\Gamma_i^-$) and \,$TV(g|_{\chi_i^+})=TV(g|_{\chi_i^-})$ (resp. $TV(g|_{\Gamma_i^+})=TV(g|_{\Gamma_i^-})$), one can see that the map ${\bf T}$ is well defined and it is also one-to-one (see Lemma \ref{cont-T} below). {Hence, the inverse map of \,$\bf{T}$ is well defined and} we will denote it by ${\bf{T}}^{[-1]}$. On the other hand, it is clear that by construction, ${\bf{T}}$ is a transport map from $f^+$ to $f^-$, i.e. ${\bf{T}}_{\#}f^+=f^-$.
%We make the following observation.
\begin{lemma}\label{cont-T}
For every \,$i\in I_\chi$ (resp. $i\in I_\Gamma$), the restriction of the {transportation} map ${\bf T}$ to each arc $\chi^+_i$ {(resp.  \,$\Gamma_i^+$)} is a homeomorphism onto $\chi_i^-$ {(resp. $\Gamma_i^-$)}.
\end{lemma}
\begin{proof}
The restriction of ${\bf T}$ to $\chi^+_i$ is strictly monotone. In fact, if $x_1,\,x_2\in \chi^+_i$ and $x_1<x_2$ with respect to the positive orientation of $\partial \Omega$, then $\bbT(x_1) > \bbT(x_2)$ (if $\bbT(x_1)= \bbT(x_2)$ then this would mean that $g$ is constant on $\arc{{x_1}{x_2}}$, which is impossible). Hence, $\bbT$ is continuous except for at most countably many points.
%and any discontinuity is a jump. 
However, if $\bbT$ has a jump at $x_0 \in \chi^+_i$, then $g$ would be constant on the arc $\arc{ \bbT(x_0^+)\bbT(x_0^-)} \subset \chi_i^-$, where $\bbT(x_0^+):=\lim_{x \to x_0,\,x>x_0} \bbT(x)$ and $\bbT(x_0^-):=\lim_{x \to x_0,\,x<x_0} \bbT(x)$. But, this is impossible due to the strict monotonicity of $g$ on $\chi_i^-$. 

%We also claim that $\bbT$ is strictly monotone. Indeed, if $x_1,\,x_2\in \chi^+_i$, $x_1<x_2$ and $\bbT(x_1)= \bbT(x_2)$, then this would mean that $g$ is constant on $\arc{{x_1}{x_2}}$ that is impossible.

Since $\bbT$ is continuous and strictly monotone, then we deduce that the inverse is also continuous and strictly monotone as desired. 

The same argument applies to $\Gamma_i^+$. $\qedhere$
%{Due to the definitions of $\bbT$ and $\chi_i^\pm$, $\Gamma_i^\pm$ we see that $\bbT(a^+_i) = a^-_i$ and $\bbT(b^+_i) = b^-_i$. Hence, continuity of $\bbT|_{\Gamma^+_i}$ implies that $\bbT(\Gamma^+_i) =\Gamma^-_i$. The same argument yields $\bbT(\chi^+_i)= \chi^-_i$}.
\end{proof}

Now, we continue stating our assumptions. Since the domain \,$\Omega$\, is not necessarily strictly convex, we want to prohibit any transport ray from gliding along the boundary $\partial \Omega$. So, we impose the following condition:\\

$\bullet$ {\bf{Condition (H2).}} For every $x^+ \in \Gamma^+ \cup \chi^+$, the open segment $]x^+,{\bf{T}}(x^+)[$ is contained in %the open set 
$\Omega$. \\

Finally, since we need %the transportation map 
${\bf T}$ to be %indeed the
an optimal transport map between $f^+$ and $f^-$, where $f=\partial_\tau g$, so we make the following assumption that will be crucial in the course of proving this fact:\\
%More precisely, we assume the following:\\

$\bullet$ {\bf{Condition (H3).}} Let us consider any finite sequence of points $\{e_{i}^+\}_{1 \leq i \leq m}$ (where $m \in \mathbb{N}$) in $\Gamma^+\cup \chi^+$, %such that $e_{k}^+ \in \chi_{i_k}^+$\, or \,$ e_{k}^+ \in \Gamma_{i_k}^+$ for every $k=1,\cdots, m$.
then we assume the
%\begin{definition}\label{dH3}
%We shall say that the couple $(\Omega, g)$ satisfies  condition (H3) if for any $m\in\bN$ and  
%for any finite sequence of points $\{e_{k}^+\}_{1 \leq k \leq m}$ as above the 
following inequality:
\begin{equation}\label{rH3}
\sum_{i=1}^m |e_i^+ - {\bf{T}}(e_i^+)|
<
\sum_{i=1}^{m-1} 
|e_i^+ - {\bf{T}}(e_{i+1}^+)| + |e_m^+ - {\bf{T}}(e_{1}^+)|.
\end{equation}
%\end{definition}
We note that this condition is related to the cyclical monotonicity property which is satisfied by any optimal transport plan $\gamma$ of Problem \eqref{Kantorovich} (here, $\gamma=(Id,{\bf{T}})_{\#}f^+$), see \cite[Theorem 1.38]{Santambrogio}. We also remark that this condition is weaker than \cite[Condition (C2)]{rs1} and so, this allows us to extend the class of boundary data for which we can deduce the existence of solutions to \eqref{LGP}, see Example \ref{Example1} below.

Now, we are ready to begin our work on the existence of a solution to Problem \eqref{LGP}. For this purpose, we divide our task into several parts. First, we assume that $I_{\Gamma}$ and $I_{\chi}$ are finite %and that the closed arcs $\overline{\Gamma_i}$ $i\in I_{\Gamma}$ and $\overline{\chi_j}$, $j\in I_{\chi}$ are disjoint.
 and we solve our problem in the strictly convex case and so, by means of approximation of the domain $\Omega$, we address the case of convex domains that are not necessarily strictly convex. 
%The assumption that the arcs have disjoint closures, allows the extension of the map ${\bf T}$ defined in \eqref{df-T} to the vertices of $\Gamma_i^-$ and $\chi_{i}^-$ by mapping them to the corresponding vertices on $\Gamma_i^+$ and $\chi_i^+$. 
Finally, by approximating the boundary data, we complete the analysis considering the case when the number of arcs is infinite.

%But for this aim, 
%we divide our task into two parts. The first one is devoted to the case of strictly convex domains $\Omega$, and then by means of approximation, we address in the second part the case when the domain $\Omega$ is convex but necessarily strictly convex.

Before starting, we note that existence of a unique  solution
%we already know that
to Problem \eqref{LGP}  for strictly convex domains $\Omega$ and continuous data $g$ is well-know, see \cite{sternberg}. However, our idea here is to characterize this solution in order to be able to pass to the limit in the convex case. 
%\PR{THIS TEXT DOES NOT REFLECT WELL WHAT HAPPENS NEXT.}

\subsection{The case of strictly  convex domain}

In this section, we will show that thanks to 
%explain the necessary and sufficient 
conditions (H1) and (H3), ((H2) is trivial here), we can construct a 
%for existence of 
solution to Problem (\ref{LGP}). %We would like to have a tool for construction, which is manageable and easy to express in terms of $(\Omega, g)$. 
Since we know that the transport rays are the boundaries of the level sets of the solutions to Problem (\ref{LGP}), we would like to identify these rays or equivalently, to characterize the optimal transport map $R$. More precisely, our bet is that if $x\in \Gamma^+ \cup \chi^+$, then 
the segment $[x,{\bf T}(x)]$ is  a transport ray and so, $R(x)={\bf T}(x)$\, for $f^+-$a.e. $x$. Establishing this fact is the main result of this section.

\begin{proposition}\label{strictly_convex_case} 
%\marginnote{SO here \\ $\Omega$ is not \\ S. c?}
Assume that \,$\Omega$ is bounded strictly convex and \,$g\in BV(\partial \Omega)\cap C(\partial\Omega)$ satisfies conditions (H1) $\&$ (H3), where the sets $I_\chi$ $\&$ $I_\Gamma$ are finite. We set $f= \partial_\tau g$. Let \,$\gamma$\, be the optimal transport plan for Problem \eqref{Kantorovich} between $f^+$ and $f^-$.
%If \,$\Omega$ is strictly convex,  
For \,$\gamma-$a.e. $(x^+,x^-) \in \partial\Omega \times \partial\Omega$, we either have \,$x^+ \in \chi_i^+$ and \,$x^-={\bf{T}}(x^+) \in \chi_i^-$ (\,$i\in I_\chi$) or \,$x^+ \in \Gamma_i^+$ and \,$x^-={\bf{T}}(x^+) \in \Gamma_i^-$ (\,$i\in I_\Gamma$). In other words, we have $\gamma=(Id,{\bf{T}})_{\#}f^+$.
\end{proposition}

Before we embark on proving this proposition, we lay out our tools. An important part of our argument is monitoring how the boundary $\partial\Omega$ is divided by a transportation ray. Namely, we notice that any point $e^+$ and its image $ R(e^+)$ defined in Lemma \ref{l-2.7}) separate $\partial\Omega$ into two parts of zero measure $f$. More precisely, we have the following.
\begin{lemma}\label{lzero}
If \,$e^+\in \chi^+_i$, $i\in I_\chi$ (resp. $e^+\in \Gamma^+_i$, $i\in I_\Gamma$), then $f(\arc{\,e^+ R(e^+)}) =0$.
%and $ \bbT^{[-1]}(R(e^+))\in \chi^+_i$, then 
\end{lemma}
\begin{proof}
Let us note that the definition of the arc $\arc{e^+R(e^+)}$ is ambiguous, however, the result holds independently of the chosen orientation.

We consider the transportation ray $[e^+,R(e^+)]$. Its endpoints separate $\partial \Omega$ into two open sets $E_1,\,E_2\subset \partial \Omega.$ %Here and further in the text, `open' means relative to $\partial\Omega$, when we talk about subsets of $\partial\Omega$.
We claim that 
%if there is $C\subset E_1$ with $f^+ (C)>0$, then 
$R(E_i)\subset E_i$, for $i=1,\,2$. Indeed, if there is a point $x \in E_1$ such that $R(x) \in E_2$
%of $R(E_1)$ and $E_2$ had a positive $f^-$ measure,
then, thanks to the convexity of \,$\Omega$, the transportation rays $[e^+,R(e^+)]$ and $[x,R(x)]$ must intersect, but this is impossible. 
%Let us note that the definition of the arc $\arc{e^+R(e^+)}$ is ambiguous, however it  does not affect the result. Let us first consider the transportation ray $[e^+R(e^+)]$. Its end points separate $\partial \Omega$ into two open sets $E_1, E_2\subset \partial \Omega.$ %We claim that 

Since \,$R^{-1}(E_i)\subset E_i$ ($i=1,\,2$) and due to Lemma \ref{l-2.7}  we have $f^- = R_\#f^+$, then  we see $f^-(E_i) = f^+(R^{-1}(E_i)) \leq f^+(E_i)$, $i=1,\,2$. In particular, we get that $f^-(\partial\Omega)=f^-(E_1)+f^-(E_2) \leq f^+(E_1)+f^+(E_2)=f^+(\partial\Omega)$. But, we know that $f^+(\partial\Omega)=f^-(\partial\Omega)$. Hence, $f^+(E_i)=f^-(E_i)$, $i=1,\,2$. $\qedhere$
%,  so $f(E_1)=0$, i.e.  our claim follows.
\end{proof}

Keeping in mind  the setting of Proposition \ref{strictly_convex_case} we also make the following observation.
%We note that in this case, the set of multiple points is at most finite.
\begin{lemma}\label{lem:Finite double point}
Under the conditions of Proposition \ref{strictly_convex_case}, the set of multiple points \,$\mathcal N$
%=\{x\in \partial \Omega: R_{\mathcal N}(x)\text{ is not single valued}\}$
is finite.
\end{lemma}
\begin{proof}
In Lemma \ref{l-neg-N}, we already showed that $\mathcal N$ is at most countable. Let $x_0\in \mathcal N$ and $x_1,x_2\in R_{\cN}(x_0)$. We claim that $x_1$ and $x_2$ do not belong to the closure of the same arc from $\Gamma \cup \chi$. Indeed, if the claim is not true, then $\arc{x_1x_2}\subset \Gamma_i^-$ or $\arc{x_1x_2}\subset \chi_i^-$ and this arc must be contained in $R_\cN(x_0)$. As a result, %and so, we get
$$
0< f^-(\arc{x_1x_2}) = R_\# f^+ (\arc{x_1x_2}) = f^+ (R^{-1}(\arc{x_1x_2})) = f^+(\{x_0\}) =0,
$$
which is a contradiction. 

Let us suppose now that $x_0'\neq x_0$ is another multiple point and $x_1'\neq x_2'$ are in $R_{\mathcal N}(x_0')$. Since transport rays cannot intersect we see that %$x_1,\,x_2,\,x_1',\,x_2'$ must belong to at least three different closures of arcs of \,$ \Gamma^- \cup \chi^-$, that is 
at least one of the $x_1',\,x_2'$ belongs to the closure of an arc not containing neither \,$x_1$\, nor \,$x_2$.
%Now, if $x_0^{\prime\prime}$ is another multiple point and $x_1^{\prime\prime}\neq x_2^{\prime\prime}$ in $R_{\mathcal N}(x_0^{\prime\prime})$, then again at least one of the points $x_1^{\prime\prime},\,x_2^{\prime\prime}$ must belong to the closure of an arc containing none of the points $x_1$, $x_2$, $x_1^\prime$ and $x_2^\prime$.
We can iterate this process to infer that the number of elements in $\mathcal N$ must be finite. $\qedhere$
%In Lemma \ref{l-neg-N}, we already showed that $\mathcal N$ is at most countable. We claim that no arc $\alpha$ is contained in $R_\cN(x_0)$, for all $x_0\in \cN$. Indeed, if there exists an arc $\arc{x_1x_2} =\alpha\subset R_\cN(x_0)$ then 
%$$
%0< f^-(\alpha) = R_\# f^+ (\alpha) = f^+ (R^{-1}(\alpha)) = f^+(\{x_0\}) %=0,
%$$
%but this is impossible. As a result, if $[x_0, x_1]$ and $[x_0, x_2]$ are transportation rays, then the points $x_1$, $x_2$ must belong to closures of different arcs from  $\Gamma \cup \chi.
%Since the number of these arcs is finite, we conclude that the set $\cN$ must be finite too.
%\begin{comment
%If $x_1,x_2\in R_{\mathcal N}(x_0)$  are distinct, then $x_1$ and $x_2$ belong to different arcs of $\chi^-\cup\Gamma^-$.
 \end{proof}

Now, we can state another important tool. 
%\todo{what is the difference with Lemma 4.2?}
\begin{lemma}\label{lyy}
If $e^+\in \chi^+_i$, $i\in I_\chi$ 
(resp. $e^+\in \Gamma^+_i$, $i\in I_\Gamma$)
and $ \bbT^{[-1]}(R(e^+))\in \chi^+_i$ (resp.  $\bbT^{[-1]}(R(e^+))\in \Gamma^+_i$),
then \,$e^+ =  \bbT^{[-1]}(R(e^+)).$
\end{lemma}
\begin{proof}
%Let us set $e^+_2 =  \bbT^{[-1]}(R(e^+))$. 
By the definition of \,${\bf{T}}$, we note that if \,$\bbT^{[-1]}(R(e^+))\in \chi^+_i$ (resp.  $\bbT^{[-1]}(R(e^+))\in \Gamma^+_i$), then we have  $ R(e^+)\in \chi^-_i$ (resp.  $R(e^+)\in \Gamma^-_i$).
%Now, if $e^+\in \chi^+_i$ %$i\in I_\chi$ 
%(resp. $e^+\in \Gamma^+_i$, $i\in I_\Gamma$)
%and $ \bbT^{[-1]}(R(e^+))\in \chi^+_i$
%Due to the definition of $R$ in \S 2.2 the point $R(e^+)$ is 
%uniquely-defined, so is $e^+_2$. 
%Our claim follows, because 
Thanks to Lemma \ref{lzero}, we know that $f(\arc{\,e^+R(e^+)})=0$. Moreover, for every $x^+ \in \chi_i^+$, $i \in I_\chi$ (resp. $x^+\in\Gamma_i^+$, $i \in I_\Gamma$), we show that
\begin{equation}\label{r-pro}
f(\arc{\,x^+\,\bbT}(x^+))
%= f(\arc{x\, b^+_{i}}) + f(E_1) + f(\arc{{\bf{T}}(x)\,b^-_{i}}) = f^+(\arc{x\, b^+_{i}})- f^-(\arc{{\bf{T}}(x)\,b^-_{i}})
%TV g|_{\arc{e^+_k b^+_{i_k}}} + 0 
%- TV g|_{\arc{e^-_k b^-_{i_k}}}
= 0.
\end{equation}
%This follows immediately from the definition of the transportation map ${\bf{T}}$. 
In fact, if  $x^+\in \chi^+_{i}$ then (\ref{r-pro}) immediately follows from the definition of $\bbT$ (one can assume that the arc is so chosen that $c_{i}\in\, \arc{\,x^+{\bf{T}}(x^+)}$). 

If $x^+ \in \Gamma^+_{i}$ then, again due to the definition of $\bbT$, we have
$f^+(\arc{x^+\,b^+_{i} })= f^-(\arc{{\bf{T}}(x^+)\,b^-_{i}})$
and the arcs are chosen so that  $\arc{x^+b^+_{i}}\subset \Gamma^+_{i}$ and ${\bf{T}}\!\arc{(x^+)\,b^-_{i}}=\bbT(\arc{x^+b^+_{i}}) \subset \Gamma_i^-$. We also see that $\partial\Omega \setminus(\arc{x^+b^+_{i}}\cup \arc{{\bf{T}}(x^+)\,b^-_{i}})$ has exactly two connected components $E_1$ and $E_2.$ For the sake of definitness, we assume that $b^+_{i},\,b^-_{i}\in E_1.$ Notice that if there is $j\in I_\Gamma$ ($i\neq j$) and such that $E_1\cap \Gamma^\pm_j \neq \emptyset$, then $\Gamma^\pm_j\subset E_1$. Indeed, if $\Gamma^\pm_j\setminus E_1\neq \emptyset$, then $\Gamma^\pm_j$ must intersect $\Gamma^+_i\cup \Gamma^-_i$, but this is impossible. The same conclusion, i.e. $\chi_j^\pm\subset E_1$ holds when  $E_1\cap \chi^\pm_j \neq \emptyset$ ($j\in I_\chi$). 

Let us suppose now that $\Gamma^+_j\subset E_1$, $j \in I_\Gamma$.
Then, we  also have $\Gamma^-_j\subset E_1$.
%(resp. if there is $j \in I_\chi$ with $\chi^+_j\subset E_1$ so does $\chi^-_j$). 
%This is indeed so because the closures of $\chi^+_j$ and $\chi^-_j$ have a common point while the distance between $E_1$ and $E_2$ is strictly positive and, since $. 
%A similar argument shows that if $\chi_j^-\subset E_1,$ then $\chi_j^+\subset E_1$.
%We also claim that if $\Gamma^+_j\subset E_1$ so does $\Gamma^-_j.$
Indeed, if  $\Gamma^-_j\subset  E_2$ then the geometry would imply that $T_i \cap T_{j}\neq\emptyset$ which contradicts (H1). {By} the same argument, if $\chi^+_j\subset E_1$ ($j \in I_\chi$), then {$\chi^-_j\subset E_1$.} %$T_i \cap D_j \neq \emptyset$). 
The observations we made imply that 
$$
E_1= \bigcup_{\alpha\in A} \chi_\alpha \cup \bigcup_{\beta\in B}\Gamma_\beta  \cup \bigcup_{\gamma\in C} F_\gamma \cup N,
$$
for appropriate sets of indices $A,\,B,\,C$, where $|f|(N) =0$. Hence, we have $f(E_1) =0.$
As a result, from the definition of ${\bf T}$ we get that
\begin{equation*}
f(\arc{x^+{\bf{T}}(x^+)}) = f(\arc{x^+ b^+_{i}}) + f(E_1) + f(\arc{{\bf{T}}(x^+)\,b^-_{i}}) = f^+(\arc{x^+ b^+_{i}})- f^-(\arc{{\bf{T}}(x^+)\,b^-_{i}})
%TV g|_{\arc{e^+_k b^+_{i_k}}} + 0 
%- TV g|_{\arc{e^-_k b^-_{i_k}}}
= 0.
\end{equation*}
\\
This observation completes the proof of %the claim 
\eqref{r-pro}.

Finally,
%This corresponds to 
%$TV( g|_{\arc{c_i e^+}})=
%TV( g|_{\arc{c_i R(e^+)}})$. 
take $x^+=\bbT^{[-1]}(R(e^+))$, then we get that $f(\arc{\bbT^{[-1]}(R(e^+))R(e^+)})=0$. After combining this with
%Recalling that 
$f(\arc{e^+R(e^+)})=0$, we have
%this yields that %\marginnote{why $|\cdot|$\\ below?}
\begin{align*}
0=  f(\partial\Omega) &= f(\arc{e^+\,\bbT^{[-1]}(R(e^+))})+ f(\arc{e^+\,R(e^+)}) + f(\arc{\bbT^{[-1]}(R(e^+))R(e^+)})\\
&=  f(\arc{e^+\,\bbT^{[-1]}(R(e^+))}).
\end{align*}
Since 
%f^+(\arc{e^+\,\bbT^{[-1]}(R(e^+))})=f(\arc{e^+\,\bbT^{[-1]}(R(e^+))})=
%|f(\arc{e^+\,R(e^+)}) - 
%f(\arc{\bbT^{[-1]}(R(e^+))R(e^+)})|= 
%0
$
f(\arc{e^+\,\bbT^{[-1]}(R(e^+))})  = f^+(\arc{e^+\,\bbT^{[-1]}(R(e^+))}), 
$
we reach $e^+ = \bbT^{[-1]}(R(e^+)).$ 
%The proof in the case of $\Gamma_i$ goes in a similar way, because we noticed in the course of proof of Lemma \ref{lzero} that $R(\Gamma^+_i) = \Gamma^-_i.$
\end{proof}

Now, we are ready to carry out the {\it proof of Proposition \ref{strictly_convex_case}.} We are going to show that for $f^+-$a.e. $x\in \chi^+_i$, $i\in I_\chi$ (resp.  $x\in\Gamma_i^+$, $i\in I_\Gamma$), we have $R(x)=\bbT(x)$.
%the other endpoint $x^-$ is determined by $\bbT.$
%\PR{The notion of a typical transportation ray was defined above.}  %with $x^+$ which is not a 
%multiple points. \PR{We recall that for $p\in \mathcal{N}$
%because there 
%the %transportation 
%map $R$} is defined in an artificial way. Moreover,  see Lemma \ref{l-neg-N},  the set of multiple points $\mathcal{N}$ is $f^+-$negligible.
%For $x^+$ specified above 
%Here, we have several possibilities. The first one is that $R(x^+)$ belongs to $\chi^-_i $, $i\in I_\chi$ (resp. $R(x^+)\in\Gamma_i^-$, $i\in I_\Gamma$). This a favorable case because by Lemma \ref{lyy}, we get that $R(x^+)=\bbT (x^+)$ and so, our claim follows. 
Assume that our claim does  not hold. Hence, due to {piecewise continuity of \,$R$\, implied by} Lemma \ref{l-cR}, there is an arc $E_1^+ \subset \chi_{i_1}^+$ 
(resp. $E_1^+\subset \Gamma_{i_1}^+$) 
that is transported outside $\chi_{i_1}^-$ (resp. $\Gamma_{i_1}^-$) to an arc $E_2^-:=R(E_1^+)$ such that $E_2^- \subset \chi_{i_2}^-$ or $E_2^- \subset\Gamma_{i_2}^-$.
%, where \,$i_1 \neq i_2$\, if \,$E_1^+\subset \chi_{i_1}^+$ and $E_2^-\in \chi_{i_2}^-$ (resp. $E_1^+\subset \Gamma_{i_1}^+$ and $E_2^-\subset \Gamma_{i_2}^-$). 
%(resp. $e^-_2 \in \Gamma_{i_2}^-$ and $i_1\neq i_2$ or $e^-_2\in \chi_{i_2}^-$).
Without loss of generality, we may assume that $E_1^+ \subset \chi_{i_1}^+$ and $E_2^- \subset \chi_{i_2}^-$. The argument when $E_2^-\subset\Gamma^-_{i_2}$ is the same and it will be omitted. %\PR{WHY $E_1\subset \chi_1$ and $E_2\subset \Gamma_2$ IS EXCLUDED?} 
Set $E_2^+:={\bf{T}}^{[-1]}(E_2^-)\subset \chi^+_{i_2}$.  
%\sout{Since by Lemma \ref{l-cR} the arc $E^-_2$  does not intersect $\cN$} 
{Due to Lemma \ref{lem:Finite double point}, we may make the arcs $E^+_2$ and $E^-_2$ disjoint from $\cN$.} 
%Then, we see that by the definition of $\bbT$ the arc $E^+_2$ also does not contain any point of $\cN$.
We claim that $R(E_2^+) \cap  \chi_{i_2}^-=\emptyset$. %(resp. $R(e_2^+) \notin \Gamma_{i_2}^-$). 
Let us suppose the contrary, i.e. 
%again due to Lemma 
%\ref{l-cR}, 
there is an arc $E^+$ of $E_2^+$ such that 
$R(E^+) \subset \chi_{i_2}^-$. %(resp. $R(e_2^+) \in \Gamma_{i_2}^-$), 
Then, Lemma \ref{lyy}
implies that  $R(E^+)=\bbT(E^+) \subset E_2^-$.
%, meaning  that $e_2^-$ is  a multiple point. \PR{
However, this is impossible due to the strict convexity of \,$\Omega$\, and the fact that two different transport rays cannot intersect at an interior point and
that $E_2^-=R(E_1^+)$. %^that 
%due to our definition of a typical point. 
Our claim follows. %\sout{We note also that it is possible to have some multiple points in $E_2^+$ but thanks to Lemma \ref{lem:Finite double point}, the set of multiple points $\mathcal{N}$ is finite and so, by modifying $E_1^+$ one can always assume that $E_2^+$ does not contain any multiple point.}

%WE HAVE TO RULE OUT THE POSSIBILITY THAT $e^+_1\in \chi^+_{i_2}$ AND $e^-_2 \in \Gamma^-_{i_2}$!} 

%But, thanks to Lemma \ref{l-conv} (2), \PR{[THIS REF. IS NOT RIGHT! WHICH DEFINITION OF $R$ IS USED HERE? IF ITS DOMAIN]}  the continuity of ${\bf{T}}$ (see Lemma \ref{cont-T}), the convexity of $\Omega$ and, the fact that two different transport rays cannot intersect at an interior point, then we see that $e_2^+$ must be also a multiple point. So, one can assume that . 

%\PR{
We presented above a general procedure: given an arc $E^+_1$ such that $E^+_1\cap\cN= \emptyset$ we constructed two arcs $E^-_2$ and $E^+_2$ by setting $R(E_1^+)=E_2^-$ and $\bbT(E_2^+)=E_2^-$. 
%In order to be able to keep going, to create a sequence of points. We have to check that $e^+_2$ is typical, i.e. $e^+_2.$ Since $e^+_2$ we may invoke Lemma \ref{l-neg-N} (2) to deduce that $e^+_2\in \cG$.}
Let us
suppose that arcs $E_1^+$, $E_2^-$, $E_2^+, \dots,\,E_n^+$ (where $n\ge 1$) have been already constructed by the above algorithm, where $R(E_n^+) \cap E_n^-=\emptyset$.
%\cap \chi_{i_n}^- =\emptyset$\, if \,$E_n^- \subset \chi_{i_n}^-$\, or \,$R(E_n^+) \,\cap \,\Gamma_{i_n}^- =\emptyset$\, if \,$E_n^- \subset \Gamma_{i_n}^-$.
%Then, \PR{all these points are in $\cG$ by Lemma \ref{l-neg-N}}. 
Then, we define $E^-_{n+1}$ as follows: if we have
%there is an arc of $E_n^+$ that we always denote by $E_n^+$ such that
$R(E^+_n) \subset  \chi^-_{i_1}$, %\PR{(resp.  $R(E^+_n) \subset  \Gamma^-_{i_1}$),}
%$R(E^+_n) \in \Gamma^-_{i_1}$) 
then we set
$E^-_{n+1}:= R(E^+_n)$ and 
the construction terminates. If  $R(E^+_n) \cap  \chi^-_{i_1}=\emptyset$, %\PR{(resp. $R(E^+_n) \cap  \Gamma^-_{i_1}=\emptyset$)},
%(resp.  $R(E^+_n) \cap \Gamma^-_{i_1}=\emptyset$), 
then %\sout{one can always assume that $R(E^+_n) \,\cap\,  E_n^-=\emptyset$\, }
%if \,$E_n^- \subset \chi_{i_n}^-$ (resp.  $R(E^+_n) \cap \Gamma^-_{i_n}=\emptyset$ if  $E_n^- \subset \Gamma_{i_n}^-$) 
%\sout{and so, }
we set 
$$
E^-_{n+1} := R(E^+_n)\qquad\mbox{and}\qquad E^+_{n+1} := \bbT^{[-1]}(E^-_{n+1})
$$
and due to Lemma  \ref{lem:Finite double point} we may guarantee, after possible restriction of the initial set $E^+_1$, that $E^-_{n+1} \cap \cN= \emptyset$. Hence, we know that $R(E^+_{n+1}) \,\cap\,  E_{n+1}^-=\emptyset$.
%where 
%$$
%e^\pm_{n+1} \in \chi^\pm_{i_{n+1}}\cap \cG\quad\hbox{or}\quad e^\pm_{n+1} \in \Gamma^\pm_{i_{n+1}}\cap \cG\,\,\,\, (\mbox{with}\,\,\,i_n \neq i_{n+1}\,\,\,\mbox{if}\,\,\,\{i_1,i_2\}\subset I_\chi\,\,\,\mbox{or}\,\,\,\{i_1,i_2\}\subset I_\Gamma).
%$$\\
%and
%$$
%e^+_{n+1} \in \chi^+_{i_{n+1}}\quad\hbox{or}\quad e^+_{n+1} \in \Gamma^+_{i_{n+1}}.
%$$ 
%More precisely, w
We note that this sequence of arcs $\{E_k^\pm\}$ may be either: (1) finite with elements
\begin{equation}\label{r-fin}
E^+_{1},\,E^-_{2},\,E^+_{2}, \ldots,\, E^-_{m},\,E^+_{m},\,E_{m+1}^-.
\end{equation}
(2) or infinite
\begin{equation}\label{r-linf}
E^+_{1},\,E^-_{2},\,E^+_{2}, \ldots,\,E^+_{n},\,E^-_{n}, \ldots
\end{equation}
In addition, thanks to Lemma \ref{lem:Finite double point} %\sout{one can assume that} 
we choose these arcs $E_k^\pm$ such that they do not contain any multiple point. 

Assume that the case \eqref{r-fin} holds. Let $e_{m+1}^-$ be any point in $E_{m+1}^-$ and set $e_m^+=R^{-1}(e_{m+1}^-) \in E_m^+$. Then, we define $e_m^-=\bbT(e_m^+)$ and $e_{m-1}^+=R^{-1}(e_{m}^-) \in E_{m-1}^+$. In this way, we get a sequence of points $\{e_k^\pm:1 \leq k \leq m+1 \}$ such that $e_k^\pm \in E_k^\pm$, 
\begin{equation}\label{df-epm}
R(e_k^+)=e_{k+1}^-\quad\hbox{and} \quad\bbT(e_k^+)=e_k^-.    
\end{equation} Therefore, Lemma \ref{l4.5} below tells us that
not only $e^-_{m+1}\in\chi^-_{i_1}$ %(resp.  $e^-_{m+1}\in\Gamma^-_{i_1}$) 
but also $e^-_{m+1} = e^-_1$. The final observation is made in Lemma \ref{l4.6}, where we claim that
in fact the sequence (\ref{r-fin}) consists of two elements 
$E^+_1$, $E^-_{m+1}=E_1^-.$ 

At last, we address the case of seemingly infinite sequence (\ref{r-linf}). We notice that since the number of arcs $\Gamma_i$ and $\chi_i$ is finite, then some of these arcs must be visited by the sequence more than once. In other words, a loop forms.
% and so,
%: if $e^+_k\in \chi_{i_k}^+$ (resp. $e^+_k\in \Gamma_{i_k}^+$), then there will be some $l\in \bN$ such that we have $e^+_{k+l}\in \chi_{i_k}^+$ (resp. $e^+_{k+l}\in \Gamma_{i_k}^+$). In this way
{As a result we are back to \eqref{r-fin}, hence the seemingly infinite sequence consists only of $E^+_1$ and $E^-_1$.} 
%we reduce the reasoning to the previous case, yielding the seemingly infinite sequence consisting only of $E^+_1$ and $E^-_1$. 
This is the content of
Lemma \ref{l4.7} below. Consequently, our proposition follows. $\qed$\\
Here are the Lemmas we referred to.
\begin{lemma}\label{l4.5} %step3}
Assume that \eqref{r-fin} holds. Let us consider the finite sequence of points \,$e^+_{1},\,e^-_{2},\,e^+_{2}$,\,..., $e^-_{m},\,e^+_{m}, \,e_{m+1}^-$ described above. Then, we have
$$
e^-_{m+1} = e^-_1:= \bbT(e^+_1).
$$
\end{lemma}
\begin{proof}
First, we claim that $f(\arc{e^+_1e^-_k})=0$ for all $k\ge 2$ (we note that this property holds even if the sequence \eqref{r-linf} is infinite). We will prove this by induction. Thanks to Lemma \ref{lzero}, we see that  $f(\arc{e^+_1e^-_2})=0$, because by definition $e_2^-=R(e_1^+)$.
%For this purpose, let us look at $R(\chi_1)$, {\color{red}{it is contained either in $\chi^-_{i_2}$ or in $\Gamma^-_{i_2}$. Due to the transport plan properties $f^+(\chi^+_{i_1}) = f^-(R(\chi^+_{i_1}))$. 
%Let us notice that the set $\partial \Omega \setminus( \chi^+_{i_1} \cup R(\chi^+_{i_1}) )$ has two connected components $E_1$ and $E_2$. Of course, $f(E_1 \cup E_2)=0$. We claim that $f(E_1) = f(E_2) =0.$ Indeed, if $f(E_1)>0$, then the mass must have been transported to $E_2$, but the transportation rays would have intersected rays $[y, R(y)]$ for $y\in \chi_1$, but this is impossible. 
%Now, we have 
%$$
%f(\arc{e^+_1e^-_2})= f^+(\arc{c_i e^+_1}) + f(E_1) - f^-(R(\arc{c_i e^+_1})) =  f^+(\arc{c_i e^+_1}) + 0 -  f^+(\arc{c_i e^+_1}) =0.
%$$}}
We also recall from \eqref{r-pro}  that $f(\arc{e^+_k e^-_k}) =0$, for all $2\le k\le m$.
Recursively, assume that $f(\arc{e_1^+ e_k^-})=0$, for some $k$. We shall prove that we also have $f(\arc{e_1^+ e_{k+1}^-})=0$. For this aim, we have to consider the following three possible configurations, while taking the arc $\arc{e_1^+ e^-_{k+1}}$ containing $e^-_k$ (the other cases when $e_k^- \notin \arc{e_1^+ e^-_{k+1}}$ can be treated similarly):\\
(i) $e^+_k \in \ \arc{e_k^- e_{k+1}^-}\,\subset\ \arc{e_1^+ e_{k+1}^-}$,\,\,\,\,\qquad(ii) $e^+_k \in\ \arc{e_1^+ e_{k}^-}\,\subset\ \arc{e_1^+ e_{k+1}^-}$,\,\,\,\qquad\,\,(iii) $e^+_k \in \partial\Omega\setminus \arc{e_1^+ e_{k+1}^-}$.\\
Let us start by considering case (i). Then, we see that the arcs $\arc{e^+_1 e^-_k}$, $\arc{e^+_k e^-_k},$ $\arc{e^+_k e^-_{k+1}}$ are disjoint. Hence,
$$
f ( \arc{e^+_1 e^-_{k+1}}) =
f( \arc{e^+_1 e^-_k} \cup \arc{e^+_k e^-_k} \cup \arc{e^+_k e^-_{k+1}})= f( \arc{e^+_1 e^-_k})+f(\arc{e^+_k e^-_k})+f( \arc{e^+_k e^-_{k+1}})=0.
$$
Indeed, here we used the inductive assumption $f( \arc{e^+_1 e^-_k})=0$, property (\ref{r-pro}) as well as Lemma \ref{lzero}, $f( \arc{e^+_k e^-_{k+1}}) = f( \arc{e^+_k R(e^+_k}))=0$. Now, we  take care of case (ii). We have
$$
f ( \arc{e^+_1 e^-_{k+1}}) = f( \arc{e^+_1 e^-_{k}} \cup \arc{e^+_k e^-_{k+1}}).
$$
However, the arcs $\arc{e^+_k e^-_{k+1}}$ and $\arc{e^+_1 e^-_{k}}$ have a nontrivial intersection which is the arc $\arc{e_k^+e_k^-}$. Thus, one has
$$
f ( \arc{e^+_1 e^-_{k+1}}) = f( \arc{e^+_1 e^-_{k}}) - f( \arc{e^+_k e^-_{k}})+ f( \arc{e^+_k e^-_{k+1}}) = 0,
$$\\
because each term in the sum above is zero. Finally, we consider the case (iii). We apply the same argument we used in (ii) and obtain
\begin{align*}
f ( \arc{e^+_1 e^-_{k+1}}) &=  f( \arc{e^+_1 e^-_{k}})+ 
 f( \arc{e^+_k e^-_{k}})
 %+  f( \arc{e^+_k e^-_{k+1}})
 -
 f( \arc{e^+_k e^-_{k+1}})=0.
 %&= 0 +   f( \arc{e^-_k e^+_{k}}) - 0 = 0.
\end{align*}
In particular, induction yields $f( \arc{e^+_1 e^-_{m+1}})=0$. We  also know that $e^-_{m+1}\in \chi^-_{i_1}$ 
%(resp.  $e^-_{m+1}\in \Gamma^-_{i_1}$)
and $f(\arc{e_1^+e_1^-})=0$. Hence, we infer that $f^-( \arc{e^-_1 e^-_{m+1}})=0$ and so, $e_{m+1}^-=e_1^-$. $\qedhere$
%{\color{red}{We have two case to consider: (a) $c_{i_1}\in \arc{e^+_k e^-_{m+1}}$ or (b) $c_{i_1}\not\in \arc{e^+_k e^-_{m+1}}$.
%In the first case we see that 
%$$
%0= f(\arc{e^+_1 e^-_{m+1}})= f(\arc{e^+_1 c_{i_1}} \cup \arc{c_{i_1} e^-_{m+1}}) = TV g|_{\arc{e^+_1 c_{i_1}}} -  
%TV g|_{\arc{c_{i_1}e^-_{m+1}}}.
%$$
%This, means that $e^-_{m+1} =T(e^+_{i_1}) = e^-_{i_1}$.
%In the second case, we set $\gamma = \partial \Omega \setminus  \arc{e^+_k e^-_{m+1}}$. Since $f(\gamma) =0,$ we may apply the argument used to settle (a).}}
\end{proof}
%In a moment we shall see that the sequence $\{e_k^\pm:1 \leq k \leq m+1 \}$
%is finite, then it
%consists of the couple $\{e^+_1,\,e^-_1\}$. Indeed, we have:
In addition, we have the following:
\begin{lemma}\label{l4.6}%-sk}
The sequence $\{e_k^\pm:1 \leq k \leq m+1 \}$ consists of the couple $\{e^+_1,\, e^-_1\}$. 
\end{lemma}
\begin{proof} Let us suppose the contrary and $m>1$. Thanks to
%We established in the course of the proof of
Lemma \ref{l4.5}, %step3}  
we have that $(e_m^+,e_1^-) \in \spt(\gamma)$. Yet, $(e_k^+,e_{k+1}^-) \in \spt(\gamma)$, for all $1 \leq k \leq m-1$. Let $\phi$ be a Kantorovich potential for measures $f^+$ and $f^-$. Hence, by the equality \eqref{transport ray}, we must have the following inequality:
\begin{align*}
\sum_{k=1}^{m-1} |e_k^+-e_{k+1}^-| + |e_m^+ -e_1^-| &= \sum_{k=1}^{m-1} [\phi(e_k^+)-\phi(e_{k+1}^-)] + [\phi(e_m^+) -\phi(e_1^-)] = \sum_{k=1}^m [\phi(e_k^+) - \phi(e_k^-)]\\
&\le \sum_{k=1}^m |e_k^+ - e_k^-|,
\end{align*}
where the last inequality follows from the fact that $\phi$ is $1-$Lipschitz. Considering the sequence's definition in (\ref{df-epm}), it follows that the above inequality contradicts our condition (H3). %This concludes the proof. 
\end{proof}
Finally, it remains to show that the sequence (\ref{r-linf}) cannot be infinite.
%Then, the above Lemma will do the job.
\begin{lemma}\label{l4.7}
The apparently infinite sequence (\ref{r-linf}) is in fact finite and it consists of the arcs $E^+_1,\,E^-_1$. 
\end{lemma}
\begin{proof}
Let us consider the infinite sequence (\ref{r-linf}). Since the number of arcs \,$\Gamma_i$\, and \,$\chi_i$\, is by assumption finite, then there exist $l\geq 1$ and $k\geq 2$ such that $E^-_l,\,E^-_{l+k}\subset \chi^-_{i_l}$ (or  $E^-_l,\,E^-_{l+k}\subset \Gamma^-_{i_l}$).
%with $e_{l+1}^- \neq e_l^-$. 
%Suppose that for a fixed $l$, number $k$ is the smallest with this property. 
Let us denote again by $\{e_i^\pm\}_i$ the sequence of points on these arcs $E_i^\pm$ defined above in the course of proof of Proposition \ref{strictly_convex_case}.
Thanks to Lemma \ref{l4.5},  
one can show similarly that
%We first claim that
$e^-_{l+k}=e^-_l$. 
%%Indeed, by the course of the proof of Lemma \ref{step3}, we know that $f(\arc{e^+_1 e^-_{l+k}}) =0.$ At the same time the arc $\arc{e^+_1 e^-_{l+k}}$ is the sum of the following arcs, $\arc{e^+_1 e^-_l}$, $\arc{e^-_l e^-_{k+l}}$ and the arc starting from $e^-_{l+k}$, passing   through $e^+_{k+l-1}$ and going back to $e^-_{l+k}$. These arc are disjoint. Thus,
%$$
%0 = f(\arc{e^+_1 e^-_{k+l}}) = f(\arc{e^+_1 e^-_l}) 
%+f( \arc{e^-_l e^-_{k+l}}) + f( \arc{e^-_{l+k}e^+_{k+l-1}}).
%$$
%Due to  Lemma \ref{step3} the first and the third terms on the RHS are zero. This means that $f( \arc{e^-_l e^-_{k+l}}) =0$ that is possible if and only if $e^-_l = e^-_{k+l}$.
Now, we have a finite sequence starting at $e^+_{l}\in \chi^+_{i_{l}}$ (or \,$e^+_{l}\in \Gamma^+_{i_{l}}$) (this point will take the role of a new $e^+_1$) and terminating at $e^-_{l+k}=e_l^-\in \chi^-_{i_l}$ (or \,$e^-_{l+k}=e_l^-\in \Gamma^-_{i_l}$). Hence, we are in a situation, when we deal with a finite sequence, we invoke Lemma \ref{l4.6} to deduce that the loop consists of  a pair $\{e^+_{l},\,e^-_l\}$. But, this is a contradiction. $\qedhere$
%with existence of  a loop.
\end{proof}
In this way, we have provided all the details for the proof of Proposition \ref{strictly_convex_case}, which will play 
%by the way 
an important role in the next section to show that even if $\Omega$ is convex and not necessarily strictly convex, the transportation map ${\bf{T}}$ will be also the optimal transport map $R$ between $f^+$ and $f^-$. We conclude this section with the following observation.
\begin{remark} \label{Remark 3.7}
The scrutiny of the proof of Lemma \ref{l4.6} shows  that in the case when $\Omega$ is strictly convex, one can relax the condition (H3) by assuming that inequality (\ref{rH3}) holds only for every sequence of points $\{e_i^+\}_{1 \leq i \leq m}$ %such that $e_i^+ \in \Gamma_{i}^+ \cup \chi_{i}^+$.
%, for some index $i_k \in I_\Gamma \cup I_\chi$, with $i_k \neq i_{k^\prime}$ if $k \neq k^\prime$ and $\{i_k,\,i_{k^\prime}\}\subset I_\chi$ or $\{i_k,\,i_{k^\prime}\}\subset I_\Gamma$. 
%In other words, we only need to assume that the points $e_i^+$
that belong to different arcs $\Gamma_i^+$ ($i \in I_\Gamma$) or $\chi_i^+$ ($i\in I_\chi$). This remark will be used in Step 3 of Proposition \ref{Prop. convex case with finite arcs}. \\
\end{remark}

%{\color{blue}{
\subsection{The case of convex, but not strictly convex domain} In this section, we will extend the result of Proposition \ref{strictly_convex_case} to
the  case of %merely 
convex domain $\Omega$, without assuming its strict convexity. We will proceed in a natural manner by finding a sequence of strictly convex domains $\Omega_n$ whose closures converge to $\overline{\Omega}$ in the Hausdorff metric. At the same time, we have to come up with a good choice of the boundary data $g_n$ defined on $\partial\Omega_n$ and approximating $g$, to keep the assumptions (H1) and (H3) in force. %Thus, our task is to ensure that the approximating sequence enjoys this property.
Since our approximation is based on the orthogonal  projection onto $\overline{\Omega}$, then the singular points of the support of $f$  will play a role. We recall their definition
%we define the set of singular points on the set where $f$ is concentrated as follows:
%starts playing a role. We set
$$
\mathcal{S}:=\bigg\{x\in \Gamma \cup \chi\,:\, \mbox{there is no tangent line to }\partial\Omega \,\,\mbox{at}\,\,x\bigg\}.
$$

We give our first observation about the projection map onto a convex set.

\begin{lemma} \label{Lemma 4.7}
Suppose that \,$\Omega\subset \bR^2$ is an open bounded convex  set and \,$\overline{\Omega}\subset \Omega^\prime$, where $\Omega^\prime$ is open bounded and strictly convex. We assume that $\tilde P:\mathbb{R}^2\mapsto \overline{\Omega}$ is the orthogonal projection, its restriction to $\partial \Omega^\prime$ will be denoted by $P$. 
%If \,$\gamma\subset \partial\Omega$ is a piecewise $C^1$ arc (with a finite number of $C^1$ pieces), 
For every \,$x_0 \in \partial\Omega$, we have the following:\\
(1) if \,$x_0\notin \mathcal{S}$, then %for all \,$x\in \gamma$ 
the preimage $P^{-1}(x_0)$ is a singleton,\\
(2) if %\,$\gamma$ is  piecewise $C^1$ and 
%\,$x_0\in \partial\Omega$ is a singular point (i.e. 
\,$x_0\in \mathcal{S}$, then $P^{-1}(x_0)\subset \partial\Omega^\prime$ is an arc of positive Hausdorff measure.
\end{lemma}
\begin{proof}
(1) In this case, the normal cone $N(x_0)$ reduces to a ray, which has a unique intersection with  $\partial\Omega^\prime$.

%\medskip\noindent
(2) If $x_0$ is a singular point, then the normal cone has a positive opening. Its intersection with   $\partial\Omega^\prime$ has a positive Hausdorff measure. $\qedhere$
\end{proof}

The conclusion from this lemma is as follows. If we approximate $\Omega$ by a sequence of strictly convex sets $\Omega_n$ and we try to partition the boundaries of these sets into arcs of types $\chi_{i,n}^\pm$ and $\Gamma_{i,n}^\pm$, then we will meet some difficulty since Lemma \ref{Lemma 4.7} tells us that $(P_{|\partial\Omega_n})^{-1}$ will not preserve the structure of arcs $\chi_i^\pm$ and $\Gamma_i^\pm$ due to possible presence of singular points on these arcs. As a result, we are forced to rearrange our partition into arcs $\chi^\pm_i$ and $\Gamma^\pm_i$ in such a way that the new partition avoids singular points. We will then make the following assumption about the set of singular points
%$$
%\mathcal{S} \hbox{ is at most countable, but not dense anywhere.} \leqno(S)
%$$ \marginnote{replace with}
%\PR{$$
%\mathcal{S} \hbox{ is closed and } \cH^1(S) + |f|(S) =0. \leqno(S)
%$$}
$${|f|(\overline{\mathcal{S}}) =0.} \leqno(\rm{S})
$$

We note that the assumption (S) is satisfied as soon as the set \,$\mathcal{S}$\, is countable, in particular, this will cover the case of polygonal \,$\Omega$.

\begin{lemma}\label{l-repa}
Let us suppose that $\Omega$ is convex and conditions (H1), (H2), (H3) and (S) are satisfied. Then, there exists  another partition of \,$\partial\Omega$ into smooth arcs \,$\tilde\chi_i^\pm$ ($i\in I_{\tilde{\chi}} \equiv I_\chi$), $\tilde \Gamma^\pm_i$ ($i\in I_{\tilde{\Gamma}}$ where $I_\Gamma \subset I_{\tilde{\Gamma}}$) and $\tilde{F}_i$ ($i \in I_{\tilde{F}}$ where $I_F \subset I_{\Tilde{F}}$), satisfying the same conditions (H1), (H2) and (H3).
%%and, a family of exceptional arcs ${\chi}_{i,0} \subset \chi_i$ with  $\chi_{i,0}\backslash \{c_i\}$ is smooth and $f( {\chi}_{i,0})=0$, for every $i \in I_\chi$ such that $c_i$ is a singular point.
\end{lemma}
\begin{proof}
{Since %$\bar S$ is closed, then 
$\partial\Omega\setminus \overline{ \mathcal{S}}$ is open, then there are open arcs $U_k$ ($k\in I \subset \mathbb{N}$) such that
$$
\partial\Omega\setminus \overline{\mathcal{S}} = \bigcup_{k\in I}U_k.
$$
Consider the following families of open arcs
$$
\tilde \Gamma^+_{k,i,j}:=\Gamma_k^+ \cap U_i \cap \bbT^{[-1]}(U_j)\qquad \mbox{and}\qquad \tilde \Gamma^-_{k,i,j}:=\Gamma_k^- \cap \bbT(U_i) \cap U_j,
$$
with $k\in I_{\Gamma} $ and $i,j\in I$. In fact, these arcs $\tilde \Gamma^\pm_{k,i,j}$ are in a one to one correspondence, because we have 
$$
\bbT(\tilde \Gamma^+_{k,i,j}) = \bbT(\Gamma_k^+ \cap U_i \cap \bbT^{[-1]}(U_j)) =\Gamma_k^- \cap \bbT(U_i) \cap U_j
= \tilde \Gamma^-_{k,i,j}.
$$
\\
Moreover, it is clear that $\tilde \Gamma^+_{k,i,j}$ and $\tilde \Gamma^-_{k,i,j}$ inherit the properties of $\Gamma_k^+$, $\Gamma^-_k$.

When we subdivide arcs $\chi_k^\pm$ we proceed slightly differently. We first  construct similar families
$$
\chi_k^+ \cap U_i \cap \bbT^{[-1]}(U_j)\qquad \mbox{and}\qquad \chi_k^- \cap \bbT(U_i) \cap U_j,\quad\ k\in I_\chi,\,\,i,\,j\in I.
$$
We consider two cases: $c_k \in \overline {\mathcal{S}}$\, and \,$c_k \notin \overline{\mathcal{S}}$. In the first case, we set
$$
 \hat \Gamma^+_{k,i,j}:=\chi_k^+ \cap U_i \cap \bbT^{[-1]}(U_j)\qquad \mbox{and}\,\,\,\,\quad  \hat \Gamma^-_{k,i,j}:=\chi_k^- \cap \bbT(U_i) \cap U_j,\quad\ k\in I_\Gamma,\,\,i,\,j\in I.
$$
Since \,$\bbT(\hat \Gamma^+_{k,i,j}) = \hat \Gamma^-_{k,i,j}$, we can see that these arcs satisfy the conditions required for arcs of type $\Gamma$.

In the case when $c_k \notin \overline{\mathcal{S}}$, there exist $i_0,\,j_0\in I$ such that if we set
$$
\hat \chi^+_{k,i_0, j_0}:= \chi^+_k \cap U_{i_0} \cap \bbT^{[-1]}(U_{j_0}), \qquad \hat \chi^-_{k,i_0, j_0}:= \chi^-_k \cap\bbT (U_{i_0}) \cap U_{j_0},
$$
then $\overline{\chi^+_{k,i_0, j_0}}\cap \overline{\chi^-_{k,i_0, j_0}}= \{c_k\}$ and $\chi^+_{k,i_0, j_0}$, $ \chi^-_{k,i_0, j_0}$
inherit the properties of $\chi^+_k$, $\chi_k^-$. This concludes the proof.}  $\qedhere$
\end{proof}

Thanks to Lemma \ref{l-repa}, one can assume that the decomposition of $\partial\Omega$  is such that all arcs $\chi^\pm_i$ ($i\in I_\chi$) and $\Gamma^\pm_i$ ($i\in I_\Gamma$) are smooth. 
%However, we will closely monitor those $\chi_i$ whose vertex is a singular point, i.e. $c_i\in S$. At the same time 
However, it is possible  that the point %vertex 
$c_i = \overline{\chi^+_i} \cap \overline{\chi^-_i}$ %of an arc $\chi_i$ 
is a singular point. Moreover, we note that we do not care about singular points on the flat parts $F_i$ ($i\in I_F$).\\

Now, we aim to extend the result of Proposition \ref{strictly_convex_case} to the case of \,$\Omega$\, being  just convex. We do this by using an approximation argument. First, let us assume that the set of singular points $\mathcal{S}$ is finite. In addition, we %also need to 
strengthen our condition (H3) a little bit by assuming %in the sense 
that for every $m \in \mathbb{N}$, there is a \,$\delta_m >0$\, such that for 
%$\bullet$ {\bf{Condition (H3).}} Let us consider 
any sequence of points $\{e_{k}^+\}_{1 \leq k \leq m}$ described in (H3), the 
%such that for every $k \in \{1,...,m\}$, there exists an index $i_k\in I_\chi \cup I_\Gamma$ with the property that $e_{k}^+ \in \chi_{i_k}^+$\, or \,$ e_{k}^+ \in \Gamma_{i_k}^+$. Then, we assume the
%\begin{definition}\label{dH3}
%We shall say that the couple $(\Omega, g)$ satisfies  condition (H3) if for any $m\in\bN$ and  
%for any finite sequence of points $\{e_{k}^+\}_{1 \leq k \leq m}$ as above the 
following inequality holds:
$$
\sum_{k=1}^m |e_k^+ - {\bf{T}}(e_k^+)|
\leq
\sum_{k=1}^{m-1} %\left(
|e_k^+ - {\bf{T}}(e_{k+1}^+)| + |e_m^+ - {\bf{T}}(e_{1}^+)| - \delta_m.\leqno(\mbox{H}3)^\prime$$ %\right).
%\end{equation}
%\end{definition}
%We remark that this condition is weaker than \cite[Condition (C2)]{rs1} and so, this allows us to extend the class of boundary data to which we can deduce existence, see Example \ref{Example1} below.\\

%Hence, we have the following:

%At this point we note that if $c_i \notin S_i$ then we will still have a smooth arc of the form $\chi_{i,0} \subset \chi_i$ containing $c_i$ with $f(\chi_{i,0})=0$. Consequently,  the number of arcs $\Gamma_i^\pm$ increases, we denote by $I_\Gamma'$ the new set of indices. 

\begin{proposition}\label{Prop. convex case with finite arcs}
Let us assume that $\Omega$ is convex, conditions (H1), (H2) $\&$ (H3)$^{\,\prime}$ are satisfied, the set \,$\mathcal{S}$ is finite and the number of arcs $\chi_i^\pm$ ($i \in I_\chi$) as well as   \,$\Gamma_i^\pm$ ($i \in I_\Gamma$)
%such that $\chi_i^\pm$ (resp. $\Gamma_i^\pm$) is not strictly convex 
is finite. %Moreover, each of the arcs $\chi_i^\pm$ (resp. $\Gamma_i^\pm$) is of class $C^1$. %except for  a finite set of points. %Let $\gamma$ be an optimal transport plan. 
Then, ${\bf{T}}$ is an optimal transport map from $f^+$ to $f^-$.
\end{proposition}
\begin{proof}
%Due to Lemma \ref{l-repa} we can assume that the arcs $\chi_i$, $i\in I_\chi$ and $\Gamma_j$, $j\in I_\Gamma'$ do not contain any singularity point. The only exception may be the vertex $\{c_i\} = \overline{\chi^+_i}\cap \overline{\chi^-_i}$.
Our task will be to find a sequence of decreasing strictly convex regions $\Omega_n$ such that $\overline{\Omega}_n$ converges to $\overline{\Omega}$ in the Hausdorff distance. At the same time, we need to approximate the boundary datum \,$g$\, by a sequence of functions $g_n$ defined on $\partial\Omega_n$. Our choice of $(\Omega_n, g_n)$ must be such that for every $n \in \mathbb{N}$, the boundary $\partial\Omega_n$ can be decomposed into arcs $\chi_{i,n}$, $\Gamma_{i,n}$ and $F_{i,n}$ for appropriate sets of indices. 

We divide the proof into several steps:

{\it Step 1.} Since $\Omega$ is convex, then due to \cite[Theorem 20.4]{Rockafellar} one can  always 
%The proof is by approximation of a convex set with strictly convex domains. This ispossible since if \,\,  we can
find a sequence of decreasing closed polygons $\Delta_n$ containing \,$\Omega$ and
%contained in $\Omega_{n-1}$ 
%\todo{possible\\ problems}
converging to $\overline{\Omega}$ in the Hausdorff distance. After that, we use the argument in \cite{rs1} to construct a strictly convex region $\tilde\Omega_n$ containing $\Delta_n$ such that the  Hausdorff distance between them does not exceed $\frac1n$. %$\Delta_n$. 
Thus, we obtain a decreasing sequence of strictly convex domains $\Omega_n$ such that $\overline{\Omega}_n \to \overline{\Omega}$ in the Hausdorff metric. Indeed, if we set $\Omega_n = \bigcap_{i=1}^n\tilde{\Omega}_i$, then %we see that
$$
\Omega \subset \Omega_{n+1} = \tilde\Omega_{n+1}\cap \Omega_n \subset \Omega_n.
$$
%Hence, the sequence $\Omega_n $ has the desired properties.

{\it Step 2.}  Due to Lemma \ref{l-repa} we may assume that all arcs $\chi^\pm_i$ and $\Gamma^\pm_j$ are smooth because \,$\mathcal{S}$\, is finite.

We recall that $\tilde P$, the orthogonal projection %map 
onto $\overline{\Omega}$, is Lipschitz continuous on \,$\mathbb{R}^2$. Let us fix $n \in \mathbb{N}$ and set $\tilde P_n:= \tilde P|_{\partial\Omega_n}$, $\mathcal{C}_{i,n}^\pm:=\tilde P_n^{-1}(\chi_i^\pm)$ for $i\in I_\chi$, and ${\mathcal{C}^\prime}^\pm_{i,n}:=\tilde P_n^{-1}(\Gamma_i^\pm)$ for $i \in I_\Gamma$,
%$\mathcal{C}_n^\pm:=\bigcup_{i \in 
%I_\chi}\mathcal{C}_{i,n}^\pm$, ${\mathcal{C}^\prime_n}^\pm:=\bigcup_{i \in I_\Gamma}{\mathcal{C}^{\,\prime}}^\pm_{i,n}$ and, \,
$\mathcal{C}_n:=\bigcup_{i \in 
I_\chi \cup I_\Gamma}(\mathcal{C}_{i,n}^\pm \cup {\mathcal{C}^{\prime}}^\pm_{i,n})$.
%(\mathcal{C}_n^+ \cup {\mathcal{C}_n^\prime}^+) \cup (\mathcal{C}_n^- \cup {\mathcal{C}^\prime_n}^-)$.
Let $P_n: \mathcal{C}_n \mapsto \Gamma \cup \chi$ be a further restriction of $\tilde{P}_n$ to \,$\mathcal{C}_n$.  Thanks to Lemma \ref{Lemma 4.7}, we see that $P_n$ 
 %restricted to each $\mathcal{C}_{i,n}^\pm$
is one-to-one, hence by the open mapping theorem the inverse of $P_n$ is continuous too. %However, the one sided limits of  $P^{-1}_{n}$ at $c_i\in S$ may differ, where we 
 %%%This follows directly from Lemma \ref{Lemma 4.7}. 
%the fact that there is a normal space to $\partial \Omega$ at $x\in\partial\Omega\setminus S_\Omega.$
In the sequel, we
denote by 
$P^{-1}_{n}:\Gamma \cup \chi \mapsto \mathcal{C}_{n}$ the inverse map of $P_n$. Then, we define a measure $\tilde f_n$ on 
%$\partial \Omega_n$
{$\cC_n$} as follows:
    $$\tilde f_n:= (P_{n}^{-1})_{\#}f.$$
\\
Since $f$ is concentrated on \,$\Gamma \cup \chi$\, and \,$P_{n}^{-1}$ is continuous on \,$\Gamma \cup \chi$, then the measure $\tilde f_n$ is well defined. Now, we extend it to a Borel measure on $\partial \Omega_n$ by setting
$$
f_n(B) := \tilde f_n (B \cap \cC_n),
$$
for any Borel set $B\subset\partial\Omega_n$.
By the definition of 
$f_n$ we have  that $|f_n|(\partial\Omega_n \backslash \mathcal{C}_{n})=0$. It is also clear that $f_n(\partial\Omega_n)=f(\partial\Omega)=0$ and
$|f_n|(P^{-1}(\mathcal{S})\cap \partial\Omega_n) = 0.$

Moreover, if $x_1,\,x_2 \in \mathcal{C}_{n}$ then we have
$$f_n(\arc{x_1x_2})=f(\arc{P_n(x_1)P_n(x_2)})$$
because %where the arc 
$\arc{P_n(x_1)P_n(x_2)}=P_n(\arc{x_1 x_2})$. %(here, $\arc{x_1 x_2}$ is the arc from \,$x_1$\, to \,$x_2$\, going in the positive orientation).
%Let us keep in mind the notation convention we introduced before 
%in the Introduction 
%regarding the definition of arcs. 

After these preparations, we define the trace function $g_n$ on $\partial \Omega_n$. Since it has %supposed 
to satisfy %such that 
$\partial_\tau g_n=f_n$, then we proceed
as follows. For a fixed $x_0 \in \partial\Omega\backslash \mathcal{S}$, we define $x_n := P_n^{-1}(x_0)\in \partial\Omega_n$. Then, we set
$$g_n(x):=f_n(\arc{x_{n}x}),\,\,\mbox{for every }\,\,x \in \partial\Omega_n.$$\\
%where $\arc{x_n x}$ is the arc from $x_n$ to $x$ going in the positive orientation.
Since $f$ is atomless,  so is $f_n$. Hence, $g_n \in C(\partial\Omega_n)$.\\

Let us discuss the forms of the sets \,$\cC^\pm_{i,n}$ and ${\cC^\prime}^\pm_{i,n}$. We notice that 
${\mathcal{C}^\prime}^\pm_{i,n}=P_n^{-1}(\Gamma_i^\pm)$ are  two arcs of the form $\Gamma_{i,n}^\pm$, because we have \,$\overline{{\mathcal{C}^\prime}^{\,+}_{i,n}}\cap \overline{{\mathcal{C}^\prime}^{\,-}_{i,n}} =\emptyset$\, and
$$f_n({\mathcal{C}^\prime}^{\,+}_{i,n})=f^+(\Gamma_i^+)\,\,\,\,\,\,\mbox{and}\,\,\,\,\,\,f_n({\mathcal{C}^\prime}^{\,-}_{i,n})=f^-(\Gamma_i^-).$$
%and $\overline{\mathcal{C}_{i,n}^+}\cap \overline{ \mathcal{C}_{i,n}^-} =\emptyset.$
Moreover, $g_n$ is strictly increasing on ${\mathcal{C}^\prime}^{\,+}_{i,n}$
(resp. decreasing on ${\mathcal{C}^\prime}^{\,-}_{i,n}$),  because if \,$x_1,\,x_2 \in {\mathcal{C}^\prime}^{\,\pm}_{i,n}$ are such that $x_1<x_2$, then $P_n(x_1),\,P_n(x_2) \in \Gamma_i^\pm$ with $P_n(x_1)<P_n(x_2)$ and, we have
$$g_n(x_2)-g_n(x_1)=f_n(\arc{x_1x_2})=f(\arc{P_n(x_1)P_n(x_2)}),$$
where $\arc{x_1 x_2} \subset {\mathcal{C}^\prime}^{\pm}_{i,n}$ and $\arc{P_n(x_1)P_n(x_2)}=P_n(\arc{x_1 x_2}) \subset \Gamma_i^\pm$.
 
The above  argument applies too when the %vertex 
point $c_i$ (the common endpoint of $\chi^\pm_i$) %of the arc $\chi_i$ 
is a singular point. Hence, %so in this case 
$\mathcal{C}_{i,n}^\pm=P_n^{-1}(\chi^\pm_i)$ are  two arcs of the form $\Gamma_{i,n}^\pm$, because by Lemma \ref{Lemma 4.7}, $P_n^{-1}(c_i)$ is %in that case 
an arc of positive $\cH^1$ measure (see Figure \ref{Fig. 1}).
\begin{figure}[h]
\begin{tikzpicture}
  % Draw the transp
  \draw (0, 0) rectangle (4, 4);
% Drawing the points
 % \node at (1, -.3) {$\delta$};
  %\node at (2.5, -.3) {$1-\delta$};
%   \fill (1,0) circle (2pt);
 %   \fill (0,1) circle (2pt);
  %    \fill (3,0) circle (2pt);
   %     \fill (0,3) circle (2pt);
    %      \fill (4,1) circle (2pt);
     %       \fill (4,3) circle (2pt);
      %       \fill (1,4) circle (2pt);
       %        \fill (3,4) circle (2pt);
% Drawing the transport rays           
    \draw[red,line width=2pt] (1,0)--(0,1);
     \draw[blue,line width=2pt] (3,4)--(4,3);
      \draw[blue,line width=2pt] (2,4)--(4,2);
     
    %\draw[red,line width=2pt] (3,0)--(4,1);
    %\draw[red,line width=2pt] (0,3)--(1,4);
    %\draw[red,line width=2pt] (4,3)--(3,4);
% Inserting the Xi 
\node at (0.4, -.3) {$\chi_1^+$};
\node at (-.3, .5) {$\chi_1^-$};
%\node at (3.5, -.3) {$\chi_2^-$};
%\node at (4.4, 3.6) {$\chi_3^-$};
%\node at (4.4, 0.5) {$\chi_2^+$};
%\node at (0.4, 4.3) {$\chi_4^-$};
%\node at (3.5, 4.3) {$\chi_3^+$};
%\node at (-0.4, 3.6) {$\chi_4^+$};
\node at (2.4, 4.5) {$\Gamma_1^-$};
\node at (4.5, 2.4) {$\Gamma_1^+$};
\node at (2,2){ $\Omega$};

\draw (6,2) arc (0:360:4);
\draw[dashed] (0,0)--(0,-1.4);
\draw[dashed] (1,0)--(1,-1.8);
\draw[dashed] (0,0)--(-1.4,0);
\draw[dashed] (0,1)--(-1.8,1);

\draw[dashed] (3,4)--(3,5.8);
\draw[dashed] (2,4)--(2,6);
\draw[dashed] (4,3)--(5.8,3);
\draw[dashed] (4,2)--(6,2);

\node at (0.5, -2.4) {$\Gamma_{1,n}^+=P_n^{-1}(\chi_1^+)$};
\node at (-3,0.5) {$\Gamma_{1,n}^-=P_n^{-1}(\chi_1^-)$};

\node at (2.5, 6.3) {$\Gamma_{2,n}^-=P_n^{-1}(\Gamma_1^-)$};
\node at (7.3,2.5) {$\Gamma_{2,n}^+=P_n^{-1}(\Gamma_1^+)$};

\node at (0,5.1) 
{$\Omega_n$};

\draw[color=red] (0,-1.42)--(-1.42,0);
\draw[color=red] (1,-1.85)--(-1.85,1);
\draw[color=blue] (5.9,2.95)--(2.96,5.9);
\draw[color=blue] (2,6)--(6,2);

%\draw[dashed] (0,0)--(0,-1.4);
%\draw[dashed] (0,0)--(0,-1.4);

   % \coordinate (Pt1) at (0, 1);
  %  \fill (Pt1) circle (1pt)
   % \coordinate (Pt2) at (0, 3);
    %\fill (Pt2) circle (1pt)
    %\coordinate (Pt3) at (1, 0);
    %\fill (Pt3) circle (1pt)
    %\coordinate (Pt4) at (3, 0);
    % \fill (Pt4) circle (1pt)
      
  % Draw the points on each side
  %\fill (0.5, \delta) circle (2pt);
  %\fill (0.5, 1-\delta) circle (2pt);
  %\fill (\delta, 0.5) circle (2pt);
  %\fill (1-\delta, 0.5) circle (2pt);
\end{tikzpicture}
\caption{}
\label{Fig. 1}
\end{figure}

%if \,$\chi_i$\, is $C^1$ then the arc \,$\mathcal{C}_{i,n}=\mathcal{C}_{i,n}^+ \cup \mathcal{C}_{i,n}^-$, where

Finally, when $c_i\not\in \mathcal{S}$, then by Lemma \ref{Lemma 4.7}, $P_n^{-1}(c_i)$ is a singleton and $\overline{\mathcal{C}_{i,n}^+}\cap \overline{ \mathcal{C}_{i,n}^-} =\{P_n^{-1}(c_i)\}$. Hence,
$\mathcal{C}_{i,n}^\pm=P_n^{-1}(\chi_i^\pm)$ are two arcs of the form $\chi_{i,n}^\pm$. Indeed, we have
$$f_n(\mathcal{C}_{i,n}^+ \cup \mathcal{C}_{i,n}^-)=f(\chi_i)=0.$$
We also see  that $g_n$ is strictly increasing on \,$\mathcal{C}_{i,n}^+$ (resp. decreasing on \,$\mathcal{C}_{i,n}^-$), because if \,$x_1,\,x_2 \in \mathcal{C}_{i,n}^\pm$ are such that $x_1<x_2$, then $P_n(x_1),\,P_n(x_2) \in \chi_i^\pm$ with $P_n(x_1)<P_n(x_2)$. Moreover, 
%and one has again that
 $$g_n(x_2)-g_n(x_1)=f_n(\arc{x_1x_2})=f(\arc{P_n(x_1)P_n(x_2)}), $$
 where \,$\arc{x_1 x_2} \subset \mathcal{C}_{i,n}^\pm$\, and \,$\arc{P_n(x_1)P_n(x_2)}=P_n(\arc{x_1 x_2}) \subset \chi_i^\pm$.\\
%if $\Gamma_i^\pm$ are $C^1$ then again $\mathcal{C}_{i,n}=\mathcal{C}_{i,n}^+ \cup \mathcal{C}_{i,n}^-$, where 
 
%In the same way, one can see that if $\chi_i$ is not $C^1$ then $\mathcal{C}_{i,n}:=\mathcal{C}_{i,n}^+ \cup \mathcal{C}_{i,n}^-$ can be decomposed into a finite number of arcs: $\chi_{i,n}^\pm$ and $\Gamma_{i,j,n}^\pm$, where $1 \leq j \leq m_i$, for some $m_i \in \mathbb{N}$ depending on the number of singularity points on $\chi_i$. 
%A similar argument applies if $\Gamma_i^+$ or $\Gamma_i^-$ is not $C^1$, then $\mathcal{C}_{i,n}:=\mathcal{C}_{i,n}^+ \cup \mathcal{C}_{i,n}^-$ can be decomposed into a finite number of arcs: $\Gamma_{i,j,n}^\pm$, where $1 \leq j \leq m_i$; again $m_i \in \mathbb{N}$ depends on the number of singularity points on $\Gamma_i$. Since the number of arcs $\chi_i^\pm$ and $\Gamma_i^\pm$
%such that $\chi_i^\pm$ (resp. $\Gamma_i^\pm$) is not strictly convex 
%is finite (i.e. $n<\infty$ in (H1)) and, $\chi_i^\pm$ (resp. $\Gamma_i^\pm$) is of class $C^1$ except at a finite set of points, then we have a finite number (say $m_\star$) of arcs $\chi_{i,n}^\pm$ and $\Gamma_{i,j,n}^\pm$. Note that this number $m_\star$ does not depend on $n$.
{\it Step 3.} Now, let us check that the assumption (H3) is also satisfied by $g_n$.
     %Let $P^{[-1]}_n$ be the inverse map of $P$ on $\partial\Omega_n$. 
Let ${\bf{T}}_n$ be the transportation map defined on $\Gamma_n^+ \cup \chi_n^+$ (see \eqref{df-T}). First, we see that if $e^+ \in \Gamma_n^+ \cup \chi_n^+$,
%(resp. $e^+ \in \Gamma_{i,n}^+$) 
%[\PR{PLEASE ADD $e^+$ etc FURTHER IN THE TEXT}] 
then we have
\begin{equation}\label{r-dfTn}
e^-:={\bf{T}}_n(e^+)=P_{n}^{-1}({\bf{T}}(P_n(e^+))).
\end{equation}
%\PR{INCONSISTENCY: WHAT IS $P_{i,n}$}
%If $e^+ \in \chi_{i,n}^+$ (resp. $\Gamma_{i,n}^+,\,\Gamma_{i,j,n}^+$), then $P(e^+) \in \chi_i^+$ (resp. $\Gamma_i^+$, $\chi_i^+$ or $\Gamma_i^+$) and so, ${\bf{T}}(P(e^+))\in \chi_i^-$ (resp. $\Gamma_i^-$, $\chi_i^-$ or $\Gamma_i^-$). Moreover, we have
Indeed, if 
%$e^+ \in \chi_{i,n}^+$, $i \in I_{\chi_n}$ (resp. 
$e^+ \in \Gamma_{i,n}^+$, $i \in I_{\Gamma_n}$, ({the argument when $e^+ \in \chi_{i,n}^+$, $i \in I_{\chi_n}$, is the same and it will be omitted}), %then $P_n(e^+) \in \chi_i^+$ (resp. 
then $P_n(e^+) \in \Gamma_j^+$ for some $j \in I_\Gamma$ or $P_n(e^+) \in \chi_j^+$ for some $j \in I_\chi$, in case $c_j$ is a singular point. %) and ${\bf{T}}(P_n(e^+)) \in \chi_i^-$ (resp. 
As a result, ${\bf{T}}(P_n(e^+)) \in \Gamma_j^-$ or ${\bf{T}}(P_n(e^+)) \in \chi_j^-$, %) and so, $P_{n}^{-1}({\bf{T}}(P_n(e^+))) \in \chi_{i,n}^-$ (resp. 
hence $P_{n}^{-1}({\bf{T}}(P_n(e^+))) \in \Gamma_{i,n}^-$. 
%But, we also have thanks 
Due to \eqref{r-pro} we also notice that %the following:
$$f_n(\arc{e^+P_{n}^{-1}({\bf{T}}(P_n(e^+)))})=f(\arc{P_n(e^+){\bf{T}}(P_n(e^+))})=0.$$

%where the last equality follows from the definition of the transportation map ${\bf{T}}$ (see \eqref{r-pro}).
%which yields that $P_n^{[-1]}({\bf{T}}(P(e^+)))  \in \chi_{i,n}^-$ (resp. $\Gamma_{i,j,n}^-$).
%On the other hand, since we have
We can use (\ref{r-dfTn}) again to %then we can 
deduce that for all $x\in \partial\Omega_n$ we have
\begin{equation}\label{r-dsTTn}
|\bbT_n(x) - \bbT(P_n(x))| = |P^{-1}( \bbT(P_n(x))) -  \bbT(P_n(x))|\le d_H(\overline{\Omega}_n, \overline{\Omega})\to 0,\,\,\,\,\,\hbox{when }n\to \infty.
\end{equation}

Let $\{e_{i}^+\}_{1 \leq i \leq m}$ be any finite sequence of points such that $e_{i}^+ \in \chi_{i,n}^+ \cup \Gamma_{i,n}^+$. We recall that since $\Omega_n$ is strictly convex, then by Remark \ref{Remark 3.7}, it is sufficient to consider points belonging to different arcs $\Gamma_{i,n}^+$ or $\chi_{i,n}^+$. In this way, the index $m$ appearing in (H3) will be at most the number of arcs $\chi_i^\pm$ and $\Gamma_i^\pm$, which is finite and does not depend on $n$. Now, applying \eqref{r-dsTTn} we see that
\begin{align*}
\sum_{i=1}^{m} |e_i^+ - {\bf{T}}_n(e_i^+)|&\leq \sum_{i=1}^{m} |e_i^+ - P_n(e_i^+)|+|P_n(e_i^+) - {\bf{T}}(P_n(e_i^+))|+|{\bf{T}}(P_n(e_i^+)) - {\bf{T}}_n(e_i^+)|\\
&\leq \sum_{i=1}^{m} |P_n(e_i^+) - {\bf{T}}(P_n(e_i^+))|+2m\,\varepsilon_n,
\end{align*}
where $\varepsilon_n:=d_H(\overline{\Omega}_n, \overline{\Omega})$. %$ \to 0$ as $n\to \infty$. The last inequality follows from 
%\PR{(\ref{r-dfTn}), continuity of $\bbT$, $P_n$, $P_n^{-1}$ and} 
%the fact that for any $x\in \partial\Omega_n$, we have
%\begin{equation}\label{r-Hdist}
%|x-P_n(x)| = \dist(x, \overline{\Omega})\le d_H(\overline{\Omega}_n, \overline{\Omega}).
%\to 0,\qquad\hbox{ as } {n\to\infty},
%\end{equation}
%we deduce that $\lim_{n\to\infty}\varepsilon_n =0$.
%2\, \mathrm{d}_H(\partial\Omega_n,\partial\Omega)
%\leq \sum_{i=1}^{m_\star} |P(e_i^+) - {\bf{T}}(P(e_i^+))|+\frac{2\,m_\star}{n}.$$
%For
Moreover, we have the following inequality:
%$$
%\sum_{i=1}^{m_\star} |P(e_i^+) - {\bf{T}}(P(e_i^+))|<\,\,
\begin{align*}
    \sum_{i=1}^{m-1}
|P_n(e_i^+) - {\bf{T}}(P_n(e_{i+1}^+))| &+ |P_n(e_m^+ )-{\bf{T}}(P_n(e_{1}^+))|\\
&\leq  
\sum_{i=1}^{m-1}
\left(|P_n(e_i^+) - e_{i}^+| 
+ |e_i^+ - {\bf{T}}_n(e_{i+1}^+)| + |{\bf{T}}_n(e_{i+1}^+)-{\bf{T}}(P_n(e_{i+1}^+))|\right)
\\
 &\qquad\qquad+\,|P_n(e_m^+) - e_{m}^+| + |e_m^+ - {\bf{T}}_n(e_{1}^+)|+|{\bf{T}}_n(e_{1}^+)-{\bf{T}}(P_n(e_{1}^+))|\\
 &\leq  
\sum_{i=1}^{m-1}
 |e_i^+ - {\bf{T}}_n(e_{i+1}^+)|  + |e_m^+ - {\bf{T}}_n(e_{1}^+)| + 2m\,\varepsilon_n.
\end{align*}
%\begin{align*} %$$ $$
 %\end{align*} 
%Yet, we know 
Since by assumption $g$ satisfies (H3)$^{\,\prime}$,
%. that is %as a result, 
there exists a $\delta_m>0$ (independent of $n$) such that
 $$\sum_{i=1}^{m} |P_n(e_i^+) - {\bf{T}}(P_n(e_i^+))| \leq\,\,
\sum_{i=1}^{m-1}
|P_n(e_i^+) - {\bf{T}}(P_n(e_{i+1}^+))| + |P_n(e_m^+ )-{\bf{T}}(P_n(e_{1}^+))|-\delta_m.$$
These inequalities imply that $g_n$ satisfies the assumption (H3) because for $n$ large enough, we get the following inequality:
\begin{align*}
 \sum_{i=1}^{m} |e_i^+ - {\bf{T}}_n(e_i^+)| &\leq \sum_{i=1}^{m-1}
 |e_i^+ - {\bf{T}}_n(e_{i+1}^+)|  + |e_m^+ - {\bf{T}}_n(e_{1}^+)| -(\delta_m - 4m\,\varepsilon_n )  \\
 & <\sum_{i=1}^{m-1}
 |e_i^+ - {\bf{T}}_n(e_{i+1}^+)|  + |e_m^+ - {\bf{T}}_n(e_{1}^+)|.
\end{align*}
 
{\it Step 4.} For the sake of studying the convergence of the sequence of measures $f_n \in \mathcal{M}(\partial\Omega_n)$, we need first to extend them all to a common domain. For this purpose, we fix any $n_0 \in \mathbb{N}$. For $n\ge n_0$, we define a Borel measure $\tilde f_n$ on $\overline{\Omega}_{n_0}$ by setting
$\tilde f_n(A) := f_n(A \cap \partial\Omega_n)$, for every Borel set $A$. For the sake of a convenient notation, we will drop the tildes. Let $f_n^+$ and $f_n^-$ be the positive and negative parts of $f_n$. Hence, we claim that $f_n^\pm \rightharpoonup f^\pm$. Indeed, for any continuous function $\varphi$ on $\overline{\Omega}_{n_0}$, one has
$$
\lb f_n^\pm,\varphi\rb= \lb(P_{n}^{-1})_{\#}f^\pm,\varphi\rb= \int_{\partial\Omega} \varphi(P_{n}^{-1}(x))\,\mathrm{d}f^\pm(x) \rightarrow  \int_{\partial\Omega} \varphi(x)\,\mathrm{d}f^\pm(x)=\lb f^\pm,\varphi\rb$$
because due to (\ref{r-dsTTn}), $P_n^{-1}(x)$  converges to $x$, for all $x\in \Gamma \cup \chi$.

Let $\gamma_n$ be an optimal transport plan in Problem \eqref{Kantorovich} between $f_n^+$ and $f_n^-$, where $\overline{\Omega}_{n_0}$ plays now the role of $\overline{\Omega}$. By Lemma \ref{l-conv}, we infer that
up to a subsequence,
%It is clear that up to a subsequence,
$\gamma_n$ weakly converges to a measure $\gamma$ in  $\mathcal{M}^+(\overline{\Omega}_{n_0} \times\overline{\Omega}_{n_0} )$ with $(\Pi_x)_{\#}\gamma=f^+$, $(\Pi_y)_{\#}\gamma=f^-$. Moreover, $\gamma$
%Lemma \ref{l-conv} to deduce that
%up to a subsequence, $\gamma_n \rightharpoonup \gamma$, where $\gamma \in \mathcal{M}^+(\overline{\Omega} \times\overline{\Omega})$
is an optimal transportation plan between $f^+$ and $f^-$.

%Let $\phi_n$ be a Kantorovich potential associated to $\gamma_n$ such that $\phi_n(x_0)=0$, for a fixed point $x_0 \in \overline{\Omega}$. Then, $\{\phi_n\}_{n=1}^\infty$, up to a subsequence, converges uniformly to a function $\phi \in \Lip_1(\overline{\Omega}_{n_0})$. Yet, due to the duality $\min\eqref{Kantorovich}=\sup\eqref{dual}$, we have
%$$\int_{\overline{\Omega}_{n_0} \times \overline{\Omega}_{n_0}}|x-y|\,\mathrm{d}\gamma_n=\int_{\overline{\Omega}_{n_0}} \phi_n\,\mathrm{d}f_n.$$
%Passing to the limit when $n \to \infty$, we get that $\gamma$ is an optimal transport plan between $f^+$ and $f^-$. At the same time, $\phi$ is the corresponding Kantorovich potential, because
%$$
%\min\eqref{Kantorovich} 
%\leq \int_{\overline{\Omega}_{n_0} \times \overline{\Omega}_{n_0}}|x-y|\,\mathrm{d}\gamma
%=\int_{\overline{\Omega}_{n_0}} \phi\,\mathrm{d}f \leq \sup \eqref{dual}.$$

%{\it Step 5.}
Yet, due to Proposition \ref{strictly_convex_case}, we have that $\gamma_n=(Id,{\bf{T}}_n)_{\#}{f_n^+}$.
%So, it remains to 
We will show that $\gamma=(Id,{\bf{T}})_{\#}{f^+}$. Take $\xi \in C(\overline{\Omega} \times \overline{\Omega})$, then we have
$$\int_{\overline{\Omega}\times \overline{\Omega}} \xi(x,y)\,\mathrm{d}\gamma_n(x,y)=\int_{\overline{\Omega}} \xi(x,{\bf{T}}_n(x))\,\mathrm{d}f_n^+(x)=\int_{\overline{\Omega}} \xi(x,{\bf{T}}_n(P_n^{-1}(x)))\,\mathrm{d}f^+(x)$$
$$\rightarrow \int_{\overline{\Omega}} \xi(x,{\bf{T}}(x))\,\mathrm{d}f^+(x)$$\\
%We will investigate convergence of the transport maps $\bbT_n$. 
%We already know that $|P_n^{-1}(x)-x| \to 0$ for all $x\in \partial\Omega$ due to (\ref{r-Hdist}).
%For any $x \in \chi_i^+$ (resp. $x \in \Gamma_i^+$), we recall that $P_{n}^{-1}(x) \in \chi_{i,n}^+$ (resp. $P_{n}^{-1}(x) \in \Gamma_{i,n}^+$) and, we have $P_{n}^{-1}(x) \to x$. 
%Moreover,
%In fact, it is easy to see that
where the convergence of
${\bf{T}}_n(P_{n}^{-1}(x))$  to  ${\bf{T}}(x)$ follows from (\ref{r-dsTTn}). %, for all $x\in \Gamma^+ \,\cup\, \chi^+$. Indeed, this follows immediately from ${\bf{T}}_n(P_{n}^{-1}(x))=P_{n}^{-1}({\bf{T}}(x))$. 
%%At the same time, we have 
%%$$\phi_n(P_{n}^{-1}(x))-\phi_n({\bf{T}}_n(P_{n}^{-1}(x)))=|P_{n}^{-1}(x)-{\bf{T}}_n(P_{n}^{-1}(x))|,\,\,\,\,\,\mbox{for all}\,\,\,\,x \in \Gamma^+ \cup \chi^+.$$
%Passing to the limit when $n \to \infty$, we infer that $\phi$ and ${\bf{T}}$ satisfy together the equality \eqref{transport ray}, i.e. we have
%\begin{equation}\label{r-poK}
%\phi(x)-\phi({\bf{T}}(x))=|x-{\bf{T}}(x)|\,\,\,\,\,\mbox{for all}\,\,\,\,x \in \Gamma^+ \cup \chi^+.
%\end{equation}
%In other words, this means that $[x,{\bf{T}}(x)]$ is a transportation ray, for every $x \in \Gamma^+ \cup \chi^+$. Yet, we recall that ${\bf{T}}$ is a transport map (i.e., ${\bf{T}}_{\#}f^+=f^-$). Hence, we get
%$$\int_{\partial\Omega}|x-{\bf{T}}(x)|\,\mathrm{d}f^+(x)=
%\int_{\partial\Omega}[\phi(x)-\phi({\bf{T}}(x))]\,\mathrm{d}f^+(x)=\int_{\partial \Omega} \phi\,\mathrm{d}(f^+-f^-).$$
%Since $\phi$ is a Kantorovich potential, t
This implies that $\gamma=(Id,{\bf{T}})_{\#}f^+$. 
%is an optimal transport plan.
%or equivalently, that ${\bf{T}}$ is an optimal transport map. 
%\qedhere
\end{proof}
We will see in the next proposition that one can relax (H3)$^{\prime}$ in Proposition \ref{Prop. convex case with finite arcs} by showing that (H3) is in fact sufficient to get that ${\bf{T}}$ is an optimal transport map.
\begin{proposition}\label{NEW Prop}
Assume $\Omega$ is convex, (H1), (H2) $\&$ (H3) hold, \,$\mathcal{S}$ is finite and the number of arcs \,$\chi_i^\pm$ ($i \in I_\chi$)\, and \,$\Gamma_i^\pm$ ($i \in I_\Gamma$)
%such that $\chi_i^\pm$ (resp. $\Gamma_i^\pm$) is not strictly convex 
is finite. Then,  ${\bf{T}}$ is an optimal transport map from $f^+$ to $f^-$. 
\end{proposition}
\begin{proof}
   % Finally, it remains to relax (H3)$^\prime$ in Proposition \ref{Prop. convex case with finite arcs}. %show that in Proposition \ref{Prop. convex case with finite arcs}, one can also remove the assumption (H3)$^\prime$ 
%For this purpose, w
We will construct a sequence $g_n \in C(\partial\Omega) \cap BV(\partial\Omega)$ satisfying (H3)$^\prime$ and converging uniformly to $g$.
%by an approximation argument. To be more precise, if the boundary datum $g$ satisfies (H3) (but not (H3)$^\prime$) then we will approximate it by a sequence of trace functions $g_n$ on $\partial\Omega$ such that $g_n$ satisfies (H3)$^\prime$, for all $n$. Since the proof is quite similar to the previous one, so we will omit some details. 

We set $$t^\star:=\min_{i \in I_\Gamma \cup I_\chi}\{TV(g\res\chi_i^\pm),TV(g\res\Gamma_i^\pm)\}>0.$$ 
We fix $n \in \mathbb{N}^\star$ large enough so that $\frac{1}{n}<t^\star$.
%%We introduce index sets $I^n_\chi$, $I^n_\Gamma$ and $I^n_F$:
%\begin{align*}
For every 
%$$
%I^n_\chi= \{ 
$i\in I_\chi $, there exists $c_{i,n}^\pm \in \chi_i^\pm$ such that $|g(c_{i,n}^\pm)-g(c_i^\pm)|=\frac{1}{n}$. For every $ i \in I_\Gamma$, there are
$a_{i,n}^\pm,\,b_{i,n}^\pm \in \Gamma_i^\pm$ such that $g(a_{i,n}^\pm)=g(a_i^\pm)+\frac{1}{n}$ and $g(b_{i,n}^\pm)=g(b_i^\pm)-\frac{1}{n}$. Set $I^n_\chi=I_\chi$, $I^n_\Gamma=I_\Gamma$ and 
$
I^n_F = I_F \cup I_\chi
 \cup 
%(I_\Gamma\setminus I^n_\Gamma)\times\{\Gamma\}\cup 
%I_\chi \times \{F\}\cup
I_\Gamma. 
%\times \{F\}
$
%\end{align*}}
%For every $i \in I_\chi$, we pick two points $c_{i,n}^\pm \in \chi_i^\pm$ such that $|g(c_{i,n}^\pm)-g(c_i^\pm)|=\frac{1}{n}$. On the other hand, for every $i \in I_\Gamma$, we take four points $a_{i,n}^\pm,\,b_{i,n}^\pm \in \Gamma_i^\pm$ such that $g(a_{i,n}^\pm)=g(a_i^\pm)+\frac{1}{n}$ and $g(b_{i,n}^\pm)=g(b_i^\pm)-\frac{1}{n}$. 
So, we define the arcs
$$\chi_{i,n}^\pm=\arc{\,c_i\,c_{i,n}^\pm}\,\,(i \in I^n_\chi),\,\,\,\,\,\,\,\,\,\,\,\,\,\,\Gamma_{i,n}^\pm=\arc{a_{i,n}^\pm b_{i,n}^\pm}\,\,\,(i \in I^n_\Gamma),$$ 
%{We also introduce an additional index set
%$$
%I^n_F = I_F \cup I_\chi
 %\cup 
%(I_\Gamma\setminus I^n_\Gamma)\times\{\Gamma\}\cup 
%I_\chi \times \{F\}\cup
%I_\Gamma, 
%\times \{F\}
%$$
%}
and 
$$F_{i,n}=
\begin{cases}
%\left\{
%\begin{array}{ll}
 % \chi_i   & \,\,\,\,\,\,\hbox{if }\,\,\,\,\, i \in I_\chi\setminus I^n_\chi \\
 % \Gamma_i   & \,\,\,\,\,\,\hbox{if }\,\,\,\,\, i \in I_\Gamma\setminus I^n_\Gamma \\
  F_i & \,\,\,\,\,\,\hbox{if }\,\,\,\,\, i \in I_F\\
  \Gamma^\pm_i \setminus \Gamma_{i,n}^\pm & \,\,\,\,\,\,\hbox{if }\,\,\,\,\, i \in I^n_\Gamma\\
  \chi^\pm_i \setminus \chi_{i,n}^\pm & \,\,\,\,\,\,\hbox{if }\,\,\,\,\, i \in I^n_\chi
  \end{cases}.$$\\
%\end{array}
%\right $$\\ \\
%j \setminus \chi_{j,n}\,\,\,(\,j \in I_\chi),\,\,F_{i,n}=\Gamma_j \setminus \Gamma_{j,n}\,\,(\,j \in I_\Gamma)\,\,\,\,\,\mbox{or},\,\,\,F_{i,n}=F_j\,\,\,(\,j \in I_F).$$\\ Again, w

We define measure $f_n$ on $\partial\Omega$ as follows:
$$
f_n:= \sum_{i\in I_\chi^n}{f}\LC{\chi_{i,n}}+
 \sum_{i\in I_\Gamma^n} {f}\LC\Gamma_{i,n}.
$$
%It is clear that
Then, we have $f_n \rightharpoonup f$. Now, set
$$
g_n(x) = g(x_0) + f_n(\arc{x_0 x}).
$$\\
We construct the corresponding ${\bf{T}}_n$ as in \eqref{df-T}. 
Since $g$ satisfies (H3) and the arcs $\chi_{i,n}$ (resp. $\Gamma_{i,n}$) are compactly contained in the open arcs $\chi_i$ (resp. $\Gamma_i$) and ${\bf{T}}_n$  is continuous on the compact set 
%$\Gamma^+ \cup \chi^+$
$\overline{\Gamma_n^+} \cup \overline{\chi_n^+} \subset \Gamma^+ \cup \chi^+$
(see Lemma \ref{cont-T}), then 
we deduce that  (H3)$^\prime$ is satisfied with this choice of boundary datum 
$g_n$. {Thus, by Proposition \ref{Prop. convex case with finite arcs}, we get that $\bbT_n$ is an optimal transport map, i.e. $\gamma_n = (Id,\bbT_n)_\# f^+_n$ solves \eqref{Kantorovich} between $f_n^+$ and $f_n^-$. We notice that the construction of new arcs yields that for every $x\in \Gamma_n^+\cup \chi_n^+$, we have 
\begin{equation}\label{r-Tn}
\bbT_n(x)={\bf{T}}(x).
\end{equation}}
\begin{comment}We have to show that $(I, \bbT)_\# f^+ (U) = \gamma (U)$ for all measurable sets. Since the $\sigma$-field of $\gamma$-measureable sets is generated by $A\times B$, where $A,B\subset \partial \Omega$ are Borel sets it suffices to establish that
$$
(I, \bbT)_\# f^+ (A\times B) = \gamma (A\times B).
$$ 
Since the sequences of sets $\{\spt f_n^+\}_{n\in \bN}$ and $\{\spt f_n^-\}_{n\in \bN}$ are increasing, then we have
\begin{align*}
& (I, \bbT)_\# f^+ (A\times B)  = \lim_{n\to\infty} (I, \bbT)_\# f^+ ((A\times B)\cap (\spt f_n^+ \times \spt f_n^-)), \\
&\gamma (A\times B) = \lim_{n\to\infty}\gamma ((A\times B) \cap (\spt f_n^+ \times \spt f_n^-)).
\end{align*}
Hence, this observation and (\ref{r-Tn}) imply that
\begin{align*}
 (I, \bbT)_\# f^+ (A\times B)  &= \lim_{n\to\infty}(I, \bbT_n)_\# f_n^+ ((A\times B)\cap (\spt f_n^+ \times \spt f_n^-)) \\
 & =\lim_{n\to\infty}\gamma_n ((A\times B)\cap (\spt f_n^+ \times \spt f_n^-))  =\lim_{n\to\infty}\gamma_n ((A\times B)\cap (\spt f_n^+ \times \spt f_n^-))
 \\
 &=\lim_{n\to\infty}\gamma ((A\times B)\cap (\spt f_n^+ \times \spt f_n^-)) = \gamma (A\times B) .
\end{align*}
\end{comment}
%\begin{comment}\PR{ [HERE COMES SAMER'S TEXT]}
%I suggest the following.
%Moreover, if we take $x\in \Gamma^+\cup \chi^+$ then $x\in \Gamma_n^+\cup \chi_n^+$ for sufficiently large $n$.
Thus, one can show again that $\gamma_n \rightharpoonup \gamma=(Id,\bbT)_\# f^+$ and $\gamma$ is an optimal transport plan between $f^+$ and $f^-$. $\qedhere$
%letting $\phi_n$ be the Kantorovich potential corresponding to ${\bf T}_n$, we have $$\phi_n(x)-\phi_n({\bf T}_n(x))=|x-{\bf T_n}(x)|.$$ Then, passing with $n$  to infinity, the rest of the proof follows exactly as in Step 5 of Proposition \ref{Prop. convex case with finite arcs}. $\qedhere$  %The details are the same as in Step 5 of P
%\end{comment}
%$\overline{\Gamma}_n^+ \cup \overline{\chi}_n^+$ 
%which is a compact set of $\Gamma^+ \cup \chi^+$ with ${\bf{T}_n}={\bf{T}}$, for all $n \in \mathbb{N}$.
%The rest of the proof follows exactly as in the previous approximation.
%$\cup_{i \in I_\chi \cup I_\Gamma}  \overline{\Gamma_{i,n}} \cup \overline{\chi_{i,n}} \subset %\cup_{i \in I_\chi \cup I_\Gamma}  \overline{\Gamma_{i,n}} \cup {\chi_{i}}$
\end{proof}

\bigskip
Now, we take another approximation in order to cover the case of convex domain with infinite number of arcs $\chi_i^\pm$ ($i \in I_\chi$) and $\Gamma_i^\pm$ ($i \in I_\Gamma$) and, possibly an infinite number of singular points.
%(but, we always assume that $\mathcal{S}$ is at most countable and not dense anywhere).
%At the same time, we will see that one can relax the assumption (H3)$^{\prime}$ by showing that (H3) is in fact a sufficient condition to get that ${\bf{T}}$ is an optimal transport map.
%Let us denote by $S$ the set of singularity points on $\spt(f)$. 
%In the sequel, we will assume that the set of singular points $\mathcal{S}$ is at most countable, but not dense anywhere.  
%need %say that the (H) holds if we have 
%the following  assumption:
%$$
%S\,\,\,\mbox{is at most countable, but not dense anywhere}.\leqno({\rm H4})
%$$
%Set $S=\bigcup_i S_i$, where $S_i=\{s_{i,k}^\pm \in \chi_i^\pm \,\,(\mbox{resp.}\,\,\,s_{i,k}^\pm \in \Gamma_i^\pm):1 \leq k \leq n_i^\pm\}$, $n_i^\pm \in \mathbb{N}\cup\{+\infty\}$. Then, w
More precisely, we have the following:
\begin{proposition}\label{Prop. convex case}
Assume that \,$\Omega$ is convex and conditions (H1), (H2), (H3) and {(S)} hold. Then, ${\bf{T}}$ is an optimal transport map from $f^+$ to $f^-$. 
\end{proposition}
\begin{proof}
In the present case, we have a priori an infinite number of arcs $\chi_i^\pm$ and $\Gamma_i^\pm$ with an infinite number of singular points. However, due to Lemma \ref{l-repa}, we may assume after a repartition of $\partial \Omega$, that all the arcs $\chi_i^\pm$ and $\Gamma_i^\pm$ are smooth.  
%Initially we may %Let us 
%assume that (H3)$^{\prime}$ holds. Later, we will relax this additional condition.
%in Proposition \ref{Prop. convex case with finite arcs}. %(we will show at the end of this proof that we can relax this condition ). 
We then proceed again by approximation, but this time we approximate
the trace functions $g$ by
$g_n$ defined on %the same boundary 
$\partial\Omega$ and %converging to $g$ 
such that $g_n$ has a finite number of smooth arcs $\chi_{i,n}^\pm$ and $\Gamma_{i,n}^\pm$.
%having a finite number of smooth arcs $\chi_{i,n}^\pm$ and $\Gamma_{i,n}^\pm$. % and a finite number of singularity points on $\chi_i$ and $\Gamma_i$. 
%
%We assume that there is Similarly to Proposition \ref{Prop. convex case with finite arcs}, one can without loss of generality 
%
%We set again $L_{i,k}^+:=[s_{i,k}^+,{\bf{T}}(s_{i,k}^+)]$ and  $L_{i,k}^-:=[s_{i,k}^-,{\bf{T}}^{[-1]}(s_{i,k}^-)]$, for all $1 \leq k \leq n_i^\pm$. Since $S_i$ is countable but it is not dense in $\chi_i^\pm$ (resp. $\Gamma_i^\pm$), then we see that this family of line segments $\{L_{i,k}^+,\,L_{i,k}^-\}$ decomposes $\chi_i^+$ and $\chi_i^-$ (resp. $\Gamma_i^+$ and $\Gamma_i^-$) into arcs of the forms $\Gamma_{i,k}^+$ and $\Gamma_{i,k}^-$ such that $f^+(\Gamma_{i,k}^+)=f^-(\Gamma_{i,k}^-)$.
    %(again if $c_i \notin S_i$ then there will be also an arc of the form $\chi_{i,0} \subset \chi_i$).
 %   Hence, by repartitioning the arcs
    %By increasing the number of arcs 
  %  $\chi_i^\pm$ and $\Gamma_i^\pm$, 
    %(let us denote by $m$ the number of these arcs $\chi_i$ and $\Gamma_i$), 
   % we assume that $\chi_i$ (resp. $\Gamma_i$) does not contain any singularity point, for all $i$.
%For this purpose, we will construct a sequence of trace functions $g_n$ on $\partial\Omega$ converging to $g$ such that $g_n$ has a finite number of arcs $\chi_{i,n}^\pm$ and $\Gamma_{i,n}^\pm$.
  %and a finite number of singularity points on $\chi_{i,n}^\pm$ and $\Gamma_{i,n}^\pm$. 
Namely, for a fixed $n \in \mathbb{N}^\star$, we set
$$
I^n_\chi = \bigg\{i\in I_\chi: \mathcal{H}^1(\chi_i)\geq \frac{1}{n}\bigg\},\qquad I^n_\Gamma = \bigg\{i\in I_\Gamma: \mathcal{H}^1(\Gamma_i^+ \cup \Gamma_i^-)\geq \frac{1}{n}\bigg\}
$$
and
$$
\chi_{i,n}=\chi_i \,\,\,(\mbox{for all}\,\,\,i\in I^n_\chi),\,\,\,\,\qquad \Gamma_{i,n}^\pm=\Gamma_i^\pm\,\,\,\,(\mbox{for all}\,\,\,i\in I^n_\Gamma),$$
$$F_{i,n}=\chi_j\,\,(\mbox{for some}\,\,\,j\in 
I_\chi\setminus I^n_\chi),\,\,\ F_{i,n}=\Gamma_j\,\,\,(\mbox{for some}\,\,\,j\in I_\Gamma \setminus I^n_\Gamma)\,\,\,\mbox{or}\,\,\,F_{i,n}=F_j\,\,(\mbox{for some}\,\,\,j \in I_F).
$$
%If  set  and $f_n=f$ on $\chi_{i,n}$. If $\mathcal{H}^1(\chi_i)<\frac{1}{n}$, set and $f_n=0$ on $F_{i,n}$. Similarly, if \,$\mathcal{H}^1\geq \frac{1}{n}$ then we set and $f_n=f$ on $\Gamma_{i,n}^\pm$, while if \,$\mathcal{H}^1(\Gamma_i^+ \cup \Gamma_i^-)<\frac{1}{n}$, then we set $F_{i,n}=\Gamma_{i,n}^\pm$ and $f_n=0$ on $F_{i,n}$. 
Now, we define  measure $f_n$ on $\partial\Omega$ as follows:
$$
f_n:= \sum_{i\in I^n_\chi}{f}\LC{\chi_i}+
 \sum_{i\in I^n_\Gamma} {f}\LC\Gamma_{i}.
$$

We shall   see that $f_n \rightharpoonup f$. Indeed, for any continuous function $\varphi$ on $\partial\Omega$, we have that\\
%\begin{align*}
 $$\lb f_n,\varphi\rb=
\sum_{i\in I^n_\chi}
%{i\,:\,\mathcal{H}^1(\chi_i),\,\mathcal{H}^1(\Gamma_i^+ \cup \Gamma_i^-) \geq \frac{1}{n}} 
\lb {f \LC \chi_i},\varphi\rb +
\sum_{i\in I^n_\Gamma}
\lb {f\LC \Gamma_{i}},\varphi\rb  \,\longrightarrow \sum_{i\in I_\chi} \lb {f\LC\chi_i},\varphi\rb +
\sum_{i\in I_\Gamma}
\lb {f\LC\Gamma_{i}},\varphi\rb =\lb f,\varphi\rb .$$
%\end{align*}
We note that $f_n(\partial\Omega)=0$, for all $n \in \mathbb{N}^\star$. Fix $x_0 \in \partial\Omega$, so we define the trace function $g_n$ on $\partial\Omega$ as follows:
$$g_n(x):=f_n(\arc{x_0\,x})$$
where $\arc{x_0\,x}$ is the arc from $x_0$ to $x$ going in the positive orientation. So, we clearly have $g_n \in C(\partial\Omega) \cap  BV(\partial\Omega)$ with $\partial_\tau g_n=f_n$. Let ${\bbT_n}$ be the transport map from $f_n^+$ to $f_n^-$ defined as in \eqref{df-T}. It is convenient to extend it %${\bbT_n}$ 
to $\Gamma^+ \cup \chi^+$ by setting ${\bbT_n}(x)=x$, for every $x \in (\bigcup_{i\in I_\chi\setminus I^n_\chi}\chi^+_i)\cup(\bigcup_{i\in I_\Gamma\setminus I^n_\Gamma}\Gamma^+_i)$.
%, for every $x \in \spt(f^+) \backslash\spt(f_n^+)$. 
It is obvious that ${\bbT_n}(x) = {\bf{T}}(x)$, for every $x \in \Gamma_n^+ \cup \chi_n^+$.
%and $g_n(x) \rightarrow g(x)$, for all $x \in \partial\Omega$.
It is also easy to see that the assumption (H3) is satisfied by $g_n$, because for every $i \in I^n_\chi$ (resp. $i \in I^n_\Gamma$), we have $\chi_{i,n}=\chi_i$ (resp. $\Gamma_{i,n}=\Gamma_i$). In particular, Proposition \ref{NEW Prop} yields that ${\bf T}_n$ is an optimal transport map from $f_n^+$ to $f_n^-$.  %, while ${\bbT_n}={\bf{T}}$ on $\spt(f_n^+)$.

Let $\gamma_n=(Id,{\bf T}_n)_{\#}f_n^+$ be the optimal transport plan between $f_n^+$ and $f_n^-$. Now, we invoke Lemma \ref{l-conv} to deduce that
up to a subsequence, $\gamma_n \rightharpoonup \gamma$, where $\gamma \in \mathcal{M}^+(\overline{\Omega} \times\overline{\Omega})$ is an optimal transportation plan between $f^+$ and $f^-$. We notice that for any $\xi \in C(\overline{\Omega}\times \overline{\Omega})$, we have
$$\int_{\overline{\Omega}\times \overline{\Omega}} \xi(x,y)\,\mathrm{d}\gamma_n(x,y)=\int_{\overline{\Omega}} \xi(x,{\bf{T}}_n(x))\,\mathrm{d}f_n^+(x)=\int_{\overline{\Omega}} \xi(x,{\bf{T}}(x))\,\mathrm{d}f_n^+(x)\rightarrow \int_{\overline{\Omega}} \xi(x,{\bf{T}}(x))\,\mathrm{d}f^+(x).$$

Hence, $\gamma=(Id,{\bf T})_{\#}f^+$. $\qedhere$
\end{proof}
\begin{remark}%\marginnote{If we agree on\\ (S), this should\\ be re-written\\ Come Back Later}
%\PR{
We have seen that the projection map $P$ is easy to visualize geometrically, but the 
set $\mathcal{S}$ leads to technical difficulties when we deal with $P$. However, we could introduce a more complicated map that does not take into account the presence of singular points on $\partial\Omega$ and, we will do so later in the proof of Proposition \ref{General case}. 
%This new map does not take into account the presence of singular points on $\partial\Omega$.
%, see the proof of Proposition \ref{General case}.}

%\sout{In fact, the assumption that the set of singular points \,$\mathcal{S}$\, \PR{can be further relaxed}
%is at most countable but not dense can be removed 
%if we replace the projection map in our approximation (see the proof of Proposition \ref{Prop. convex case with finite arcs}) by another appropriate map that does not take into account the presence of singular points on $\partial\Omega$, \PR{see the proof of Proposition \ref{General case}}. %Anyway, w
 %   We will introduce and use this suitable map in the next section (the non-convex case). Since the projection map is more understandable, we found better to use it here (in the convex case), in particular that it covers a large class of convex domains except those with dense sets of singularity points on the boundary (but this type of domains is not really crucial).}
%  since the projection map is more understandable then we kept using it here (in the convex case), in particular that it covers a large class of convex domains except those with dense sets of singularity points on the boundary (but these types of domains are not very useful).
\end{remark}
From Proposition \ref{Prop. convex case}, we know that $\gamma=(Id,{\bf{T}})_{\#}f^+$ is an optimal transport plan  between $f^+$ and $f^-$. In fact, one can show that this is a unique optimal transport plan in \eqref{Kantorovich}. More precisely, we have the following: %\todo{comments}
\begin{proposition}  \label{Uniqueness in the convex case} 
Under the assumptions that \,$\Omega$ is convex and (H1), (H2), (H3)  and (S) hold,
%$\&$ (H4) hold. %and 
%Suppose that $\Omega^+$ is  strictly convex, $f^+$ is non-atomic and, $f^+$ and $f^-$ have no common mass. 
%Moreover, assume that (H1), (H2), (H3) $\&$ (H4) hold.
Problem \eqref{Kantorovich} has a unique optimal transport plan \,$\gamma=(Id,{\bf{T}})_{\#}f^+$.
%, between $f^+$ and $f^-$, which will be induced by a transport map $S$,
%provided that $f^+$ is nonatomic.
\end{proposition}
\begin{proof}
Let \,$\gamma$\, be an optimal transport plan in Problem \eqref{Kantorovich} and let us assume that $\gamma \neq (Id,{\bf{T}})_{\#}f^+$, i.e. the set
$$
A = \bigg\{ x\in \partial \Omega:\, \exists\,\, y\in \partial \Omega,\ y\neq \bbT(x),\, (x,y) \in \spt(\gamma)\bigg\}
$$
has a positive measure, i.e., $f^+(A)>0$.
At the same time,
%Hence, there will be a set $A\subset\partial\Omega$ with $f^+(A)>0$ and such that for $\gamma-$a.e. $(x,y) \in A \times \partial\Omega$, we have $y \neq {\bf{T}}(x)$. 
%or every $i \in \mathcal{I}_\chi$ (resp. $i \in \mathcal{I}_\Gamma$), there will be a set $A_i \subset \chi_i^+$ (resp.  $A_i \subset \Gamma_i^+$) such that $f^+(A_i)>0$ and for $\gamma-$a.e. $(x,y) \in A_i \times \partial\Omega$, we have $y \neq {\bf{T}}(x)$. 
by Proposition \ref{Prop. convex case}, $]x,{\bf{T}}(x)[ \subset \Omega$ is a transport ray for $f^+-$a.e. $x$.  %and thanks to the fact 
We know that two transport rays cannot intersect at an interior point of one of them. As a result, %[\PR{DO WE REALLY USE THE CONVEXITY??}]
%one of them, then 
we get that $A \subset \mathcal{N}$. Yet, thanks to the convexity of $\Omega$ (see Lemma \ref{l-neg-N}), $\mathcal{N}$ is at most countable. 
%must be contained in the \PR{complement } \sout{set of endpoints} of $\Gamma^+ \cup \chi^+\PR{\cup \Gamma^-\cup \chi^-\cup  F}$. In particular, $A\PR{\subset N}$, \PR{hence due to (H3) part (5)} 
%is at most countable. Yet, $f^+$ is atomless and so, 
Hence, this yields that $f^+(A)=0$, which is a contradiction. 
%Thanks to Proposition \ref{Prop. nonconvex case}, we see that if \,$\gamma$\, is an optimal transport plan then 
We conclude %This implies 
that $\gamma=(Id,{\bf{T}})_{\#}f^+$
%, where ${\bf{T}}(x)=\bbT^i(x)$ for $f^+-$a.e. $x \in \partial\Omega_i$. But, this implies that $\gamma$
is the unique optimal transport plan in Problem \eqref{Kantorovich}.
%since, if $\gamma^\prime$ is another optimal transport plan then \,$\gamma^{\prime\prime}=(\gamma + \gamma^\prime)/2$\, is also optimal, which is not possible as $\gamma^{\prime\prime}$ must be induced by a transport map.  
%On the other hand, we recall that, under the  (H1), (H2), (H3) \& (H4), all the transport rays lie inside $\Omega$ 
\end{proof}
In the sequel, we set $\gamma:=(Id,{\bf{T}})_{\#}f^+$.
%By Proposition \ref{Prop. convex case}, $\gamma$ is an optimal transport plan.
Thanks to the convexity of $\Omega$, the vector measure $v_\gamma$ given by \eqref{flow} (resp. the transport density $\sigma_\gamma$ \eqref{Transport density definition}) is well defined. Moreover, due to condition (H2), one has
\begin{equation}\label{r-pr3.15}
|v_\gamma|(\partial\Omega)=\sigma_\gamma(\partial\Omega)=\int_{\partial\Omega }\mathcal{H}^1([x,{\bf{T}}(x)] \cap\partial\Omega)\,\mathrm{d}f^+(x)=0.
\end{equation}

Hence, we get the following:
\begin{proposition} \label{existence & uniqueness of an optimal flow 0}
Assume that \,$\Omega$ is convex and (H1), (H2), (H3) $\&$ (S) hold. 
%Let $\gamma:=(Id,{\bf{T}})_{\#}f^+$ be an optimal transport plan and $v_\gamma$ be the flow defined by \eqref{flow}. 
Then, 
%we have \,$\min\eqref{Beckmann}=\min\eqref{Kantorovich}$. Moreover,
$v_\gamma$ is the unique solution for Problem \eqref{Beckmann}. 
\end{proposition}
\begin{proof}
 Thanks to \cite[Chapter 4]{Santambrogio} and the fact that $\Omega$ is assumed to be convex, we have that \,$\eqref{Beckmann}= \eqref{Kantorovich}$ %%%Thanks to (H2),
%Proposition \ref{Prop. nonconvex case},
%%the flow $v_\gamma$ defined in \eqref{flow} is well defined. 
%In addition, $v_\gamma$ is clearly admissible in Problem \eqref{Beckmann} and, we have
%$$\int_{\bar{\Omega}}|v_\gamma| =\int_{\bar{\Omega} \times \bar{\Omega}}|x-y|\mathrm{d}\gamma=\min\eqref{Kantorovich}=\sup\eqref{dual}\leq \min\eqref{Beckmann}.$$
%Hence,
and,
$v_\gamma$ is a minimizer for Problem \eqref{Beckmann}. Moreover, we recall that
%Thanks to \cite[Chapter 4]{Santambrogio} and the fact that $\min\eqref{Beckmann}=\min\eqref{Kantorovich}$, one can show 
if \,$v$\, is another minimizer in \eqref{Beckmann}, then there will be an optimal transport plan $\gamma^\prime$ in \eqref{Kantorovich} such that $v=v_{\gamma^\prime}$ (we refer the reader to \cite[Theorem 4.13]{Santambrogio} for more details). Yet, by Proposition \ref{Uniqueness in the convex case}, we know that the optimal transport plan $\gamma$ is unique. Hence, Problem \eqref{Beckmann} has a unique solution as well. 
%Moreover, thanks again to Proposition \ref{Prop. nonconvex case}, we see that 
\end{proof}
Finally, we are ready to state the first of our main theorems in this section:
\begin{theorem} \label{main theorem convex case}
Assume that \,$\Omega$ is convex, $g \in BV(\partial\Omega) \cap C(\partial\Omega)$ and conditions (H1), (H2), (H3), and (S) are satisfied. Then, the
least gradient problem \eqref{LGP} has a unique solution.
\end{theorem}
\begin{proof}
Due to Proposition \ref{t-equiv}, Problems \eqref{LGP} and \eqref{Beckmann} are  equivalent. As a result, our claim follows immediately from Proposition \ref{existence & uniqueness of an optimal flow 0} and the fact that $|v_\gamma|(\partial\Omega)=0$, see (\ref{r-pr3.15}).  
%\end{proof}
%\begin{proposition} \label{uniqueness of the optimal flow}
%Under the assumptions (C1), (C2) $\&$ (C3), the solution of Problem \eqref{Beckmann} is unique.
%when $f^+$ or $f^-$ is nonatomic.
%\end{proposition}
%\begin{proof}
%\end{proof}
%\begin{theorem}
%Under the assumptions (C1), (C2) $\&$ (C3), the BV least gradient problem \eqref{LGP} has a unique solution $u$  provided that $g \in BV(\partial\Omega) \cap C(\partial\Omega)$.
%\end{theorem}
%\begin{proof}
%This follows immediately from \cite[Lemma 3.4]{DweGor} and Proposition \ref{uniqueness of the optimal flow}.
%, and the fact that $f$ is nonatomic since $g \in C(\partial\Omega)$. 
\end{proof}
Now, we present an example showing that violation of (H3) might lead to the non-existence of solutions to Problem \eqref{LGP}.
\begin{example} \label{Example1}\rm 
Let \,$\Omega:=(0,1)^2$ be the unit square and define a function $g\in BV(\partial\Omega) \cap C(\partial\Omega)$ as follows
$$g(x_1,x_2):=\begin{cases}
\min\{x_1,\,\delta,\,1-x_1\} & \,\,\mbox{if} \,\,\,\,\,\,\, x_1 
\in[0,1], %\leq \delta\,\,\,\mbox{or}\,\,\,x_1 \geq 1-\delta,
\,\,\,x_2=0\,\,\,\,\mbox{or}\,\,\,\,x_2=1,\\
\min\{x_2,\,\delta,\,1-x_2\} & \,\,\mbox{if} \,\,\,\,\,\,\, x_2\in[0,1], 
%\leq \delta\,\,\,\mbox{or}\,\,\,x_2 \geq 1-\delta,
\,\,\,x_1=0\,\,\,\,\mbox{or}\,\,\,\,x_1=1,\\
%   x_2 & \mbox{if} \,\,\,\,\,\,\,0 \leq x_1 \leq \delta,\,\,x_2=0\,\,\,\mbox{or}\,\,\,x_2=1,\\
%   \delta & \,\,\mbox{if} \,\,\,\,\,\,\,\delta \leq \,x_1 \, \leq 1-\delta \,\,\,\,\mbox{or}\,\,\,\,
 %  \delta \leq \,x_2 \, \leq 1-\delta,
   %\\
   %-\max\{x_1,x_2\} +1 & \mbox{if} \,\,\,\,\,\,\,1-\delta \leq \max\{x_1,x_2\}  \leq 1,\\
\end{cases}$$
where 
$$\delta \in \bigg]\frac{1}{2+\sqrt{2}},\frac{1}{2}\bigg[.$$
\begin{figure}[h]
\begin{tikzpicture}
  % Draw the transp
  \draw (0, 0) rectangle (4, 4);
% Drawing the points
  \node at (1, -.3) {$\delta$};
  \node at (2.5, -.3) {$1-\delta$};
%   \fill (1,0) circle (2pt);
 %   \fill (0,1) circle (2pt);
  %    \fill (3,0) circle (2pt);
   %     \fill (0,3) circle (2pt);
    %      \fill (4,1) circle (2pt);
     %       \fill (4,3) circle (2pt);
      %       \fill (1,4) circle (2pt);
       %        \fill (3,4) circle (2pt);
% Drawing the transport rays           
    \draw[red,line width=2pt] (1,0)--(0,1);
    \draw[red,line width=2pt] (3,0)--(4,1);
    \draw[red,line width=2pt] (0,3)--(1,4);
    \draw[red,line width=2pt] (4,3)--(3,4);
% Inserting the Xi 
\node at (0.4, -.3) {$\chi_1^+$};
\node at (3.5, -.3) {$\chi_2^-$};
\node at (4.4, 3.6) {$\chi_3^-$};
\node at (4.4, 0.5) {$\chi_2^+$};
\node at (0.4, 4.3) {$\chi_4^-$};
\node at (3.5, 4.3) {$\chi_3^+$};
\node at (-0.4, 3.6) {$\chi_4^+$};
\node at (-0.4, 0.5) {$\chi_1^-$};

   % \coordinate (Pt1) at (0, 1);
  %  \fill (Pt1) circle (1pt)
   % \coordinate (Pt2) at (0, 3);
    %\fill (Pt2) circle (1pt)
    %\coordinate (Pt3) at (1, 0);
    %\fill (Pt3) circle (1pt)
    %\coordinate (Pt4) at (3, 0);
    % \fill (Pt4) circle (1pt)
      
  % Draw the points on each side
  %\fill (0.5, \delta) circle (2pt);
  %\fill (0.5, 1-\delta) circle (2pt);
  %\fill (\delta, 0.5) circle (2pt);
  %\fill (1-\delta, 0.5) circle (2pt);
\end{tikzpicture}
\caption{}
\label{Fig. 2}
\end{figure}

We notice that in this case, the boundary of \,$\Omega$ can be divided into arcs $\chi_i^{\pm}$ with $i=1,...,4$ (as shown in Figure \ref{Fig. 1}):
\begin{align*}
& \chi_1^+:=\{(x_1,0):0 \leq x_1\leq \delta\},& \chi_1^- :=\{(0,x_2):0\leq x_2\leq \delta\},   \\
& \chi_2^+:=\{(1,x_2): 0 \leq x_2\leq \delta\},&
\chi_2^-:=\{(x_1,0):1-\delta \leq x_1\leq 1\},\\
&\chi_3^+:=\{(x_1,1):1-\delta \leq x_1\leq 1\},&
\chi_3^-:=\{(1,x_2):1-\delta \leq x_2\leq 1\}, \\&
\chi_4^+:=\{(0,x_2):1-\delta \leq x_2\leq 1\}, &
\chi_4^-:=\{(x_1,1):0 \leq x_1\leq \delta\}.
\end{align*}
{Due to the definition of $g$}, the measures $f^+$, $f^-$ are given by
$f^+= \cH^1 \LC (\chi_1^+ \cup \chi_2^+ \cup \chi_3^+ \cup \chi_4^+)$ and $f^-=\cH^1 \LC(\chi_1^- \cup \chi_2^-\cup \chi_3^- \cup \chi_4^-)$.

Now, we claim that there are transport rays between \,$f^+$ and \,$f^-$ which are contained in the boundary of $\Omega$. Indeed, if this was not the case then it would not be difficult to see that
%either
the line segments %$[(\delta,0),(1,1-\delta)]$ and $[(0,\delta),(1-\delta,1)]$ are two transport rays or 
$[(\delta,0),(0,\delta)]$, $[(1-\delta,0),(1,\delta)]$, $[(1,1-\delta),(1-\delta,1)]$ and \,$[(\delta,1),(0,1-\delta)]$ must be transport rays. Hence, we should have the following inequality:
\begin{align*}
    |(\delta,0)-(0,\delta)|&+|(1-\delta,0)-(1,\delta)|+|(1,1-\delta)-(1-\delta,1)|+|(\delta,1)-(0,1-\delta)|\\
&\leq |(\delta,0)-(1-\delta,0)|+|(1,\delta)-(1,1-\delta)|+|(1-\delta,1)-(\delta,1)|+|(0,1-\delta)-(0,\delta)|,
\end{align*}
which is not possible since $\delta>\frac{1}{2+\sqrt{2}}$. In particular, the assumption (H3) is not satisfied and, a solution $u$ does not exist. On the other hand, we see that (H3) is satisfied as soon as $\delta \leq \frac{1}{2+\sqrt{2}}$. Thus, 
%under this condition on $\delta$, 
we infer that Problem \eqref{LGP} has a unique solution, provided that this inequality holds.

%Notice that if \,$\delta<\frac{1}{1+\sqrt{2}}$, then we must have
%$$d_M(\chi_i^+,\chi_i^-) +d_M(\chi_j^+,\chi_j^-) <\mbox{dist}(\chi_i^+,\chi_j^-)+\mbox{dist}(\chi_j^+,\chi_i^-),\,\,\mbox{for all}\,\,\,i \neq j.$$
%This shows that this inequality is not enough to prove existence of a solution to Problem \eqref{LGP}.
%$$3 \delta\sqrt{2}<2(1-2\delta)+(1-\delta)\sqrt{2}$$
%$$\frac{2+\sqrt{2}}{4+4\sqrt{2}}$$
%$$\sqrt{2}<2(1-2\delta)$$
%$$2 \delta \sqrt{2} \leq \min\{2-2 \delta,2(1-\delta)\sqrt{2}\}$$
%Yet, it is easy to see that the first case may hold if and only if $\delta \geq \frac{1}{2}$. 
\end{example}

From Theorem \ref{main theorem convex case}, we have already seen that if (H1), (H2), (H3) and {(S)} are satisfied, then Problem \eqref{LGP} attains a minimum. In other words, these conditions are sufficient to get existence of a solution to Problem \eqref{LGP}.

In the
%The 
following example the data violates
%shows that 
condition (H1) but  \eqref{LGP} has
a solution, hence (H1) is not necessary for existence.  
%However, we depend on it in our construction. We could deduce it for example if $g$ were. \AS{We provide an example for existence of solutions with boundary data that  (H1)-5 .}

\begin{example}\label{Ex:Cantor}\rm 
 Let $\Omega=[0,1]^2$ and $C:[0,1] \mapsto [0,1]$ be the Cantor function. Then, we define the boundary datum $g$ on $\partial\Omega$ as follows:
 $$g(x_1,x_2)=\begin{cases}
   C(x_1) & \qquad\mbox{if}\qquad\,0 \leq x_1 \leq 1\,\,\,\,\mbox{and}\,\,\,\,x_2=0\,\,\,\,\mbox{or}\,\,\,\,x_2=1,\\  
   0 & \qquad\mbox{if}\qquad\,0 \leq x_2 \leq 1\,\,\,\,\mbox{and}\,\,\,\,x_1=0,\\
   1 & \qquad\mbox{if}\qquad\,0 \leq x_2 \leq 1\,\,\,\,\mbox{and}\,\,\,\,x_1=1.
 \end{cases}$$    
 With this choice of \,$g$, it is clear that the least gradient problem \eqref{LGP} admits a solution $u(x_1,x_2)=C(x_1)$. Indeed, let us define the transport map
 $$T(x_1,0)=(x_1,1).$$
 The Kantorovitch potential is given by $\phi(x_1,x_2)=-x_2,$
 because
$$
1=\int_\Omega \phi \,\mathrm{d}(f^+-f^-)\le \sup_{\Lip(\psi)\le 1}\int_\Omega \psi \,\mathrm{d}(f^+-f^-)=
\int_0^1  (\psi(x_1,0) - \psi(x_1,1)) \, \mathrm{d}f^+
\le %\int_0^1  1 \mathrm{d}f^+ =1
1.
$$
Moreover,
$$
\int_\Omega \phi\, \mathrm{d}(f^+-f^-) =\int_\Omega |x-T(x)|\,\mathrm{d} f^+.
$$
\\
Notice that for such a function $g$, (H1) is not satisfied.
\end{example}
%\PR{We assumed that $g\in W^{1,1}(\partial \Omega)$ in this theorem to guarantee that the set $$
%N   = \partial \Omega\setminus(\bigcup_{i\in I_\Gamma} (\Gamma_i^+\cup\Gamma^-)\cup \bigcup_{i\in I_\chi} (\chi^+_i\cup \chi^-)\cup \bigcup_{i\in I_F}F_i$$ is negligible i.e. to  establish (H1) part (5).} \marginnote{no proof!}

In the example above, sets $I_\chi$ and $I_\Gamma$ are empty; but $g$ is monotone on each side of the square and $\partial_\tau g$ is a singular measure with $\int_{\partial \Omega} |\partial_\tau g| =2$. These observations suggest the following definition.
\begin{definition}\label{d-ps-mono}
%We say that a function $g\in BV(\partial\Omega)$ is {\it piecewise monotone} provided  that $\partial\Omega$ may be decomposed into disjoint sets such that $\partial \Omega = \bigcup_{i\in \cI_m} I_i \cup \bigcup_{i\in \cI_f} J_i \cup N$, where $I_i$ and $J_i$ are open intervals and $N = \bigcup_{i\in \cI_m} \partial I_i \cup \bigcup_{i\in \cI_f} \partial J_i$. Moreover, $g$ is strictly monotone on each $I_i$ and it is constant on each $J_i$.

 We say that a function $g\in W^{1,1}(\partial\Omega)$ is {\it piecewise monotone} provided  that $\partial\Omega$ may be decomposed into disjoint sets such that $\partial \Omega = U^+ \cup U^- \cup U_0$, where $U^\pm$ are open and $\partial_\tau g>0$ a.e. on $U^+,$ $\partial_\tau g<0$ a.e. on $U^-$ and $\int_{U_0} |\partial_\tau g| \, ds =0.$
\end{definition}
If we keep this definition in mind, then one can show that (H1), (H2), and (H3) are all at the same time necessary conditions 
%Theorem \ref{th-necessity} 
%one can provide necessary conditions 
for the solvability of the least gradient problem for piecewise monotone data. This is the content of the second of our main theorems:
%But, one can also show that (H1), (H2), and (H3) are all at the same time necessary conditions for the existence of such a solution in \eqref{LGP}. %More precisely, we have the following.
%\marginnote{Not clear if $W^{1,1}$ is enough}
\begin{theorem}\label{th-necessity}
Let \,$\Omega$ be a convex domain and \,$g$ in $W^{1,1}(\partial\Omega)$ be piecewise monotone. Assume that Problem \eqref{LGP} has a solution $u$. Then, the boundary datum $g$ must satisfy  conditions (H1), (H2) and (H3).
\end{theorem}
\begin{proof} {\it Step 1.}
First, we define $v:=R_{\frac{\pi}{2}} Du$. Due  to Proposition \ref{t-equiv}  the  vector field $v$ turns out to be a solution for problem \eqref{Beckmann}. At the same time, by \cite[Theorem 4.13]{Santambrogio}, there will be an optimal transport plan $\gamma$ for Problem \eqref{Kantorovich} such that $v=v_\gamma$. Since $u$ is a solution to Problem \eqref{LGP}, then no level set $\partial\{u \geq t\}$ of $u$
%, for $t<\sup\{u(x):\ x\in \Omega\}$
is contained in the boundary $\partial\Omega$. In optimal transport terms, this means that there are no transport rays between $f^+$ and $f^-$ which are contained in $\partial\Omega$ (we recall that $f^+$ and $f^-$ are the positive and negative parts of $f$, the tangential derivative of $g$). 

Let $\Delta$ be the interior of the set where $u$ is locally constant or equivalently, where the transport density $\sigma=|v|$ vanishes. This set has at most a countable number of disjoint connected components denoted by
$\Delta_i$, $i\in I_\Delta$, with positive Lebesgue measures. We shall introduce $I_F=\{i\in I_\Delta: \cH^1(\partial\Delta_i \cap\partial\Omega)>0\}$. Then, we define the flat parts as follows:
$$
F_i = \partial\Delta_i \cap\partial\Omega, \qquad i\in I_F.
$$
%However, the set defined above needs not be the closure of an open arc.
We shall see that for any $i$, the set $\Omega\cap \partial\Delta_i$ is composed of transport rays. Indeed, if $z \in  \Omega \cap\partial \Delta_i$ then there will be a sequence $z_n \in \Omega \setminus \overline{\Delta}$ such that $z_n \to z$. Since $z_n \notin \overline{\Delta}$ then there will be a transport ray $\cR_n=[x_n,y_n]$ such that $z_n \in \cR_n$. Yet, after extracting a subsequence (not relabeled), these transport rays $\cR_n$ converge to a line segment $\cR=[x,y]$, where $x_n \to x$ and $y_n \to y$. Moreover, %we have 
$z \in \cR$  and due to Lemma \ref{l-conv} (2) applied to constant sequences $f^\pm_n = f^\pm$, we infer that $\cR$ %\AS{Confusion of Notation $R$} 
is a transport ray. %\marginnote{\PR{suggestion:\\
%$R$ transport map\\$\mathcal{R}$ transport ray}}
In particular, $\partial\Delta_i$ intersects $\partial\Omega$.
%and these sets $\{\Delta_i\}_i$ are essentially disjoint, so there will be a denombrable set of such regions $\Delta_i$. 
At the same time, we see that the set $\partial \Omega \setminus \overline{\bigcup_{i\in I_\Delta} \Delta_i}$ is a sum of open arcs, let us call it $\bigcup_{j\in J} \alpha_j$, where $\alpha_j\subset \partial\Omega$ {and $J$ is at most countable index set}.

{\it Step 2.} Let us suppose that $g$ has at least one strict local minimum or maximum. If this is not the case we will proceed to Step 3. We take $c_k$ a strict local minimum or maximum of $g$. Since $c_k\in \partial \Omega \setminus \overline{\bigcup_{i\in I_\Delta} \Delta_i}$, then there is $j\in J$ so that $c_k\in\alpha_j$. Moreover, there will be an {open} arc $\tilde\chi_k:=\arc{c_k^+ c_k^-}$ around $c_k$, where $g$ is strictly increasing on $\chi_k^+:=\arc{c_k\,c_k^+}$ and strictly decreasing on $\chi_k^-:=\arc{c_k\,c_k^-}$. {We notice that} $[c_k^+,c_k^-]$ is a transport ray. {This follows from the fact that any level set of a solution to \eqref{LGP} is a transportation ray.
%definitiothe solutions to \eqref{LGP}
%convexity of $\Omega$, the fact that level sets of the solutions to \eqref{LGP} have minimal perimeter \cite{sternberg}} {and the definition of $v$}. Thus, we have $TV(g_{|\chi^+})=TV(g_{|\chi^-})$. {Hence, the} arcs $\chi^\pm\subset \partial\Omega$\, are relatively open {and } $\chi \subset \alpha_j$, for some $j \in J$. We note that so far our choice of $\chi$ is arbitrary. So, in order to make it fixed, we assume that $\chi$ is the maximal open arc containing  $c$ of this form contained in $\alpha_j$. {The collection of all these maximal open arcs $\chi$ is at most countable, so that we may write, $\{\chi_i\}_{i\in I_\chi}$ with $I_\chi \subset \mathbb{N}$.}
%\sout{We set 
%$I_\chi =\{ i\,: \exists\,\,j \in J \,\,\mbox{such that}\,\, \alpha_j \hbox{ contains an arc of the form }\chi:=\chi_i\}$. Since a continuous function on $\partial\Omega$ may have at most countable number of strict local maxima or minima, the same is true for the number of maximal arcs $\chi_i$.} % around strict local minima/maxima $c_i$, where $g$ is strictly increasing on $\chi_i^+$ and strictly decreasing on $\chi_i^-$ with $TV(g_{|\chi_i^+})=TV(g_{|\chi_i^-})$. 
Let us denote by $D_k\subset \Omega$ the convex hull of $\chi_k$, we notice that $\partial D_k\cap \Omega = [c_k^+,c^-_k]$.
%, where $\chi^+_k = \arc{c_k\,c^+_k}$, $\chi^-_k = \arc{c_k\,c^-_k}$ and, $c_i$ is a strict local minimum or maximum. %(we recall that $[c_i^+,c^-_i]$ is a level set). 
%\sout{set or equivalently, a transportation ray.}

{\it Step 3.}
We define an open set  $T:=\Omega \setminus\left( \overline{\bigcup_{i\in I_\chi} D_i} \cup \overline{\bigcup_{i\in I_\Delta}\Delta_i}\right)$. We claim that any open connected component of \,$T$\, is of the form $T_i$ (see condition (H1)), i.e. it is the convex hull of two open arcs   $\Gamma_i^+$ and $\Gamma_i^-$, where $g$ is strictly increasing on $\Gamma_i^+$ and strictly decreasing on $\Gamma_i^-$ with $TV(g_{|\Gamma_i^+})=TV(g_{|\Gamma_i^-})$ and $\dist(\Gamma_i^+,\Gamma_i^-)>0$. 

Let \,$C$\, be  an open connected component of \,$T$. We claim that $\partial C\cap \partial\Omega$ consists of a sum of two disjoint closed arcs. First, it is easy to see that $\partial C\cap \Omega$ consists of transportation rays, hence $C$ is convex.

Since $C \cap \overline{\Delta}=\emptyset$, then the interior relative to $\partial\Omega$ of any arc of $\partial C \cap \partial \Omega$ does not contain any multiple point, otherwise there will be an $i \in I_\Delta$ such that $\Delta_i$ divides $C$ into two parts but this is a contradiction because $C$ is connected. By the same argument, one can see that the interior of any arc of $\partial C \cap \partial \Omega$ 
does not also intersect the flat part $F$ of $g$. Moreover, it is clear that $g$ does not attain a strict local minimum/maximum in the interior of $\partial C \cap \partial \Omega$}. Hence, on any arc of $\partial C \cap \partial \Omega$, the boundary datum $g$ is either strictly increasing or strictly decreasing. 
%We also noted above that $\partial C\cap \Omega$ is composed of transport rays. 
%\sout{, this follows from the fact that $\partial \Delta_i \cap \Omega$ and $\partial D_i \cap \Omega$ are transport rays.}
Hence, any point $x^+\in \partial C \cap \partial \Omega$ is an endpoint of a transportation ray $[x^+, x^-]$, where $x^-\in \partial C \cap \partial\Omega$.

Since the set %On the other hand, 
$U =\partial \Omega \setminus \left( \overline{\bigcup_{i\in I_\chi} D_i} \cup \overline{\bigcup_{i\in I_\Delta}\Delta_i}\right) $ is open it is a sum of open disjoint arcs. We set
$$
A_C =\{\alpha\subset U: \alpha \hbox{ is an open connected component and }\alpha\cap \partial C\neq \emptyset\}.$$
It is clear that if \,$\alpha\in A_C$ then $\overline{\alpha}\subset\partial C$. Now, we
claim that if $\alpha\neq\beta$ are both in $A_C$, then $\overline{\alpha}\cap\overline{\beta} =\emptyset.$ 
%Thanks to the definition of $U$, wWe note that 
Let us suppose otherwise, i.e.  $ \overline{\alpha}\cap\overline{\beta}=\{p\}$. %We assume that arcs $\alpha$ and $\beta$ have the following forms: $\alpha =\arc{x_1^\alpha x_2^\alpha}$ and $\beta = \arc{x_1^\beta x_2^\beta}$. We note that $x_1^\alpha$, $x_2^\alpha$, $x_1^\beta$, $x_2^\beta$  are
%limits of transportation ray endpoints, they are also 
%the endpoints of some transportation rays. 
%If \,$\overline{\alpha}\cap\overline{\beta} =\{p\}$, then 
%We may assume that $x_1^\alpha= p=x_1^\beta$. 
Hence, $g$ is strictly monotone on $\theta:=\alpha \cup \{p\}\cup\beta$, for otherwise we have an arc of type $\chi$. Then, we have two possibilities: (i) either $p$ is a multiple point, or (ii) $p$ is not a multiple point. 

In the first case, we see that there will be two transportation rays $\cR_1=[p, q_1]$ and $\cR_2=[p,q_2]$, where $q_1,\,q_2\in  \partial C \cap \partial\Omega$ and $q_1\neq q_2$. Yet, these rays $\cR_1$ and $\cR_2$ separate two components of $C$, which was assumed to be connected, a contradiction.

If (ii) holds, then there will be a unique transportation ray $\cR= [p,q]$ starting at $p$\, such that  %where $y\in 
$q$ belongs to $\partial C\cap \partial\Omega$. 
%because we know that $\partial C\cap \Omega$ consists of transportation rays. 
Since $g$ is strictly monotone on $\theta$, 
this ray %Here, we face two possibilities: either (a) 
\,$]p,q[$ must be contained in $C$.
%or (b) $\cR\subset \partial C$.
%Assume that case (a) occurs.
%Then, we notice that there are transportation rays on both sides of $\cR$, otherwise we reach a contradiction with the definition of $C$. 
Suppose that $p^\alpha_n\in \alpha$ (resp. $p^\beta_n\in \beta$) is a sequence of points converging to $p$. We take the rays emanating from these points, $[p^\alpha_n, q^\alpha_n]$,  $[p^\beta_n, q^\beta_n]$. Moreover, by strict monotonicity of $g$, these rays are contained in $C$ and we have 
$$
\lim_{n\to \infty} q^\alpha_n = q = \lim_{n\to \infty} q^\beta_n,
$$
since otherwise we would reach a contradiction with the fact that $p$ is not a multiple point. Thus, 
%for any $x_0\in \cR$ there are points in $C$ on both sides of $\cR$, hence convexity of $C$ implies that $(\cR\cap\Omega)\subset C$. But this means that
there is a ball $B(p, r)$, which does not contain any point from $\overline{\bigcup_{i\in I_\chi} D_i} \cup \overline{\bigcup_{i\in I_\Delta}\Delta_i}$, but this yields again  a contradiction with the definition of $p$.
%Let us consider case (b). If this happens, then one of the two arcs (say $\alpha$) should be contained in the transport ray $\mathcal{R}$. In particular, $\mathcal{R} \subset \partial\Omega$ and
%Without loss of generality, we may assume that $y\in \alpha$.
%Let us take any sequence $x_n\in \beta$ converging to $p$ and the corresponding rays $[x_n, y_n]$.
 %any point on $\alpha$ must be transported to a point on the same transport ray $\mathcal{R}$, which is not possible.
%The sequence $y_n$ must converge to $y$, otherwise $p$ is a multiple point. As a result, there is a  neighborhood of $p$ free of points from  $\overline{\bigcup_{i\in I_\chi} D_i} \cup \overline{\bigcup_{i\in I_\Delta}\Delta_i}$, which is contrary to the assumption.
Hence, we conclude that $\overline{\alpha}\cap\overline{\beta} =\emptyset$.
Now, take any arc $\alpha_1 \subset \partial C \cap \partial \Omega$ and any point $x^+$ in the interior of $\alpha_1$. Assume that $[x^+,x^-]$ is a transport ray. {As we noted above $g$ is strictly monotone on $\alpha_1$ so that $x^-$ cannot belong to $\alpha_1$. Thus,}
$x^- \in \alpha_2 \subset \partial C \cap \partial\Omega$, where $\overline{\alpha_1} \cap \overline{\alpha_2}=\emptyset$. 
Then, all the other points in the interior of $\alpha_1$ must be transported to the same arc $\alpha_2$ because otherwise
%f this is not the case then 
there would be a multiple point inside $\alpha_1$, which is a contradiction as we already showed that there are no multiple points inside any arc of $\partial C \cap \partial \Omega$.
%Hence, our claim follows.
%if $x^+\in \partial C \cap \partial \Omega$ is any point, then $x^+$ is an endpoint of a transportation ray $[x^+, x^-]$ with $x^-\in \partial C \cap \partial\Omega$. Since $x^+$ is not a strict maximum/minimum of $g$ and $x^+ \notin \overline{F}$, then $x^+$ belongs to a connected component in the boundary of some super-level set $E_t =\{ u>t\}$, i.e. there is $x^-\in \partial\Omega$ such that $[x^+, x^-]\subset \partial E_t$. In other words, $[x^+, x^-]$ is a transportation ray.  
%Since  two different transport rays cannot intersect at any interior point of $\Omega$, we see that $[x^+, x^-]$
%any point has to be transported to another point $x^- \in \partial C \cap \partial \Omega$.
%\PR{[FINE, HOW DOES IT IMPLY THAT A POINT FROM ONE COMPONENT OF  $\partial C \cap \partial \Omega$ IS TRANSPORTED TO ANOTHER ONE?]}
Consequently, $\partial C \cap \partial \Omega$ can be decomposed into two arcs $\Gamma^+$ and $\Gamma^-$, where $g$ is strictly increasing on $\Gamma^+$ and strictly decreasing on $\Gamma^-$ with $TV(g_{|\Gamma^+})=TV(g_{|\Gamma^-})$ and $\dist(\Gamma^+,\Gamma^-)>0$. 

{\it Step 4.} Set 
$$I_\Gamma=\{i\in I_\Delta:T_i\,\,\,\mbox{is a connected component of}\,\,T\}.
$$
Since %these 
sets $T_i$, $i\in I_\Gamma$, are disjoint and open, %$\mathcal{L}^2(T_i)>0$, for all $i$, 
then %this 
set $I_\Gamma$ is at most countable. It is also clear that the sets $T_i$ ($i \in I_\Gamma$) and $D_j$ ($j \in I_\chi$) are mutually disjoint. Hence, the condition (H1) is  satisfied.

%Since the sets of this family $\{T_i,\,D_i : i \in I_\Gamma \cup I_\chi\}$ are mutually disjoint, for all $i \in I_\Gamma$ and $j \in I_\chi$. 
%\PR{WE DO NOT USE IT EXPLICITLY HERE.}
Moreover, we obviously have that for every $x^+ \in \chi_i^+$, $i \in I_\chi$ (resp. \,$x^+ \in \Gamma_i^+$, $i \in I_\Gamma$),
the line segment $[x^+,{\bf{T}}(x^+)]$ 
%\PR{[WHAT IS $\bbT$?, THIS MAY BE ONLY THE TRANSPORTATION MAP, SHOULD DENOTE DIFFERENTLY TO AVOID CONFUSION WITH THE PREVIOUS $\bbT$]} 
is a transport ray (see Lemma \ref{lyy}), which %and so, it 
is contained in $\Omega$. Hence, (H2) holds as well.

Finally, (H3) is also satisfied. Consider any finite sequence of points $\{e_k^+\}_{1 \leq k \leq m}$ (where $m \in \mathbb{N}$) such that $e_k^+\in \chi_{i_k}^+ \cup \Gamma_{i_k}^+$, for some $i_k \in I_\chi \cup I_\Gamma$. Then,
%and any $n-$permutation $s$, 
we have $(e_k^+,{\bf{T}}(e_k^+)) \in \spt(\gamma)$, for every $1 \leq k \leq m$ and so, thanks to the equality \eqref{transport ray} which is satisfied by the Kantorovich potential $\phi$, we get the following identity:
$$
\sum_{k=1}^m |e_k^+ - {\bf{T}}(e_k^+)|=\sum_{k=1}^m \phi(e_k^+) - \phi({\bf{T}}(e_k^+))
\,\,=\,\, \sum_{k=1}^{m-1}
[\phi(e_k^+) - \phi({\bf{T}}(e_{k+1}^+))] + \phi(e_m^+) - \phi({\bf{T}}(e_{1}^+)).$$
Recalling the proof of Proposition \ref{Uniqueness in the convex case}, we see that $[e_k^+,{\bf{T}}(e_{k+1}^+)]$ is not a transport ray and then, we have $\phi(e_k^+) - \phi({\bf{T}}(e_{k+1}^+)) < |e_k^+ - {\bf{T}}(e_{k+1}^+)|$. Consequently, we get
$$ \sum_{k=1}^m |e_k^+ - {\bf{T}}(e_k^+)| <
\sum_{k=1}^{m-1}
|e_k^+ - {\bf{T}}(e_{k+1}^+)| + |e_m^+ - {\bf{T}}(e_{1}^+)|. $$

This concludes the proof.  $\qedhere$
\end{proof}

\section{Sufficient conditions for existence and uniqueness in the non-convex case}\label{sec:nonconvex}
In this section, we will extend the equivalence between Problems \eqref{Beckmann} and \eqref{Kantorovich} to the general case of a bounded, simply connected, not necessarily convex domain, $\Omega \subset \R^2$,
%But, this is not possible for an arbitrary measure $f$ on $\partial\Omega$. However, it is possible to prove that,
%More precisely, we will prove this equivalence 
however,  some admissibility conditions on the Dirichlet datum $g$ are imposed. %(or equivalently, on its tangential derivative $f:=\partial_\tau g$). 
We stress that the assumption of the piecewise monotonicity of $g\in C(\partial\Omega)$ is a standing assumption.
Then, we will be able to show, under these conditions, the existence and uniqueness of a solution to the least gradient problem \eqref{LGP}.
%provided that $g \in BV(\partial\Omega) \cap C(\partial\Omega)$. 
We stress that the condition (S), which is a regularity assumption on $\partial\Omega$ is in force (but, we will see in Proposition \ref{General case} that one can relax this condition). %Moreover, we assume throughout this section that the boundary datum $g$ is piecewise monotone in the sense of Definition \ref{d-ps-mono}.\\

To begin with, %irst of all, 
we introduce the following assumptions:\\ 
\\
$\bullet$\,\,\,{\bf{Condition (L1)}}.\,\,\,The domain $\Omega$ can be decomposed into convex disjoint sets $\tilde C_i$ ($i \in I_C$) and disjoint open sets $X_i$ ($i \in I_X$) such that $g$ is constant on $\partial X_i \cap \partial\Omega$. In order to use the setting of Section \ref{Section Convex case} we need data $(C_i, g_i),$ $i \in I_C$, where $C_i$ is open and convex, $g_i\in C(\partial  C_i)$. Namely, we set $C_i:= {\tilde C_i}^\circ$ and $g_i(x) = g(x_i)+ (f\LC (\partial C_i\cap \partial \Omega))(\arc{x_ix})$ and $x_i\in \partial C_i\cap\partial \Omega$ is fixed. Of course, sets $C_i$ are disjoint. Moreover, we have the following:

(1) \,\,For every $i \in I_C$, we assume that %the restriction of 
$g_i \in C(\partial C_i )$
%\cap \partial\Omega$ 
satisfies the condition (H1), i.e. $\partial C_i \cap \,\partial\Omega$ can be decomposed into open arcs \,${\Gamma_j^i}^\pm$ ($j\in I^i_\Gamma$), \,${\chi_j^i}^\pm$ ($j \in I^i_\chi$)\, and $F^i_j$ ($j \in I^i_F$), satisfying all the points of (H1) (see Section \ref{Section Convex case}). We note that $g_i$ is constant on each component of $\partial C_i\cap \Omega.$

In the sequel, we also use the following notations:
$\Gamma^i_j={\Gamma^i_j}^+\cup {\Gamma^i_j}^-$ ($i \in I_C$, $j \in I^i_\Gamma$), $\chi^i_j={\chi^i_j}^+\cup {\chi^i_j}^-$ ($i \in I_C$, $j \in I_\chi^i$), ${\Gamma^i}^{\pm}=\bigcup_{j\in I_\Gamma} {\Gamma^i_j}^{\pm}$, ${\chi^i}^{\pm}=\bigcup_{j\in I_\chi}{\chi^i_j}^{\pm}$, $\Gamma^i={\Gamma^i}^+ \cup {\Gamma^i}^-$, $\chi^i={\chi^i}^+ \cup {\chi^i}^-$ ($i \in I_C$), ${\Gamma}^{\pm}=\bigcup_{i\in I_C} {\Gamma^i}^{\pm}$, ${\chi}^{\pm}=\bigcup_{i\in I_C}{\chi^i}^{\pm}$, $\Gamma={\Gamma}^+ \cup {\Gamma}^-$ and, $\chi={\chi}^+ \cup {\chi}^-$. This will help us express the next conditions.\\

%\PR{(2) Let us set $N = \partial\Omega \setminus (\Gamma\cup \chi \cup \bigcup_{i\in I_F} F_i)$. We assume that
%$$
%|f|(N) =0.
%$$
%}

We will note further simple consequences of the assumption (L1).

\begin{lemma}\label{zer-bry}
Let us suppose that $C_i$, $i\in I_C$, is one of the sets defined above, then $f(\partial C_i) =0.$
\end{lemma}
\begin{proof}
First, we note that $f(\partial C_i) = f(\partial C_i \cap \partial\Omega)$. By definition of $ C_i$, the set $\partial C_i \cap \partial\Omega$ has the following structure:
$$
\partial C_i \cap \partial\Omega=
\bigcup_{j\in I^i_\chi} ({\chi^i_j}^+\cup {\chi^i_j}^-) \cup
\bigcup_{j\in I^i_\Gamma} ({\Gamma^i_j}^+\cup {\Gamma^i_j}^-) \cup
\bigcup_{j\in I^i_F} F^i_j \cup N^i,
$$
where $N^i = (\partial C_i \cap \partial\Omega) \setminus (\Gamma^i\cup \chi^i \cup \bigcup_{j\in I^i_F} F^i_j)$. But, by the assumption that $g$ is piecewise monotone, we see that this set $N^i$ is a null set with respect to the measure $|f|$. %\marginnote{correct pcwse mntn}
%is the set of the endpoints of all these arcs. Since $N$ is at most countable and $f$ has no atoms, then $f(N)=0$.
Hence, by (L1), we get that
%again the definition of $ C_i$ implies that
$$
f(\partial C_i \cap \partial\Omega)=
\sum_{j\in I^i_\chi} f({\chi^i_j}^+\cup {\chi^i_j}^-) +
\sum_{j\in I^i_\Gamma} f({\Gamma^i_j}^+\cup {\Gamma^i_j}^-) =0.\qedhere
$$
\end{proof}

We follow the same strategy of construction of solutions as in the previous section. Namely, we begin by
%On the other hand, we also need to
introducing a transport map $\bbT:\Gamma^+ \cup \chi^+ \mapsto \Gamma^- \cup \chi^-$. In the first step of the construction, we define $\bbT^i$ on $\partial C_i$ with the help of the formula (\ref{df-T}). We can do this due to Lemma \ref{zer-bry}. %we have $f(\partial  C_i) =0 $.
This definition of $\bbT^i$ %as in \eqref{df-T} 
implies in particular that ${\bf{T}}^i$ has all the properties proved for the map ${\bf{T}}$ in Section \ref{Section Convex case}, (as usual, ${\bbT^i}^{[-1]}$ denotes the inverse of $\bbT^i$). 
%Finally, since $\Gamma^+ \cup \chi^+ = \bigcup_{i\in I_C} (\Gamma^+ \cup \chi^+)\cap C_i$ 

After this preparation, we set
\begin{equation}\label{df-TTi}
\bbT(x^+) =  \bbT^i(x^+),\qquad\hbox{whenever }\,x^+\in (\Gamma^+ \cup \chi^+) \cap C_i.
\end{equation}
Since the sets $C_i$ are disjoint, then we see that this map $\bbT$ is well-defined.
%there is no ambiguity in  this definition.} 
%which is defined as follows: 
%%For a given function $g\in C(\partial\Omega)\cap BV (\partial\Omega)$ we set $f:=\partial_\tau g$. 
%$$\bbT(x^+)=\begin{cases}
%x^-\in {\chi_j^i}^- & \mbox{if}\,\,\,\,x^+ \in {\chi_j^i}^+\,\,\,\,\,\mbox{and}\,\,\,\,\,f(\arc{\,x^+\,x^-})=0,\qquad i\in I_C,\,\,\, j \in I^i_\chi,\\
%x^-\in {\Gamma_j^i}^- & \mbox{if}\,\,\,\,x^+ \in {\Gamma_j^i}^+\,\,\,\,\,\mbox{and}\,\,\,\,\,f(\arc{\,x^+\,x^-})=0,\qquad  i \in I_C,\,\,\, j\in I^i_\Gamma,
%\end{cases}
%$$
%where $\arc{\,x^+x^-}$ is any of the two arcs of \,$\partial\Omega$\, joining the points $x^+$ and $x^-$ (note that $f(\partial\Omega)=0$). For every $i \in I_C$, we will denote by $\bbT^i$ the restriction of ${\bf{T}}$ to ${\Gamma^i}^+ \cup {\chi^i}^+$ %as follows
%the %last assumption (H3), we need to define
%the map $\bf{t}$ on $\chi^+ \cup \Gamma^+$ as follows (in the sequel, 
%Moreover, ${\bbT^i}^{[-1]}$ denotes the inverse of $\bbT^i$. 
%In other words, the map ${\bf{T}}^i$ here is nothing else than the \PR{transport} map ${\bf{T}}$ that we have already introduced in the previous section in the case of convex domain (here, the convex domain is $C_i$ instead of $\Omega$). 
%The above definition of $\bbT^i$ implies in particular that ${\bf{T}}^i$ has all the properties proved for the map ${\bf{T}}$ in Section \ref{Section Convex case}.
%The structure of the decomposition of $\partial C_i$ into arcs of different types assures us that $\bbT^i$ defined on a part of $\partial C_i$ takes value in $\partial C_i$.
%\bigskip
Here, come our next requirements on the boundary datum $g$:\\
%$\bullet$ \,{\bf{Condition (A2)}}. 

(2) \,\,For every $i \in I_C$, the restriction of \,$g$\, to $\partial C_i \cap \partial\Omega$ satisfies the condition (H2), i.e. for all $x^+ \in {\Gamma}^+ \cup {\chi}^+$, 
%and $y \in \chi_i^-$ (resp. $\Gamma_i^-$),
we have $]x^+,\bbT(x^+)[ \subset \Omega$.\\

(3) \,\,For every $i \in I_C$, \,$g$\, satisfies the condition (H3) on $\partial C_i \cap \partial\Omega$, i.e. for any sequence of points $\{e_{k}^+\}_{1 \leq k \leq m}$ (where $m \in \mathbb{N}$) such that $e_{k}^+ \in {\chi_{j_k}^i}^+ \cup \,{\Gamma_{j_k}^i}^+$ (for some $j_k \in I^i_\chi \cup I^i_\Gamma$), we have the following inequality
$$%\mbox{(A3-1)} \qquad
\sum_{k=1}^m |e_k^+ - \bbT^i(e_k^+)|
\,\,<\,\,
\sum_{k=1}^{m-1}
|e_k^+ - \bbT^i(e_{k+1}^+)| + |e_m^+ - {\bf{T }}^i(e_{1}^+)|.$$ 

%(4) $g$ is constant on $\partial X_i \cap \partial\Omega$.\\
\begin{remark}\label{re-doD}
We note that these conditions encompassed in (L1) impose important restrictions on the geometry of $\partial\Omega$. Namely, even if the domain $\Omega$ is not convex,  (L1) implies that the arcs of $\partial C_i\cap \partial\Omega$ must have non-negative curvature for all $i\in I_C$. As a result the 
$D=\{(x_1, x_2)\in \bR^2: x_1\in(-1,1), \ x_2\in( \sqrt{1-x_1^2}, \sqrt{1-x_1^2}+1)\}$ does not satisfy (L1) as soon as $g \in C(\partial D)$ is not constant on the graph of $x_1 \in (-1,1) \mapsto \sqrt{1-x_1^2}$. However, we will also cover in Proposition \ref{Prop. nonconvex case general} the case when some arcs of $\spt(f)$ have negative curvature. 
\end{remark}

Since $\Omega$ is not convex, then we will also need an additional condition that guarantees that the boundary of each set $C_i$ ($i \in I_C$) is transported to itself. Namely, we assume the following:
\\
\\
$\bullet$ \,{\bf{Condition (L2)}}.
 Let $\{{e_k}^\pm\}_{1 \leq k \leq m}$ (where $m \in \mathbb{N}$) be
two finite sequences of points  %(where $m \in \mathbb{N}$)
such that ${e_{k}}^\pm \in {\chi^{i_k}}^\pm \cup {\Gamma^{i_k}}^\pm$, with $i_k \neq i_{k^\prime}$ for all $k \neq k^\prime$. %for some $1 \leq  
 %for all $1 \leq k \leq m$ and $1 \leq r \leq r_k$,
Then, we assume the following additional inequality: %\marginpar{\textcolor{black}{inconsistency}}
$$
\sum_{k=1}^m  |{e_k}^+ - {e_k}^-| <
%\sum_{k=1}^m\bigg[|{e_{j,n_{i,j}}^k}^+ - {e_{s_2(i,j),1}^{s_1(i,j)}}^-| +
%|{e_{j}^i}^+ - {\bf{t^{s(i)}}}({e_{s^i(j)}^{s(i)}}^+)| +
%|{e_{j,1}^i}^+ - {\bf{t^{i}}}({x_{j,1}^{i}})| +
\sum_{k=1}^{m-1} |{e_{k}}^+ - {e_{k+1}}^-| + |{e_{m}}^+ - {e_{1}}^-|.$$
We stress that here $e_k^+$ and $e^-_k$ are two arbitrary points on $\partial C_{i_k}$. In particular, $e_k^-$ needs not to be the image of $e_k^+$ under the map ${\bf{T}}$. 
%$$+\sum_{r=1}^{r_m-1}|{e_{r}^m}^+ - {\bf{T^m}}({e_{r+1}^{m}}^+)|+|{e_{r_m}^m}^+ - {\bf{T^1}}({e_{r_1}^{1}}^+)|.$$
%\PR{ARE BOTH CASES USED?}
% Again, let us denote by $S$ the set of singularity points on $\spt(f)$. Then, we assume %say that the (H) holds if we have 
%the following:
%$$
%S\,\,\,\mbox{is countable but not dense anywhere}.\leqno({\rm C4})
%$$
\begin{proposition}\label{Prop. nonconvex case} 
%Let $g\in BV(\partial\Omega) \cap C(\partial\Omega)$ be given, we set $f=\partial_\tau g$, where $\partial_\tau g$ is defined by (\ref{r1}) and 
Assume 
%$\Omega \subset \bR^2$ is open and bounded. Let us assume 
that 
%the couple $(\Omega, g)$ satisfies 
conditions (L1) and (L2) hold.
%Then, all the transport rays between $f^+$ and $f^-$ lie inside $\Omega$. More precisely, if 
Let \,$\gamma$\, be an optimal transport plan for Problem \eqref{Kantorovich}, then for $\gamma-$a.e. $(x^+,x^-)\in \spt (\gamma)$, there exists $i\in I_C$ such that $x^-={\bf{T}}^i(x^+)$. In particular, the optimal transport plan $\gamma$ is unique and, we have $\gamma=(Id,{\bf{T}})_{\#}f^+$.
%we either have \,$x^+ \in {\chi^i_j}^+$ and \,$x^- \in {\chi^i_j}^-$ or $j\in I^i_\Gamma$ and \,$x^+ \in {\Gamma^i_j}^+$ and \,$x^- \in {\Gamma^i_j}^-$.
\end{proposition}

\begin{proof}
We claim %It remainsto show 
that $\gamma(\bar C_i \times \bar C_j)=0$. For all $i \neq j$. The proof %of this claim 
is similar to the one in Proposition \ref{strictly_convex_case} and it is performed in Steps 1.1 till 1.3.

{\it Step 1.1.} Let us assume that there is a couple $(e_1^+, e_2^-) \in \spt(\gamma)$\, with
%Assume there is a point 
\,$e_1^+ \in  {\Gamma^{i_1}}^+ \cup {\chi^{i_1}}^+$ but %another one  
$e_2^- \in 
 \,{\Gamma^{i_2}}^- \cup {\chi^{i_2}}^-$, where $i_1 \neq i_2$.
%such that $(e_1^+,e_2^-)$.
%that is not transported to ${\chi^{i_1}}^- \cup {\Gamma^{i_1}}^-$, for some $1 \leq i_1 \leq n$. Hence, there will be another open arc ${E_{2}}^- \subset {\chi^{i_2}}^- \cup {\Gamma^{i_2}}^-$ such that $\gamma[{E_1}^+ \times \partial\Omega\backslash {E_{2}}^-]=0$, where $1 \leq i_2 \neq i_1 \leq n$. 
%, or an open arc ${E^{1}_{2,1}}^- \subset {\chi^{i_1}_{j_2}}^-$, where $1 \leq j_2 \neq j_1 \leq n_{i_1}$, such that $\gamma[{E^{1}_{1,1}}^+ \times \partial\Omega\backslash {E^{1}_{2,1}}^-]=0$.
 %As $x \in \mbox{spt}(f^+)$,
%and $y \in \spt(f^-)$, 
%then $x \in \chi_i^+$ for some $i$.
%an  
%$y \in \chi_i^+$,
%,\,\Gamma_j^-$ or $F_j^{++}$,
%for some $i$.
%Suppose that $y \notin \chi_i^-$.
%Let ${e_{1}}^+$ be a point on ${E_{1}}^+$ and $e_2^- \in {E_2}^-$ be such that $({e_{1}}^+,e_2^-) \in \spt(\gamma)$. Now, 
%We want to show that
% Let us suppose the contrary.
Then, we are going to construct two sequences of points in $\Gamma^\pm \cup \chi^\pm$,
$$
e_1^+,\,e_2^+,\,\ldots \quad\hbox{and}\quad e_2^-,\,e_3^-,\,\ldots
$$
such that $(e^+_k, e^-_{k+1})\in \spt(\gamma)$, $e^+_k \in {\Gamma^{i_k}}^+ \cup {\chi^{i_k}}^+$ and $e^-_{k+1} \in {\Gamma^{i_{k+1}}}^- \cup {\chi^{i_{k+1}}}^-$,
%,
where \,$i_k \neq i_{k+1}$. 

Let us suppose that we have $e^-_k\in {\Gamma^{i_k}}^- \cup {\chi^{i_k}}^-$ and $e^+_{k-1}\in {\Gamma^{i_{k-1}}}^+ \cup {\chi^{i_{k-1}}}^+$. We will construct $(e^+_k, e^-_{k+1})$. Since $f$ is atomless, then  
%so we will see that one can always find two points \,$e^+_{k}\in {\Gamma^{i_k}}^+ \cup {\chi^{i_k}}^+$ and $e^-_{k+1}\in {\Gamma^{i_{k+1}}}^- \cup {\chi^{i_{k+1}}}^-$, $i_k \neq i_{k+1}$, with $(e^+_k, e^-_{k+1})\in \spt(\gamma)$. Since $e^-_k\in {\Gamma^{i_k}}^- \cup {\chi^{i_k}}^-$, 
there is a Borel set $G^-_k\subset {\Gamma^{i_k}}^- \cup {\chi^{i_k}}^-$ containing $e^-_k$ with $f^-(G^-_k)>0$, which was transported from a set $G_{k-1}^+ \subset {\Gamma^{i_{k-1}}}^+ \cup {\chi^{i_{k-1}}}^+$ containing $e^+_{k-1}$. Yet, thanks to Lemma \ref{zer-bry}, we know that $f(\partial C_{i_k})=0$. Hence, the mass imported into $C_{i_k}$ must be balanced with an equal outflow of the mass.
Then, there will be a Borel set $G^+_k\subset {\Gamma^{i_k}}^+ \cup {\chi^{i_k}}^+$ with $0<f^+(G^+_k) \leq  f^-(G^-_k)$, which is transported to a set $G^-_{k+1}\subset {\Gamma^{i_{k+1}}}^- \cup {\chi^{i_{k+1}}}^-$. So, let us just pick any couple $(e^+_k, e^-_{k+1}) \in (G_k^+\times G^-_{k+1})\cap \spt(\gamma)$.  
%point ${e_2}^+\in  {\chi^{i_2}_{j_2}}^+$ or ${e_2}^+\in  {\Gamma^{i_2}_{j_2}}^+$ and another point ${e_3}^- \in {\chi^{i_3}_{j_3}}^-$ or ${e_3}^- \in {\Gamma^{i_3}_{j_3}}^-$ such that $(e_2^+,e_3^-) \in \spt(\gamma)$.
%that must be transported outside $ {\chi^{i_2}}^- \cup {\Gamma^{i_2}}^-$, say to
%and ${e^{2}_1}^+:={\bf{T^2}}^{[-1]}({e_{1}^2}^-)$. So, there are two possibilities: either ${E^2_1}^+$
%This implies that there is an open arc ${E_2}^+ \subset \chi^{i_2}^+ \cup \Gamma^{i_2}^+$ that
%is transported outside ${\chi^{i_2}}^- \cup {\Gamma^{i_2}}^-$ or not. If so, then there is 
%an open arc ${E_{3}}^- \subset {\chi^{i_3}}^- \cup {\Gamma^{i_3}}^-$ (i.e. we have $\gamma[{E_2}^+ \times \partial\Omega\backslash {E_{3}}^-]=0$, where $1 \leq i_3 \neq i_2 \leq n$). Let %${e^{2}_1}^+:={\bf{T^2}}^{[-1]}({e_{1}^2}^-$ be a point on ${E^{2}_1}^+$ and 
%$e_2^+ \in E_2^+$ and ${e_3}^- \in {E_3}^-$ be such that $({e_{2}}^+,{e_3}^-) \in \spt(\gamma)$ (we note that the point $e_2^+$ is not necessarily ${\bf{T^{i_2}}}^{[-1]}(e_2^-)$).
In this way, we get sequences with the desired properties.

{\it Step 1.2.}  Let us assume that the index set $I_C$ is finite.
Now, we claim that $i_k \neq i_{k^\prime}$, for all $k \neq k^\prime$. Let us suppose that there exist \,$l,\,m \geq 1$\, such that $e_l^\pm,\,e_{l+m+1}^\pm \in {\Gamma^{i_l}}^\pm \cup {\chi^{i_l}}^\pm$ (i.e., one has $i_l=i_{l+m+1}$).
%Now, we claim that $m<\infty$ and 
%\begin{equation}\label{r-m}
 %e_{m+1}^- \in {\Gamma^{i_1}}^- \cup {\chi^{i_1}}^-.
 %\subset C_{i_1}\qquad(\hbox{resp. } e_{m+1}^- \in {\Gamma^{i_1}}^-\subset C_{i_1}).   
%\end{equation}
Let $\phi$ be again a Kantorovich potential between $f^+$ and $f^-$. Then, thanks to the equality \eqref{transport ray}, we get that 
\begin{align*}
 \sum_{k=l}^{l+m-1} |{e^+_{k}} - {e^-_{k+1}}| + |{e^+_{l+m}} - {e^-_{l+m+1}}|&=\sum_{k=l}^{l+m-1} [\phi({e^+_{k}}) - \phi({e^-_{k+1}})] + [\phi({e^+_{l+m}}) - \phi({e^-_{l+m+1}})]\\
&= [\phi({e^+_l}) - \phi({e^-_{l+m+1}})]+\sum_{k=l+1}^{l+m}  [\phi({e^+_k}) - \phi({e^-_k})] \\
& \leq |{e^+_l} - {e^-_{l+m+1}}|+\sum_{k=l+1}^{l+m}  |{e^+_k} - {e^-_k}|.   
\end{align*}

However, this contradicts the inequality in the assumption (L2).  
%As a result, for $\gamma-$a.e. $(x^+,x^-)$, we have $x^\pm \in \partial  C_i$, for some $1 \leq i \leq n$. Yet, thanks again to (C3) and Proposition \ref{Prop. convex case}, we infer that for $\gamma-$a.e. $(x^+,x^-) \in \partial  C_i \times \partial  C_i$, we have $x^+ \in
%\partial C_i$, for some $1 \leq i \leq n$.
%{
%{\chi^i_j}^+$ or \,$x^+ \in {\Gamma^i_j}^+$, for some $1 \leq j \leq n_i$, and $x^-=\bbT^i(x^+)$.
%Hence, our claim is proved {in the case of a finite index set $I_C$}. 

%Without loss of generality, one can assume now that $i_k=k$, for all $k \in \mathbb{N}$. If the number of sets $C_i$ is finite, then this yields directly to a contradiction. 

{\it Step 1.3.} Finally, it remains to consider %how a contradiction in 
the case when %the number of sets $C_i$ is infinite (i.e., 
$I_C=\mathbb{N}^\star$. {From now on, for the sake of simplicity of notation we assume that $i_k =k$.}
%, or equivalently that $\{e_k^\pm\}_k=\{e_1^\pm\}$.
%there must exist $m>1$ satisfying (\ref{r-m}). 

%\PR{TO BE REWRITTEN}
%Since $(e_1^+,e_2^-) \in \spt(\gamma)$, then 
Since we assumed that our claim does not hold,
%We recall that 
there is an arc $E_1^+ \subset \partial C_{1}$ and another one $E_2^- \subset \partial C_{2}$ with $\gamma(E_1^+ \times E_2^-)=c>0$. Set $\gamma_1=\gamma\res{(E_1^+ \times E_2^-)^c}$, $f_1^+=(\Pi_x)_{\#}\gamma_1$\, and \,$f_1^-=(\Pi_y)_{\#}\gamma_1$. Then, we have $$f_1^\pm(\partial\Omega)=f^\pm(\partial\Omega)-c.$$ 

Since $\partial C_{1} \subset \partial\Omega\setminus E_2^-$, then we also have 
$$f_1^-(\partial\Omega) \geq f^-(\partial C_{1})=f^+(\partial C_{1}) \geq f^+(E_1^+)\geq c.$$
Yet, $f(\partial C_{2})=0$. Then, there will be a set $E_2^+ \subset \partial C_{2}$
%:=\cup_{j \in J_2} E_{2,j}^+ \subset \partial C_{i_2}$
with another set $E_3^- \subset \cup_{k \in I_3}  \partial C_{k}$ %(where $I_3 \subset \mathbb{N}^\star$) %E_{3,k}^-\subset 
%\partial\Omega\backslash [\partial C_{i_1} \cup \partial C_{i_2}]$, where $E_{3,k}^- \subset \partial C_{i_k}$ for some $k \geq 3$, 
such that $\gamma(E_{2}^+ \times E_{3}^-)=c$.
Again, we define $\gamma_2=\gamma_1\res{(E_{2}^+ \times E_{3}^-)^c}$, $f_2^+=(\Pi_x)_{\#}\gamma_2$\, and \,$f_2^-=(\Pi_y)_{\#}\gamma_2$. Thanks to $\partial C_{1} \subset \partial\Omega\setminus (E_2^- \cup E_3^-)$, we have $$f_2^\pm(\partial\Omega)=f_1^\pm(\partial\Omega)-c=f^\pm(\partial\Omega)-2c 
%$$%and, s
%and
%and, since $\partial C_{1} \subset \partial\Omega\backslash (E_2^- \cup E_3^-)$, then 
%we have again that 
\qquad \mbox{and} \qquad f_2^-(\partial\Omega) \geq f^-(\partial C_{1}) \geq c.$$
%By induction, for every $n \in \mathbb{N}^\star$,
%Since $f(\partial C_{k})=0$ for all $k \in I_3$, then there will be a set $E_{3}^+ \subset \cup_{k \in I_3} \partial C_{k}$ and another set $E_{4}^- \subset \cup_{k \in I_4} \partial C_{k}$
%, where $E_{4,k,j}^- \subset \partial C_{i_j}$ for some $j$,
%such that $\gamma(E_{3}^+ \times E_{4}^-)=c$, where $I_4 \subset \mathbb{N}^\star$.
%%I_C\backslash (\{1\}\cup I_3)$.
%and $\sum_j c_{2,k,j}=c_{1,k}$. 
%Set $E_3^+:=\cup_{k,j} E_{3,k,j}^+$ and $E_4^-:=\cup_{k,j} E_{4,k,j}^-$.
%Set again $\gamma_3=\gamma_2\res{(E_{3}^+ \times E_{4}^-)^c}$,
%Thanks to the convexity of $\Omega$ and the fact that transport rays cannot intersect, we see that all these arcs are mutually disjoint. 
%\,$f_3^+=(\Pi_x)_{\#}\gamma_3$\, and \,$f_3^-=(\Pi_y)_{\#}\gamma_3$. Then, one has $$f_3^\pm(\partial\Omega)=f^\pm(\partial\Omega)-3c\qquad \mbox{and}\qquad
%and, the fact that $\partial C_{1} \subset \partial\Omega\backslash (E_2^- \cup E_3^- \cup E_4^-)$ implies again that 
%f_3^-(\partial\Omega) \geq f^-(\partial C_{1})\geq c.$$ 
Fix $n \geq 3$. By induction: since $f(\partial C_{k})=0$\, for all $k \in I_{n}$ then there will be a set $E_{n}^+ \subset \cup_{k \in I_n} \partial C_{k}$ and another set $E_{n+1}^- \subset \cup_{k \in I_{n+1}} \partial C_{k}$ %where \,$I_{n+1} \subset \mathbb{N}^\star$
%, where $E_{4,k,j}^- \subset \partial C_{i_j}$ for some $j$,
such that $\gamma(E_{n}^+ \times E_{n+1}^-)=c$.
%and $\sum_j c_{2,k,j}=c_{1,k}$. 
%Set $E_3^+:=\cup_{k,j} E_{3,k,j}^+$ and $E_4^-:=\cup_{k,j} E_{4,k,j}^-$.
We also define $\gamma_n=\gamma_{n-1}\res{(E_{n}^+ \times E_{n+1}^-)^c}$,
%Thanks to the convexity of $\Omega$ and the fact that transport rays cannot intersect, we see that all these arcs are mutually disjoint. 
\,$f_n^+=(\Pi_x)_{\#}\gamma_n$\, and \,$f_n^-=(\Pi_y)_{\#}\gamma_n$. Using that $\partial C_1 \subset \partial\Omega \setminus \cup_{i \in I_{n+1}} E_i^-$, we have  $$f_n^\pm(\partial\Omega)=f_{n-1}^\pm(\partial\Omega)-c=f^\pm(\partial\Omega)-nc \qquad \mbox{and} \qquad
%and, the fact that $\partial C_{1} \subset \partial\Omega\backslash (E_2^- \cup E_3^- \cup E_4^-)$ implies again that $$f_3^-(\partial\Omega) \geq f^-(\partial C_{1})\geq c.$$ 
%there will be a measure $f_n^-$ on $\partial\Omega$ with $f_n^-(\partial\Omega)=f^-(\partial\Omega) - nc$ and 
f_n^-(\partial\Omega) \geq f^-(\partial C_{1}) \geq c.$$  %sequence of disjoint sets $E_k^+$, $1 \leq k \leq n$, with $f^+(E_k^+)=c$ and, another sequence of disjoint sets $E_{k}^-$, $2 \leq k \leq n+1$, such that for each $1 \leq k \leq n$, $E_k^+$ is transported to $E_{k+1}^- \subset \partial\Omega\backslash\partial\Omega_1$. Finally, define
%\,$f_n^+=f^+\res[\partial\Omega\backslash \cup_{k=1}^n E_k^+ ]$\, and \,$f_n^-=f^-\res[\partial\Omega\backslash \cup_{k=2}^{n+1} E_k^-]$. Then, one has $f_n^\pm(\partial\Omega)=f^\pm(\partial\Omega)-n c$\, and, since $\partial\Omega_1 \subset \partial\Omega\backslash \cup_{k=2}^{n+1} E_k^-$ then $f_n^-(\partial\Omega) \geq f^-(\partial\Omega_1)\geq c$. Hence, $f^-(\partial\Omega) \geq (n+1)c$, for all $n \in \mathbb{N}$. 
But, this yields obviously to a contradiction as soon as $n$ is large enough. As a consequence, the claim that $\gamma(\bar C_i \times \bar C_j)=0$, for all $i \neq j$, is proved.

{\it Step 2.}
We claim that the restriction of $\gamma$ to $C_i \times C_i$ is the unique optimal transport plan between its corresponding marginals. %want to localize the problem to sets $C_i$, $i\in I_C$. 
For $i\in I_C$, we introduce
$$
\gamma_i := \gamma \LC  (%\partial 
\bar C_i\times \bar %\partial 
C_i),\qquad f^\pm_i := f^\pm \LC  \bar C_i.
$$
Thanks to our assumption on $\gamma$, we see that $(\Pi_x)_{\#}\gamma_i=f_i^+$ and $(\Pi_y)_{\#}\gamma_i=f_i^-$.
Let us suppose that $\eta_i$ is a solution to the Monge-Kantorovich problem on $C_i$ with data $f^+_i$, $f^-_i$, $i\in I_C$. We notice that since by assumption $C_i$ is convex and $g_i:=g\res \partial C_i$ satisfies condition (H1)-(H3), then due to Proposition \ref{Uniqueness in the convex case}, $\eta_i$ is the unique optimal transportation plan and it is induced by a map $T_i$, i.e. $\eta_i =(Id, T_i)_\# f_i^+$. 
%We first prove  (\ref{r-gamma}), then we justify  (\ref{r-gamma2}).
The optimality of $\eta_i$ on $\bar C_i \times \bar C_i$ implies that
\begin{equation}\label{r-gamma3}
\int_{\bar C_i\times \bar C_i}|x-y|\,d \eta_i
%=\int_{\partial C_i\times \partial C_i}|x-y|\,d \eta_i 
\le \int_{%\partial
\bar C_i\times %\partial
\bar C_i}|x-y|\,d \gamma_i.
\end{equation}
Now, we set $\eta = \sum_{i\in I_C} \eta_i$. The optimality of $\gamma$ between $f^+$ and $f^-$ as well as the admissibility of $\eta$ lead to
$$\int_{\bar\Omega\times \bar \Omega} |x-y|\,d\gamma \le\int_{\bar\Omega\times \bar \Omega} |x-y|\,d\eta.
$$
Yet, from \eqref{r-gamma3}, we also have
$$\int_{\bar\Omega\times \bar \Omega} |x-y|\,d\eta \leq  \int_{\bar\Omega\times \bar \Omega} |x-y|\,d\gamma.$$\\
%Since $f^\pm$ are supported on $\partial\Omega$ we see that 
%$$
%\int_{\bar\Omega\times \bar \Omega} |x-y|\,d\gamma =
%\int_{\partial\Omega\times \partial \Omega} |x-y|\,d\gamma .
%$$
%By the same token, since $f^\pm_i$, $i\in I_C$, are supported on $\partial\Omega\cap \partial C_i$ we see that
Hence, we get that the inequality in \eqref{r-gamma3} is in fact an equality for every $i\in I_C$, i.e. we have 
$$
\int_{\bar C_i \times \bar C_i} |x-y|\,d\eta_i=
\int_{\bar C_i\cap\bar C_i} |x-y|\,d\gamma_i.
$$
%Since $|f| (\partial C_i\cap \partial C_j)=0$ we also notice that
%$$
%L =\int_{\bar\Omega\times \bar \Omega} |x-y|\,d\eta 
%= \sum_{i\in I_C}\int_{\bar C_i \times \bar C_i} |x-y|\,d\eta_i.
%$$
%This follows from the fact that $\eta_i(\bar C_i\times \bar C_j)=0$ when $j\neq i$. Now, we use optimality of $\eta_i$ on $\bar C_i \times \bar C_i$,
%$$
%L \le \sum_{i\in I_C}\int_{\bar C_i \times \bar C_i}  |x-y|\,d\gamma_i
%\le \sum_{i,j\in I_C}\int_{\bar C_i \times \bar C_j} |x-y|\, d \gamma =\int_{\bar\Omega \times\bar\Omega} |x-y|\,d\gamma.
%$$
%The optimality of $\gamma$ implies that we have equalities not inequalities above. In particular this implies that
%$$
%\gamma (\bar C_i \times \bar C_j)=0, \qquad i\neq j.
%$$
This means that $\gamma_i=\eta_i=(Id, {\bf T}_i)_\# f^+_i$ is the unique optimal transport plan between $f_i^+$ and $f_i^-$. Our claim follows.
Yet, we have 
\begin{equation}\label{r-gamma}
\gamma = \sum_{i\in I_C} \gamma_i. 
\end{equation}
%Once we have (\ref{r-gamma})
Thus, we deduce that
\begin{equation}\label{r-gamma2}
\gamma=\sum_{i\in I_C} (Id, {\bf T}_i)_\# f_i^+ = (Id, {\bf T})_\# f.
\end{equation}

Hence, we have proved that if $\gamma$ is an optimal transport plan between $f^+$ and $f^-$, then we have $\gamma=(Id,{\bf{T}})_{\#}f^+$. Yet, the map ${\bf{T}}$ does not depend on $\gamma$ as a result, the optimal transport plan $\gamma$ is unique. This concludes the proof. $\qedhere$
\end{proof}

\begin{proposition} \label{existence & uniqueness of an optimal flow in the nonconvex case L}
Assume that (L1) and (L2) hold. 
%Let $\gamma:=(Id,{\bf{T}})_{\#}f^+$ be an optimal transport plan and $v_\gamma$ be the flow defined by \eqref{flow}. 
Then, we have \,$\eqref{Beckmann}=\eqref{Kantorovich}$. Moreover, $v_\gamma$ is a unique solution of Problem \eqref{Beckmann}, provided that $\gamma=(Id,{\bf{T}})_{\#}f^+$. 
\end{proposition}
\begin{proof}
We recall that we always have $\eqref{dual}\leq \eqref{Beckmann}$. %(even if the domain is not convex). %%%Thanks to (H2),
%Proposition \ref{Prop. nonconvex case},
%%the flow $v_\gamma$ defined in \eqref{flow} is well defined. 
In addition, we see that $v_\gamma$ is well defined thanks to the fact that $]x,{\bf{T}}(x)[ \subset \Omega$ (see (L1)), for $f^+-$a.e. $x$. Moreover, $v_\gamma$ is clearly admissible in Problem \eqref{Beckmann} and we have
$$\int_{\overline{\Omega}}|v_\gamma| =\int_{\overline{\Omega} \times \overline{\Omega}}|x-y|\mathrm{d}\,\gamma=\eqref{Kantorovich}=\eqref{dual}\leq \eqref{Beckmann}.$$
Hence, $v_\gamma$ solves Problem \eqref{Beckmann} and, we have $\eqref{Beckmann}=\eqref{Kantorovich}$. It is worth noting that although $\Omega$ is not convex, we have proved that the values of the infima of Problems \eqref{Beckmann} and \eqref{Kantorovich} are exactly the same. Thanks to this fact, following the argument in \cite[Proposition 2.6 (3)]{SD} for non-convex domains, we can adapt the result in \cite[Theorem 4.13]{Santambrogio} 
%and the fact that $\min\eqref{Beckmann}=\min\eqref{Kantorovich}$, it is not difficult 
to deduce that if $v$ is an optimal vector field for Problem \eqref{Beckmann}, then there will be an optimal transport plan $\gamma^\prime$ for Problem \eqref{Kantorovich} such that $v=v_{\gamma^\prime}$ %(we also refer the reader to for more details).
Yet, by Proposition \ref{Prop. nonconvex case}, the optimal transport plan $\gamma$ is unique and so, the solution $v_\gamma$ of Problem \eqref{Beckmann} is unique as well. $\qedhere$
%Moreover, thanks again to Proposition \ref{Prop. nonconvex case}, we see that 
\end{proof}
%Assume that this is not the case. This means that $e_{i_k}^\pm \notin {\chi^{i_1}}^\pm \cup {\Gamma^{i_1}}^\pm$, for all $k \geq 2$. Then, there is a point $e_1^- \in  {\chi^{i_1}}^- \cup {\Gamma^{i_1}}^-$ and  another sequence of points $\{g_{j_k}^\pm\}$ such that $g_{j_{k}}^\pm \in {\chi^{j_k}}^\pm \cup {\Gamma^{j_k}}^\pm$, where $[{\chi^{i_k}}^\pm \cup {\Gamma^{i_k}}^\pm] \cap [{\chi^{j_{k^\prime}}}^\pm \cup {\Gamma^{j_{k^\prime}}}^\pm]=\emptyset$, for all $k,\,k^\prime$, and $(e_1^-,g_{j_2}^+),\,(g_{j_k}^-,g_{j_{k+1}}^+) \in \spt(\gamma)$, for all $k \geq 2$. 
%(i.e. we have $i_k \neq i_1$, for any $k \geq 2$). So, there is no point on ${\chi^{i_k}}^+ \cup {\Gamma^{i_k}}^+$ that will be transported to  ${\chi^{i_1}}^- \cup {\Gamma^{i_1}}^-$. However, this is not possible since it means that  ${\chi^{i_1}}^+ \cup {\Gamma^{i_1}}^+$ has to be transported to  ${\chi^{i_1}}^- \cup {\Gamma^{i_1}}^-$, where $f^+({\chi^{i_1}}^+ \cup {\Gamma^{i_1}}^+)=f^-({\chi^{i_1}}^- \cup {\Gamma^{i_1}}^-)$ but at the same time a part \,$E_1^+$ of \,${\chi^{i_1}}^+ \cup {\Gamma^{i_1}}^+$ is also transported outside ${\chi^{i_1}}^- \cup {\Gamma^{i_1}}^-$.  $\qedhere$
%\end{proof}

\begin{theorem} \label{Theorem 4.5}
Under the assumptions (L1) and (L2), there exists a unique solution $u$ to Problem \eqref{LGP} provided that 
%$g \in BV(\partial\Omega) \cap C(\partial\Omega)$ and 
$g\in W^{1,1}(\partial\Omega)$ is piecewise monotone.% (see Definition \ref{d-ps-mono}). 
\end{theorem}
\begin{proof}
By Proposition \ref{existence & uniqueness of an optimal flow in the nonconvex case L},  $v_\gamma$ (where $\gamma=(Id,{\bf{T}})_{\#}f^+$) is the unique optimal flow in Problem \eqref{Beckmann}. Yet, thanks to (L1), it is clear that $|v_\gamma|(\partial\Omega)=0$. Then, Proposition  \ref{t-equiv} concludes the proof. 
%\end{proof}
%\begin{proposition} \label{uniqueness of the optimal flow}
%Under the assumptions (C1), (C2) $\&$ (C3), the solution of Problem \eqref{Beckmann} is unique.
%when $f^+$ or $f^-$ is nonatomic.
%\end{proposition}
%\begin{proof}
%\end{proof}
%\begin{theorem}
%Under the assumptions (C1), (C2) $\&$ (C3), the BV least gradient problem \eqref{LGP} has a unique solution $u$  provided that $g \in BV(\partial\Omega) \cap C(\partial\Omega)$.
%\end{theorem}
%\begin{proof}
%This follows immediately from \cite[Lemma 3.4]{DweGor} and Proposition \ref{uniqueness of the optimal flow}.
%, and the fact that $f$ is nonatomic since $g \in C(\partial\Omega)$. 
\end{proof}
\begin{example}\rm
Set $\Omega=\Omega^+\setminus \Omega^-$ with $\Omega^+=[-1,1]\times [0,1]$ and $\Omega^-=[-a,a]\times [0,b]$, with $0<a,\,b<1$. We define the boundary data $g$ on $\partial \Omega$ as follows:
$$g(x_1,x_2)=\begin{cases} 
0 \quad &([a,1]\times \{0\})\cup (\{1\}\times [0,b])\cup (\{-1\}\times [0,b])\cup ([-1,-a]\times \{0\}),\\ 
\dfrac{b-x_2}{1-b} & (\{1\}\times[b,1])\cup (\{-1\}\times [b,1]), \\
-1 & ([-1, 1]\times \{1\})\cup ([-a,a]\times \{b\}), \\
-\dfrac{x_2}{b} & (\{-a\}\times [0,b])\cup (\{a\}\times [0,b]).
\end{cases} $$\\
In this case, $f^+= f^+_1 + f^+_2$\, and \,$f^-= f^-_1 + f^-_2$, where
$$
f^+_1 =\frac{1}{b} \cH^1\LC (\{a\} \times [0,b]),\qquad f^-_1=\frac{1}{1-b}\cH^1\LC (\{1\} \times [b,1]),
$$
$$
f^+_2=\frac{1}{1-b}\cH^1\LC (\{-1\} \times [b,1]),\qquad 
  f^-_2=\frac{1}{b}\cH^1\LC (\{-a\} \times [0,b]).
$$
\begin{figure}[h]
\begin{tikzpicture}
  \draw (0, 0) --(2,0)--(2,2)--(5,2)--(5,0)--(7,0)--(7,5)--(0,5)--(0,0);

  \node at (2, -.3) {$(-a,0)$};
  \node at (5, -.3) {$(a,0)$};
  \node at (2,2.2) {$(-a,b)$};

  \fill (2,0.3) circle (2pt); \node at (2.3,0.3) {$e_2^-$};
    \fill (0,2.45) circle (2pt); \node at (-0.3,2.45) {$e_2^+$};
  \fill (5,0.3) circle (2pt); \node at (4.7,0.3) {$e_1^+$};
     \fill (7,2.45) circle (2pt); \node at (7.3,2.45) {$e_1^-$};

   \draw[blue] (2,0.3)--(0,2.45); \draw[blue] (5,0.3)--(7,2.45);
   \draw[green] (2,0.3)--(5,0.3); \draw[green] (0,2.45)--(7,2.45);

   \node[rotate=90,anchor=south] at (0,3.4){{\tiny$f_2^+=\dfrac{1}{1-b}$}};
   \node[rotate=90,anchor=south] at (7.8,3.4) {{\tiny $f_1^-=\dfrac{1}{1-b}$}};
   \node[rotate=90,anchor=south] at (2.7,1.2) {{\tiny $f_2^-=\dfrac{1}{b}$}};
   \node[rotate=90,anchor=south] at (5,1) {{\tiny $f_1^+=\dfrac{1}{b}$}};

\fill[gray,opacity=0.5] (0,0)--(2,0)--(0,2); 
\fill[gray,opacity=0.5] (5,0)--(7,0)--(7,2); 
\fill[gray,opacity=0.5] (0,5)--(2,2)--(5,2)--(7,5);

\node at (3.5,3.5) {$X_1$};
\node at (0.5,0.5) {$X_2$};
\node at (6.5,0.5) {$X_3$};

%(2, s 2), (0, 2+s 3) s=0.25

\fill[red,opacity=0.1] (0,2)--(0,2.75)--(2,0.5)--(2,0); 
\fill[red,opacity=0.2] (0,2.75)--(0,3.5)--(2,1)--(2,0.5); 
\fill[red,opacity=0.3] (0,3.5)--(0,4.25)--(2,1.5)--(2,1); 
\fill[red,opacity=0.4] (0,4.25)--(0,5)--(2,2)--(2,1.5); 

\fill[red,opacity=0.5] (5,0)--(5,0.5)--(7,2.75)--(7,2); 
\fill[red,opacity=0.6] (5,0.5)--(5,1)--(7,3.5)--(7,2.75); 
\fill[red,opacity=0.7] (5,1)--(5,1.5)--(7,4.25)--(7,3.5); 
\fill[red,opacity=0.8] (5,1.5)--(5,2)--(7,5)--(7,4.25);

% \draw (-1, 0) rectangle (1, 1);
\end{tikzpicture}
\caption{Rectilinear C-shape example}
\label{fig C shape}

\end{figure}
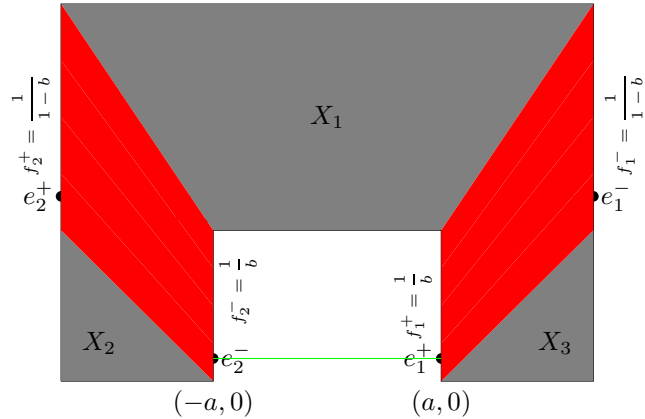 

In order to prove existence of a solution to Problem \eqref{LGP}, we subdivide our region $\Omega$ into sets satisfying conditions (L1) and (L2). As shown in Figure \ref{fig C shape}, we set $X_1$ to be the trapezoid with vertices $(-a,b), (a,b), (-1,1), (1,1)$, $X_2$ the triangle with vertices $(-a,0), (-1,b), (-1,0)$, 
and $X_3$ the triangle with 
vertices $(a,0)$, $(1,b)$ and 
$(1,0)$. We want to subdivide 
$C:=\Omega\setminus ({\bar X_1}\cup 
\bar X_2\cup \bar X_3)$ into convex 
sets $C_i$'s (in red in Figure 
\ref{fig C shape}) satisfying 
condition (L2) and such that each $C_i$ satisfies (H1)-(H3). Notice that for $s\in [0,1]$, we have
$$g(a,sb)=g(-a,sb)=g(1,b+(1-b)s)=g(-1,b+(1-b)s)=-s.$$
After taking a 
partition of $[0,1]$, $0=s_0< s_1<\cdots <s_n=1$
and $\Delta 
s_i=s_i-s_{i-1}$, we construct convex domains $C_{i,l/r}$ as follows:
$C_{i, l}$ (resp. $C_{i,r}$) is
an open
quadrilateral in red with 
vertices $(- a, s_{i-1}b),$ 
$(- 1, b+ (1-b)s_{i-1}),$ $(- 1, 
b+(1-b)s_i),$ $(- a,s_i b)$ (respectively, the 
quadrilateral in red with 
vertices $( a, s_{i-1}b),$ 
$(1, b+ (1-b)s_{i-1}),$ $( 1, 
b+(1-b)s_i),$ $(a,s_i b)$). As a result $C=\left((\cup_{i=1}^n \bar C_{i,l})\cup (\cup_{i=1}^n\bar C_{i,r})\right)^\circ$.

Notice that 
$\partial C_{i,r} \cap \partial\Omega$ 
can be decomposed into two arcs:
\begin{align*}
&{\Gamma_{i,r}}^+=(a, s_{i-1}b\arc{)\ (}a,s_{i}b),\quad {\Gamma_{i,r}^-}
=(1, b+(1-b)s_{i-1}\arc{)\ (}
1, b+(1-b)s_i).
%&F_{1,r}^i=(a,s_{i-1}b\arc{)\ (}1,b+(1-
%b)s_{i-1})\quad\hbox{and}\quad
%F_{2,r}^i=(a,s_{i}b\arc{)\ (1},b+(1-
%b)s_{i},
\end{align*}
%with boundary data on $F_{1,r}^i,$ $F_{2,r}^i$ as 
%follows $g(F_{1,r}^i)=s_{i-1}$, and $g(F_{2,r}^i)=s_i$. 
By symmetry, we decompose $C_{i,l}$ in similar way. Then, it
follows that all convex domains 
$C_{i,l/r}$ satisfy condition 
(H1), concluding (L1). 

Next, we check (L2). Once w set %Letting 
$e_1^{+}=(a, sb), e_1^-=(1,b+(1-b)s)$, $e_2^-=(-a,sb)$, $e_2^+=(-1,b+(1-b)s)$ as in Figure \ref{fig C shape}, Condition (L2) implies that 
$$
2 |(1-a,b+s(1-2b))|=
|e_1^+-e_1^-|+|e_2^+-e_2^-|< |e_1^+-e_2^-|+|e_2^+-e_1^-|=2+2a.$$
For this inequality to hold we require that %is then necessary to have
%$$|(1-a,b+s(1-2b)|<a+1,$$ or equivalently
$$|s(2b-1)-b|<2\sqrt{a},$$
for every $s\in [0,1].$
Notice that the maximum of the left-hand side is $\max(b, 1-b)$. Then, assuming that 
\begin{equation}\label{max b 1-b}
\max(b,1-b)<2\sqrt{a},
\end{equation}
(L2) follows for $e_i^\pm$ corresponding to transport rays. Since the inequality in \eqref{max b 1-b} is strict, then we choose $\Delta s_i$ small enough so that the convex sets 
$C_{i,l/r}$ satisfy inequality (L2) 
%for 
%every $i$, and 
for every sequence of points $\{e_k^\pm\}$ arbitrary on $\partial C_{i_k,l/r}$.
Hence, under the assumption \eqref{max b 1-b} and thanks to Theorem \ref{Theorem 4.5}, Problem \eqref{LGP} has a unique solution. 
%\PR{[HERE, WE NEED AN EXISTENCE THEOREM IMPLIED BY (L1) AND (L2)!]}
\end{example}

Here comes the most general instance of data we consider in this paper. Namely, we will extend the result of Proposition \ref{Prop. nonconvex case} to the case when the boundary datum $g$ is monotone on some arcs with negative curvature. %um on non-convex domains. 
%Of course, the idea is to approximate a new type of data by those %with boundary datum on  other approximated non-convex domains,
%satisfying assumptions (C1), (C2), (C3) $\&$ (C4). 
On the way, we need to introduce the
%\PR{I don't understand the function of this text}
following assumptions.\\ \\
$\bullet$\,\,\,{\bf{Condition (A1)}}. The domain $\Omega$ can be decomposed into disjoint sets $C_i$ ($ i \in I_C$), $E_i$ ($i \in I_E$) and, $X_i$ ($i \in I_X$)
such that $g$ is constant on $\partial X_i \cap \partial\Omega$ and, we have the following:\\

(1) \,\, For every $i \in I_C$, $C_i$ is convex. In addition, the family of open sets $\{C_i^\circ: i\in I_C\}$ satisfies the assumption (L1). \\ %and $\partial C_i \cap \partial\Omega$\, can be decomposed into open arcs $\Gamma_j^i:={\Gamma_j^i}^+\cup {\Gamma_j^i}^-$ ($j\in I^i_\Gamma$)\, or \,$\chi_j^i:={\chi_j^i}^+ \cup {\chi_j^i}^-$ ($j \in I^i_\chi$) and $F^i$ such that:\\\\

(2) \,\,For every \,$i \in I_E$, $\partial E_i \cap \partial\Omega$\, is the sum of %can be decomposed into 
two open arcs \,$E_i^+$\, and \,$E_i^-$ such that at least one of them is not convex and, 
%(we assume that all these arcs ${\chi_j^i}^\pm$ and ${\Gamma_j^i}^\pm$ are open)
%$\bullet$\,\,\,\,%$\Gamma_i^+\cup\Gamma_i^-$ is not contained in a line segment, 
%For every $j$, $\mbox{dist}({\Gamma_j^i}^+,{\Gamma_j^i}^-)>0$,
%the arc $\Gamma_i:=\Gamma_i^+ \cup \Gamma_i^-$ is strictly convex,
$g$ is strictly increasing on ${E_i}^+$ and it is strictly decreasing on ${E_i}^-$ with $TV(g_{|{E_i}^+}) =  TV(g_{|{E_i}^-})$.  %\\\\ 
\begin{remark}
Now, it is clear that the domain $D$ defined in Remark \ref{re-doD} may satisfy (A1) for a proper choice of $g$, even if it is not constant on the arc $\{ (x_1, x_2):\ x_1 \in (-1,1), \ x_2= \sqrt{1-x_1^2}\}$.
%for a proper choice of $g$.
\end{remark}

In order to construct a solution to (\ref{LGP}), we apply our usual strategy, namely, we introduce a transport map $\bbT$. First, we  define $\bar\bbT: {\Gamma}^+ \cup {\chi}^+ \mapsto {\Gamma}^- \cup {\chi}^-$. Since condition (A1) implies that $f(\bigcup_{i\in I_E}{E}_i^\pm) =0,$ 
%\PR{[WE MUST TAKE THE CLOSURE OF $E_i$ FIRST, OTHERWISE THE CLAIM IS TRIVIALLY TRUE, BECAUSE SETS $E_i$ ARE DISJOINT FROM THE SUPPORT OF $f$.]} 
the conclusion of Lemma \ref{zer-bry} is valid. Hence, for any $x^+\in  {\Gamma}^+ \cup {\chi}^+ $, we may define $\bar\bbT(x^+)$ by formula  (\ref{df-TTi}). A new situation arises when we want to define the transportation map on $E^+_i$, $i\in I_E$. We proceed as follows:
%For every $i \in \cI_C$, we define the map $\bbT^i: {\Gamma^i}^+ \cup {\chi^i}^+ \mapsto {\Gamma^i}^- \cup {\chi^i}^-$ as before and we extend it now to $E_i^+ \mapsto  E_i^-$, \,$i \in I_E$, by setting:
%the %last assumption (H3), we need to define
%the map $\bf{t}$ on $\chi^+ \cup \Gamma^+$ as follows (in the sequel, 
%(we will always denote by ${\bbT^i}^{[-1]}$ the inverse map of $\bf{T^i}$):
$$
\mbox{If }\,\,x^+ \in E_i^+,\mbox{ then we set }\,
\tilde{\bbT}^i(x^+):=
x^-, %\,\,\,\,\mbox{for all}\,\,\,\,\,\,
\mbox{  where}\,\,\,x^- \in E_i^-\,\,\,\mbox{is such that}\,\,\,f(\arc{\,x^+\,x^-})=0.
%,\\
%x^-\in {\Gamma_j^i}^- & \mbox{if}\,\,\,\,x^+ \in {\Gamma_j^i}^+\,\,\,\,\,\mbox{and}\,\,\,\,\,f(\arc{\,x^+\,x^-})=0,
%\end{cases}
$$
This map is well defined, thanks to the fact that $g$ is strictly monotone on $E^-_i.$ Finally, we define $\bbT: 
%\PR{{\Gamma}^+ \cup {\chi}^+ \to} 
\Gamma^+\cup \chi^+\cup (\cup_{i\in I_{E}} E_i^+)\mapsto 
\Gamma^-\cup \chi^-\cup (\cup_{i\in I_{E}} E_i^-)$ as follows:
\begin{equation}\label{df-bartildeT}
\bbT(x^+)=
\begin{cases}
\bar \bbT(x^+) & \hbox{for }\,\,\,x^+ \in {\Gamma}^+ \cup {\chi}^+ ,\\
\tilde{\bbT}^i(x^+) &  \hbox{for }\,\,\, x^+ \in E_i^+,\ i\in I_E.
\end{cases}
\end{equation}

Subsequently, we will denote by 
%${\bf{T}}$ the map on $\spt(f^+)$ such that its %restriction of ${\bf{T}}$ to ${\Gamma^i}^+ \cup {\chi^i}^+$, $i \in I_C$ (resp. $E_i^+$, $i \in I_E$) is the map ${\bf{T}}^i$ (resp. $\tilde{\bbT}^i$). We also denote by 
${\bf{T}}^{[-1]}$ the inverse map of ${\bf{T}}$.\\

%where $\arc{\,x^+x^-}$ is any of the two arcs of \,$\partial\Omega$\, joining $x^+$ and $x^-$.
%The structure of the decomposition of $\partial C_i$ into arcs of different types assures us that $\bbT^i$ defined on a part of $\partial C_i$ takes value in $\partial C_i$.
%\bigskip
%Here come our next requirements on $\Omega$ and $g$:\\

Now, we continue stating further conditions on the data.\\ \\
$\bullet$ \,\,{\bf{Condition (A2)}}.\,\,
%For every $x^+ \in {\chi_j^i}^+$ (resp. $x^+ \in {\Gamma_j^i}^+$),
%and $y \in \chi_i^-$ (resp. $\Gamma_i^-$),
%we have $]x^+,\bbT^i(x^+)[ \subset \Omega$. In addition,
For all points $x^+ \in {E_i}^+$\, and \,$x^- \in {E_i}^-$,
%and $y \in \chi_i^-$ (resp. $\Gamma_i^-$),
we have $]x^+,x^-[ \subset \Omega$. We note that $x^-$ here is an arbitrary point on $E_i^-$, so it is not necessarily the image of $x^+$ by $\tilde{\bbT}^i$ (see the difference with (H2)).\\\\
$\bullet$ \,\,{\bf{Condition (A3)}}. For any two finite sequences of points $\{e_k^\pm\}_{1 \leq k \leq m}$
%and
%$\{{e_k}^-\}_{1 \leq k \leq m}$
%(where $m \in \mathbb{N}$)
such that ${e^\pm_{k}} \in {\chi^{i_k}}^\pm \cup {\Gamma^{i_k}}^\pm$ (for some $i_k \in I_C$) or \,${e^\pm_{k}} \in {E^\pm_{i_k}}$ (for some $i_k \in I_E$), with $i_k \neq i_{k^\prime}$ for all $k \neq k^\prime$ such that $\{i_k,\,i_{k^\prime}\} \subset I_C$ or $\{i_k,\,i_{k^\prime}\} \subset I_E$, 
%${e_{k}}^\pm \in \spt(f^\pm) \cap [\partial C_{i} \cup \partial E_{i}]$,
%\PR{I LIKE THIS WAY OF WRITING, POSSIBLY WE ADOPT IT EVERYWHERE.}
%, for some $ i_k \in \cI_C \cup \cI_E$, 
 %for all $1 \leq k \leq m$ and $1 \leq r \leq r_k$,
we have the following inequality: %\marginpar{\textcolor{black}{inconsistency}}
$$
\sum_{k=1}^m  |{e_k}^+ - {e_k}^-| <
%\sum_{k=1}^m\bigg[|{e_{j,n_{i,j}}^k}^+ - {e_{s_2(i,j),1}^{s_1(i,j)}}^-| +
%|{e_{j}^i}^+ - {\bf{t^{s(i)}}}({e_{s^i(j)}^{s(i)}}^+)| +
%|{e_{j,1}^i}^+ - {\bf{t^{i}}}({x_{j,1}^{i}})| +
\sum_{k=1}^{m-1} |e^+_{k} - e^-_{k+1}| + |e_{m}^+ - e_{1}^-|. $$
Notice that this condition (A3) is just a generalization of the assumption (L2), since now we need also to guarantee that every set $E_i^+$ is transported to $E_i^-$, for all $i \in I_E$.
%$$+\sum_{r=1}^{r_m-1}|{e_{r}^m}^+ - {\bf{T^m}}({e_{r+1}^{m}}^+)|+|{e_{r_m}^m}^+ - {\bf{T^1}}({e_{r_1}^{1}}^+)|.$$
% For every $ i \in \cI_C$, we also assume that for any sequence of points $\{e_{k}^+\}_{1 \leq k \leq m}$ (where $m \in \mathbb{N}$) such that $e_{k}^+ \in {\chi_{j_k}^i}^+ \cup {\Gamma_{j_k}^i}^+$, for some $j_k \in  I^i_\chi \cup I^i_\Gamma$, we have the following inequality:
%$$
%\sum_{k=1}^m |e_k^+ - \bbT^i(e_k^+)|
%\,\,<\,\,
%\sum_{k=1}^{m-1}
%|e_k^+ - \bbT^i(e_{k+1}^+)| + |e_m^+ - {\bf{T^i }}(e_{1}^+)|.$$ 
% Again, let us denote by $S$ the set of singularity points on $\spt(f)$. Then, we assume %say that the (H) holds if we have 
%the following:
%$$
%S\,\,\,\mbox{is countable but not dense anywhere}.\leqno({\rm C4})
%$$
\begin{proposition}\label{Prop. nonconvex case general} 
Suppose that conditions (A1), (A2) and (A3) are satisfied. 
%Assume 
%$\Omega \subset \bR^2$ is open and bounded. Let us assume 
%that 
%the couple $(\Omega, g)$ satisfies 
%conditions (L1) $\&$ (L2) hold.
%Then, all the transport rays between $f^+$ and $f^-$ lie inside $\Omega$. More precisely, if 
Let \,$\gamma$\, be an optimal transport plan between $f^+$ and $f^-$. Then, for \,$\gamma-$a.e. $(x^+,x^-)\in \spt (\gamma)$, 
%there exists \,$i\in I_C$ (resp. $i\in  I_E$) such that 
we have $x^-={{{\bbT}}}(x^+)$. %(resp. $x^-=\tilde{\bf{T}}^i(x^+)$). 
In other words, $\gamma=(Id,{\bf{T}})_{\#}f^+$ is the unique optimal transport plan.
%Then, $\gamma:=(Id,{\bf{T}})_{\#}f^+$ is the unique optimal transport plan between $f^+$ and $f^-$. \PR{THE STATEMENTS OF THIS AND THE PREVIOUS PROPOSITIONS LOOK QUITE DIFFERENT.}
%$g\in BV(\partial\Omega) \cap C(\partial\Omega)$ and set $f=\partial_\tau g$, where $\partial_\tau g$ is defined by (\ref{r1}). Then, all the transport rays between $f^+$ and $f^-$ lie inside $\Omega$. More precisely, if 
%Let \,$\gamma$\, be an optimal transport plan between $f^+$ and \,$f^-$. Then, for \,$\gamma-$a.e. $(x^+,x^-)$, we either have that \,$x^+ \in {\chi^i_j}^+$ and \,$x^- \in {\chi^i_j}^-$ or \,$x^+ \in {\Gamma^i_j}^+$ and \,$x^- \in {\Gamma^i_j}^-$, for some \,$i \in \cI_C$\, and \,$j \in \cI_\chi^i \cup \cI_\Gamma^i$, or that \,$x^+ \in E_i^+$ and \,$x^- \in E_i^-$, for some $i \in \cI_E$.
\end{proposition}

\begin{proof}
Notice that by definition, we have $f(\partial E_i) =0$, for all $i\in I_E$. As a result, the argument of Lemma \ref{zer-bry} yields that $f(\partial C_i)=0$, for all $i\in I_C$.
%{We set 
%$$
%\bar f = f \LC (\bigcup_{i\in I_C} \partial C_i), \qquad \qquad
%\tilde  f_i = f \LC  \partial %E_i\,\,\,\,\,\,\,\,\,(i \in I_E).
%$$
%Fix $x_0\in \partial\Omega$, then we define
%$$
%\bar g(x) = \bar f(\arc{x_0 x}),\qquad\,\,\,\,\,\, \qquad
%\tilde g_i(x) = \tilde f_i(\arc{x_0 x})\,\,\,\,\,\,\,\,\,(i \in I_E).
%$$
%Then, we see that $\bar g$ satisfies  conditions (L1) and (L2) with the following family of sets $\bar X_i$, $i\in I_{\bar{X}}= I_X \cup I_E$,
%$$
%\bar X_i = \left\{
%\begin{array}{ll}
 %   X_i & \,\,\, i\in I_X, \\
 %   E_i &  \,\,\, i\in I_E.
%\end{array}\right.
%$$\\

The assumption (A3) here corresponds to condition (L2), which played the key role in the proof of Proposition \ref{Prop. nonconvex case}. Hence, the same argument used there
%and following the lines of the proof of Proposition \ref{Prop. nonconvex case} 
shows that $\gamma(C_i \times E_j)=0$, $i \in I_C$ and $j \in I_E$ (resp. $\gamma(C_i \times C_j)=\gamma(E_i \times E_j)=0$, for all $i \neq j \in I_C$ or $i \neq j \in I_E$). Similarly to Proposition \ref{Prop. nonconvex case}, one can see that $\gamma \res [C_i \times C_i]$ (resp. $\gamma \res [E_i \times E_i]$) is an optimal transport plan between its own marginals $\Bar{f}_i^+:=f^+\res \partial C_i$ and $\Bar{f}_i^-:=f^-\res\partial C_i$ (resp.  $\tilde{f}_i^+:=f^+\res \partial E_i$ and $\tilde{f}_i^-:=f^-\res \partial E_i$), for all $i \in I_C$ (resp. $i \in I_E$). From Proposition \ref{Prop. convex case}, we infer that 
 \begin{equation} \label{eq.4.5.1}
 \gamma \res [C_i \times C_i]=(Id, \bar T)_\# \bar f_i^+,\,\,\,\,\,\mbox{for every}\,\,\,i \in I_C.
 \end{equation}
% there is a unique optimal transport plan $\Bar{\gamma}$ between $\Bar{f}^+$ and $\Bar{f}^-$ and it has the form $\bar \gamma = (Id, \bar T)_\# \bar f^+$.  
%Let us now turn our attention to $\tilde g_i$, $i \in I_E$. 
%\PR{Proceeding as in the proof of} %Similarly to 
%Lemma \ref{lyy}, \PR{IS THIS THE CORRECT REFERENCE?} 
Now, we claim that for all $(x^+,x^-) \in \spt(\gamma) \cap (E_i^+ \times E_i^-)$ we have \,$x^-=\Tilde{T}(x^+)$. Assume that this is not the case. Then, there must be a couple $(x_1^+,x_1^-) \in \spt(\gamma) \cap (E_i^+ \times E_i^-)$ such that $x^+ < x_1^+$ and $x_1^->x^-$, since otherwise we get a contradiction with the mass balance. Thanks to  assumption (A2), we see that the transport rays $]x^+,x^-[$ and $]x_1^+,x_1^-[$ intersect, which is a contradiction. Then, we also have the following:
%there is a unique optimal transport plan $\Tilde{\gamma}_i$, between $\Tilde{f}_i^+$ and $\Tilde{f}_i^-$, and one has 
%It is  clear that it satisfies  conditions (A1), (A2) and (A3). Now, we may apply the argument we used in the proof of Proposition \ref{Prop. nonconvex case} to sets $E_i$, $i\in I_E$ instead of ${\Gamma^i}^+$, $i\in I_C$, and ${\chi^i}^+$, $i\in I_C$. We conclude that a unique optimal transport plan $\tilde\gamma$ has the following form 
\begin{equation} \label{eq.4.5.2}
\gamma \res [E_i \times E_i] = (Id, \tilde\bbT)_\# \tilde f_i^+,\,\,\,\,\,\mbox{for every}\,\,\,i \in I_E.
\end{equation}
Combining \eqref{eq.4.5.1} and \eqref{eq.4.5.2}, we get that 
%In order to conclude the proof, we just need to show that 
%\PR{Now, we claim that 
$\gamma = (Id, \bbT)_\#f^+$. Hence, $\gamma$ is unique {because it is supported on uniquely defined graph of $\bbT$} and its first marginal equals $f^+$. This concludes the proof. 
\end{proof}

\begin{proposition} \label{existence & uniqueness of an optimal flow in the nonconvex case}
Assume that (A1), (A2) $\&$ (A3) hold. 
%Let $\gamma:=(Id,{\bf{T}})_{\#}f^+$ be an optimal transport plan and $v_\gamma$ be the flow defined by \eqref{flow}. 
Then, Problems \eqref{Beckmann} and \eqref{Kantorovich} have the same minimal value. Moreover, $v_\gamma$ (where $\gamma=(Id,{\bf{T}})_{\#}f^+$) is the unique optimal flow in Problem \eqref{Beckmann}. 
\end{proposition}
\begin{proof}
The proof is exactly the same as the one for Proposition \eqref{existence & uniqueness of an optimal flow in the nonconvex case L}. $\qedhere$
\end{proof}
Finally, we get the following result:
\begin{theorem} \label{Theorem 4.7}
Assume that (A1), (A2) $\&$ (A3) hold and that $g$  %\cap C(\partial\Omega)$ 
is piecewise monotone. Then, Problem \eqref{LGP} has a unique solution. 
\end{theorem}
\begin{proof}
This will  immediately follow from Propositions \ref{existence & uniqueness of an optimal flow in the nonconvex case} $\&$ \ref{t-equiv} once we show that $|v_\gamma|(\partial\Omega)=0$. Indeed, this is implied by the fact that $\gamma=(Id,{\bf{T}})_{\#}f^+$% (see Proposition \ref{existence & uniqueness of an optimal flow in the nonconvex case})
and by the assumptions (H2), (A2).
%\end{proof}
%\begin{proposition} \label{uniqueness of the optimal flow}
%Under the assumptions (C1), (C2) $\&$ (C3), the solution of Problem \eqref{Beckmann} is unique.
%when $f^+$ or $f^-$ is nonatomic.
%\end{proposition}
%\begin{proof}
%\end{proof}
%\begin{theorem}
%Under the assumptions (C1), (C2) $\&$ (C3), the BV least gradient problem \eqref{LGP} has a unique solution $u$  provided that $g \in BV(\partial\Omega) \cap C(\partial\Omega)$.
%\end{theorem}
%\begin{proof}
%This follows immediately from \cite[Lemma 3.4]{DweGor} and Proposition \ref{uniqueness of the optimal flow}.
%, and the fact that $f$ is nonatomic since $g \in C(\partial\Omega)$. 
\end{proof}
%\PR{[THE FOLLOWING EXAMPLE IS AN ILLUSTRATION OF CONDITIONS (L1-L2), HENCE IT SHOULD APPEAR BEFORE THM 4.7.]}
%\begin{example}\rm
%Set $\Omega=\Omega^+\setminus \Omega^-$ with $\Omega^+=[-1,1]\times [0,1]$, and $\Omega^-=[-a,a]\times [0,b]$, with $0<a,b<1$. We define the boundary data $g$ on $\partial \Omega$ as follows:
%$$g(x_1,x_2)=\begin{cases} 
%0 \quad &([a,1]\times \{0\})\cup (\{1\}\times [0,b])\cup (\{-1\}\times [0,b])\cup ([-1,-a]\times \{0\}),\\ 
%\dfrac{b-x_2}{1-b} & \{1\}\times[b,1]\cup \{-1\}\times [b,1], \\
%-1 & [-1, 1]\times \{1\}\cup [-a,a]\times \{b\}, \\
%-\dfrac{x_2}{b} & \{-a\}\times [0,b]\cup \{a\}\times [0,b].\\
%\end{cases}. $$
%In this case, $f^+= f^+_1 + f^-_1 + f^+_2+ f^-_2$, where
%$$
%f^+_1 =\frac{1}{b} \cH^1\LC \{a\} \times [0,b],\quad f^-_1=\frac{1}{b}\cH^1\LC \{-a\} \times [0,b],
%$$
%$$
%f^+_2=\frac{1}{1-b}\cH^1\LC \{-1\} \times [b,1],\quad 
%f^-_2=\frac{1}{1-b}\cH^1\LC \{1\} \times [b,1].
%$$\begin{figure}[h]
%\begin{tikzpicture}
%  \draw (0, 0) --(2,0)--(2,2)--(5,2)--(5,0)--(7,0)--(7,5)--(0,5)--(0,0);
%  \node at (2, -.3) {$(-a,0)$};
 % \node at (5, -.3) {$(a,0)$};
 % \node at (2,2.2) {$(-a,b)$};
%  \fill (2,0.3) circle (2pt); \node at (2.3,0.3) {$e_2^-$};

Now, we show how the set of assumptions (A1)-(A3) works.

\begin{example}\rm
Here, we present an example of data $(\Omega, g)$, where a piece of $\partial\Omega$ has a negative curvature, nonetheless,  conditions (A1)-(A3) hold. Set \,$\Omega:=\{(x_1,x_2)\,:\,1 \leq x_1^2 + x_2^2 \leq R^2,\,\,x_2 \geq 0\}$.  For a fixed $\alpha\in (0,\pi/2)$, we define the boundary data $g$ in polar coordinates as follows:
$$g(r,\theta)=\begin{cases}
    \dfrac{\theta}{\alpha} \quad & r\in\{1,R\},\,\,\theta\in ]0,\alpha], \\
    1 & r\in \{1,R\},\,\,\theta\in ]\alpha,\pi-\alpha],\\
\dfrac{\pi-\theta}{\alpha} &r\in \{1,R\},\,\,\theta\in ]\pi-\alpha,\pi[,\\
    0 & 1\leq r\leq R,\,\,\theta\in \{0,\pi\}.\end{cases}$$
As a result, $f^+= f^+_1 + f^+_2$\, and \,$f^-=  f^-_1 +  f^-_2$, where
$$
f^+_1 =\frac{1}{\alpha} \cH^1\LC \{(1,\theta): \theta\in [0,\alpha]\},\qquad f^-_1=\frac{1}{R\alpha}\cH^1\LC \{(R,\theta):\theta\in [0,\alpha]\},
$$
$$
f^+_2=\frac{1}{R\alpha}\cH^1\LC \{(R,\theta):\theta\in [\pi-\alpha,\pi]\},\qquad 
f^-_2=\frac{1}{\alpha}\cH^1\LC \{(1,\theta):\theta\in [\pi-\alpha,\pi]\}.
$$\\    
   % $f^+=\alpha^{-1}$ on $\{(\cos(\theta),\sin(\theta)):0 \leq \theta \leq \alpha\}$, $f^-=\alpha^{-1}$ on $\{(\cos(\theta),\sin(\theta)):\pi-\alpha \leq \theta \leq \pi\}$, $f^+=(R \alpha )^{-1}$ on $\{(R\cos(\theta),R\sin(\theta)):\pi-\alpha \leq \theta \leq \pi\}$ and $f^-=(R\alpha)^{-1}$ on $\{(R\cos(\theta),R\sin(\theta)):0 \leq \theta \leq \alpha\}$. 
%\color{black}{
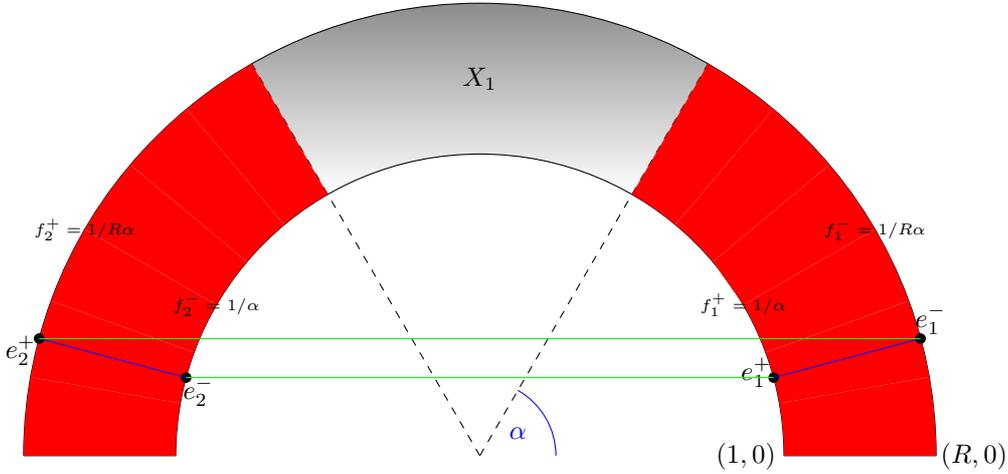
\begin{figure}[h]
\begin{tikzpicture}
%\draw (0,0) arc (-180:0:2cm);
  %\path [draw=none,fill=gray, fill opacity = 0.1]
    \shade (60:4) 
       -- (60:6) arc (60:120:6) 
       -- (120:4) arc (120:60:4);     
  \draw (6,0) arc (0:180:6);
  \draw (4,0) arc (0:180:4);
  \draw (4,0)--(6,0);
  \draw (-4,0)--(-6,0);
  \draw[dashed] (0,0)--(60:6);
  \draw[dashed] (0,0)--(120:6);
  \draw[blue] (1,0) arc (0:60:1);
  \node[blue] at (0.5,0.3) {$\alpha$};
  \node at (3.5,0) {$(1,0)$};
  \node at (6.5,0) {$(R,0)$};
  \fill[red,opacity=0.1] (0:4)--(0:6) arc (0:10:6)--(10:4) arc (10:0:4);
  \fill[red,opacity=0.2] (10:4)--(10:6) arc (10:20:6)--(20:4) arc (20:10:4);
  \fill[red,opacity=0.4] (20:4)--(20:6) arc (20:30:6)--(30:4) arc (30:20:4);
  \fill[red,opacity=0.6] (30:4)--(30:6) arc (30:40:6)--(40:4) arc (40:30:4);
  \fill[red,opacity=0.8] (40:4)--(40:6) arc (40:50:6)--(50:4) arc (50:30:4);
  \fill[red,opacity=1] (50:4)--(50:6) arc (50:60:6)--(60:4) arc (60:50:4);

  \fill[red,opacity=0.1] (180:4)--(180:6) arc (180:170:6)--(170:4) arc (170:180:4);
  \fill[red,opacity=0.2] (170:4)--(170:6) arc (170:160:6)--(160:4) arc (160:170:4);
  \fill[red,opacity=0.4] (160:4)--(160:6) arc (160:150:6)--(150:4) arc (150:160:4);
  \fill[red,opacity=0.6] (150:4)--(150:6) arc (150:140:6)--(140:4) arc (140:150:4);
  \fill[red,opacity=0.8] (140:4)--(140:6) arc (140:130:6)--(130:4) arc (130:140:4);
  \fill[red,opacity=1] (130:4)--(130:6) arc (130:120:6)--(120:4) arc (120:130:4);

 % \fill[red,opacity=0.3] (20:4)--(20:6) arc (20:40:6)--(40:4) arc (40:20:4);
%  \fill[red,opacity=0.5] (40:4)--(40:6) arc (40:60:6)--(60:4) arc (60:40:4);
%  \fill[red,opacity=0.6] (140:4)--(140:6) arc (140:120:6)--(120:4) arc (120:140:4);
%  \fill[red,opacity=0.7] (160:4)--(160:6) arc (160:140:6)--(140:4) arc (140:160:4);
%  \fill[red,opacity=1] (180:4)--(180:6) arc (180:160:6)--(160:4) arc (160:180:4);

 \fill (15:4) circle (2pt); \node at (17:3.8) {$e_1^+$};
\fill (15:6) circle (2pt); \node at (17:6.2) {$e_1^-$};
 \fill (165:4) circle (2pt); \node at (167:3.8) {$e_2^-$};
\fill (165:6) circle (2pt); \node at (167:6.2) {$e_2^+$};
 \draw[blue] (15:4)-- (15:6);
 \draw[blue] (165:4)-- (165:6);
 \draw[green] (15:4)--(165:4);
 \draw [green] (15:6)-- (165:6);

 \node at (30:4) {\tiny $f_1^+=1/\alpha$};
  \node at (150:4) {\tiny $f_2^-=1/\alpha$};
  \node at (30:6) {\tiny $f_1^-=1/R\alpha$};
  \node at (150:6) {\tiny $f_2^+=1/R\alpha$};
  \node at (90:5) {$X_1$};
\end{tikzpicture}
\caption{Example with negative curvature}
\label{Circular C shape}
\end{figure}
%}    
In this case, a part of $\partial\Omega$ has negative curvature, so in order to prove the existence of a solution to \eqref{LGP}, we decompose $\Omega$ into subsets verifying conditions (A1), (A2), (A3). We refer  to Figure \ref{Circular C shape}  for illustration. We let $X_1=\{(x,y)=(r\cos\theta,r\sin\theta):(r,\theta)\in (1,R)\times(\alpha,\pi-\alpha)\}$. 
Notice that for all $s\in [0,1]$, we have %in polar coordinates
$$g(1,s\alpha)=g(R,s\alpha)=g(1,\pi-s\alpha)=g(R,\pi-s\alpha)=s.$$

Let $0= s_0 <s_1<s_2<\cdots<s_{n-1}<s_n=1$ be a partition of $[0,1]$. We set  $\Delta s_i=s_i-s_{i-1}$. We define $E_{i,r}$ in polar coordinatinates by formula 
$$
E_{i,r}=\{(r,\theta): r\in (1,R), \theta\in (s_{i-1}\alpha, s_i\alpha)\}, \qquad i=1, \ldots, n.
$$
%be the sets with vertices in polar coordinates $(1,s_{i-1}\alpha)$, $(1,s_i\alpha)$, $(R,s_i\alpha)$, $(R,s_{i-1}\alpha)$, and
By definition, $E_{i,l}$ is the symmetric image of $E_{i,r}$ with respect to the vertical coordinate axis.
The regions $E_{i,l/r}$ are shaded in red in Figure \ref{Circular C shape}, so $\Omega\backslash[ (\cup_{i=1}^n E_{i,r})\cup (\cup_{i=1}^n E_{i,l})]=X_1$.
Notice that $\partial E_{i,r}\cap \partial \Omega$ can be decomposed into two arcs (in polar coordinates):
\begin{align*}
&E_{i,r}^+=(1,s_{i-1}\alpha)\arc{\ (}1,s_{i}\alpha),\quad E_{i,r}^-
=(R, s_{i-1}\alpha\arc{)\ (}
R, s_i\alpha).
\end{align*}
%/rand similarly 
By symmetry, we decompose $\partial E_{i,l}\cap \partial \Omega$. We choose $\Delta s_i$ small enough so that any line segments from $E_{i,l/r}^+$ to $E_{i,l/r}^-$ lies in $\Omega$. We then get that $\partial E_{i,l/r}$ satisfy conditions (A1)-(A2). 

Now, we check (A3). We begin with
%\sout{Letting}  
$e_1^{+}=(1,s\alpha),\, e_1^-=(R,s\alpha)$,\, $e_2^-=(1,\pi-s\alpha)$, $e_2^+=(R,\pi-s\alpha)$.   Condition (A3) requires that for every $s\in [0,1]$, one has 
$$2(R-1)=|e_1^+-e_1^-|+|e_2^+-e_2^-|< |e_1^+-e_2^-|+|e_2^+-e_1^-|=2\cos (s\alpha)+ 2R \cos(s\alpha).$$
%i.e. for every $s\in (0,1)$,
%$$R-1<\cos (s\alpha)+ R \cos(s\alpha),$$\\
This leads to %btaining 
the following relationship
\begin{equation}\label{df-al-R}
\cos \alpha>\dfrac{R-1}{R+1}.
\end{equation}
For such values of $\alpha$, (A3)  follows for $e_k^{\pm}$ corresponding to transport rays and, we can then further restrict $\Delta s_i$ even smaller so that the sets $E_{i,l/r}$ satisfy condition (A3) for every sequence of arbitrary points $\{e_k^\pm\}$ on $\partial E_{i_k,l/r}$. 

We conclude that if (\ref{df-al-R}) is satisfied, then we may use Theorem  \ref{Theorem 4.7} to deduce that  
Problem \eqref{LGP} has a unique solution.

%    We have 
  %  $$|e_1^+ - e_1^-| + |e_2^+ - e_2^-| = 2(R-1) $$
   % and 
    %$$|e_1^+ - e_2^-| + |e_2^+ - e_1^-| = %2(R+1)\cos(\theta)$$
    %So, 
    %$$|e_1^+ - e_1^-| + |e_2^+ - e_2^-| < |e_1^+ - e_2^-| + |e_2^+ - e_1^-|$$
    %if and only if
    %$$R-1<(R+1)\cos(\theta),\,\,\,\mbox{for all}\,\,\,\theta\in[0,\alpha]$$\\
    %$$\Leftrightarrow \cos(\alpha) >\frac{R-1}{R+1}.
    %$$\\

\end{example}
%Finally, we will relax the assumption (A3) by considering the case when we have equality instead of the strict inequality in (A3).
%is violated. 
%More precisely

\bigskip

Finally, we will also cover
%relax the assumption (A3) by considering 
the case when the %strict 
inequality in the assumption (A3) becomes the equality.
%is instead large. 
%\PR{I DON'T UNDERSTAND THIS SENTENCE.}
%\marginpar{\PR{what do mean by `well satisfied'? related to strict $<$}}
%Suppose that (A1) $\&$ (A2) hold. 
To be more precise, we introduce the following relaxation of (A3):\\

\noindent $\bullet$ \,\,{\bf{Condition ($\widetilde{A3}$)}}. \,\, Let $\{e_k^\pm\}_{1 \leq k \leq m}$ (where $m \in \mathbb{N}$) be two finite sequences of points 
such that $e_{k}^\pm \in {\chi^{i_k}}^\pm \cup {\Gamma^{i_k}}^\pm$ (for some $i_k \in I_C$) or \,$e_{k}^\pm \in E_{i_k}^\pm$ (for some $i_k \in I_E$), with $i_k \neq i_{k^\prime}$ for all $k \neq k^\prime$ such that $\{i_k,\,i_{k^\prime}\} \subset I_C$ or $\{i_k,\,i_{k^\prime}\} \subset I_E$. Then,
%${e_{k}}^\pm \in \spt(f^\pm) \cap [\partial C_{i} \cup \partial E_{i}]$,
%, for some $ i_k \in \cI_C \cup \cI_E$, 
 %for all $1 \leq k \leq m$ and $1 \leq r \leq r_k$,
we assume that we have the following  inequality: 
$$
\sum_{k=1}^m  |e_k^+ - e_k^-| \leq
%\sum_{k=1}^m\bigg[|{e_{j,n_{i,j}}^k}^+ - {e_{s_2(i,j),1}^{s_1(i,j)}}^-| +
%|{e_{j}^i}^+ - {\bf{t^{s(i)}}}({e_{s^i(j)}^{s(i)}}^+)| +
%|{e_{j,1}^i}^+ - {\bf{t^{i}}}({x_{j,1}^{i}})| +
\sum_{k=1}^{m-1} |e_{k}^+ - e_{k+1}^-| + |e_{m}^+ - e_{1}^-|. $$
%$$+\sum_{r=1}^{r_m-1}|{e_{r}^m}^+ - {\bf{T^m}}({e_{r+1}^{m}}^+)|+|{e_{r_m}^m}^+ - {\bf{T^1}}({e_{r_1}^{1}}^+)|.$$
% For every $ i \in \cI_C$, we also assume that for any sequence of points $\{e_{k}^+\}_{1 \leq k \leq m}$ (where $m \in \mathbb{N}$) such that $e_{k}^+ \in {\chi_{j_k}^i}^+ \cup {\Gamma_{j_k}^i}^+$, for some $j_k \in  I^i_\chi \cup I^i_\Gamma$, we have the following inequality:
%$$
%\sum_{k=1}^m |e_k^+ - \bbT^i(e_k^+)|
%\,\,<\,\,
%\sum_{k=1}^{m-1}
%|e_k^+ - \bbT^i(e_{k+1}^+)| + |e_m^+ - {\bf{T^i }}(e_{1}^+)|.$$ 

%Hence, we have the following:

We claim that under this weaker assumption ($\widetilde{A3}$), we can establish a version of Proposition \ref{Prop. nonconvex case general}.
\begin{proposition}\label{General case} 
Assume that (A1), (A2) $\&$ ($\widetilde{A3}$) hold and $g$ is piecewise monotone. 
%\marginnote{piecewise monotonicity!}
Then, there exists an optimal transport plan $\gamma^\star$ such that $\gamma^\star(C_i \times C_j)=\gamma^\star(E_i \times E_j)=\gamma^\star(C_i \times E_j)=0$, for all $i,\,j$ (with \,$i \neq j$ if \,$\{i,j\} \subset I_C$\, or \,$\{i,j\} \subset I_E$). Moreover,  if (S) is satisfied
then 
$\gamma^\star:=(Id,{\bf{T}})_{\#}f^+$ is an optimal transport plan between $f^+$ and $f^-$. 
%$g\in BV(\partial\Omega) \cap C(\partial\Omega)$ and set $f=\partial_\tau g$, where $\partial_\tau g$ is defined by (\ref{r1}). Then, all the transport rays between $f^+$ and $f^-$ lie inside $\Omega$. More precisely, if 
%Let \,$\gamma$\, be an optimal transport plan between $f^+$ and \,$f^-$. Then, for \,$\gamma-$a.e. $(x^+,x^-)$, we either have that \,$x^+ \in {\chi^i_j}^+$ and \,$x^- \in {\chi^i_j}^-$ or \,$x^+ \in {\Gamma^i_j}^+$ and \,$x^- \in {\Gamma^i_j}^-$, for some \,$i \in \cI_C$\, and \,$j \in \cI_\chi^i \cup \cI_\Gamma^i$, or that \,$x^+ \in E_i^+$ and \,$x^- \in E_i^-$, for some $i \in \cI_E$.
\end{proposition}
\begin{proof} 
We apply the argument used in the proofs of
%The idea is quite similar to the one in 
Propositions \ref{Prop. convex case with finite arcs} and \ref{Prop. convex case}. %{\color{black} We can't use Proposition  \ref{Prop. nonconvex case general} here!}
For this purpose we will construct an increasing sequence of  domains $\Omega_n$ whose closures converge to $\overline{\Omega}$ in the Hausdorff distance as well as a sequence of functions $g_n$ defined on $\partial\Omega_n$ such that for every $n$, the boundary $\partial\Omega_n$ can be decomposed into sets $\tilde{C}_{i,n}$ and $E_{i,n}$ 
%$\partial C_{i,n} \cap \partial\Omega_n$ can be decomposed into arcs $\chi^{i,n}_j$, $\Gamma^{i,n}_j$ and $F_j^{i,n}$ for appropriate sets of indices, 
satisfying condition (A3) and so that we have $\partial \tilde{C}_{i,n} \rightarrow \partial {C}_{i}$ and $\partial {E}_{i,n} \rightarrow \partial {E}_{i}$ in the Hausdorff sense. Here, $\tilde{C}_{i,n}$ is not necessarily convex.

{\it Step 1.} Fix $n \in \mathbb{N}^\star$.
{Let us suppose that  $\alpha$ is any of the arcs $\chi^i_j$ ($j\in I^i_\chi$, $i \in I_C$), $\Gamma^{i\pm}_j$ ($j\in I^i_\Gamma$, $i \in I_C$) or $E_i^\pm$ ($i\in I_E$).} %\marginnote{\PR{we may\\ have many $E_j$'s}}
We may assume after an appropriate choice of the coordinate system that $\alpha$ is the graph of a  Lipschitz continuous function $h_\alpha:[-r_\alpha, r_\alpha]\mapsto \bR$, i.e. $\alpha = G(h_\alpha)$, where $h_\alpha(-r_\alpha) = 0 = h_\alpha(r_\alpha)$ and $G(h_\alpha)$ denotes the graph of $h_\alpha$.
%Moreover, 
%We notice that our
We adopt a convention requiring $h_\alpha\ge 0$ when $\alpha$ is convex (i.e. $\conv \alpha\subset \Omega$), this
implies that $h_\alpha$ is %always 
concave. Moreover, $h_\alpha\le 0$ when $\alpha$ is concave, so $h_\alpha$ is a convex function. In particular, there is a $r>0$ such that $\{ (x,y): \ x\in [-r_\alpha, r_\alpha], y\in [h_\alpha(x), h_\alpha(x) +r ]\}\cap \Omega=\emptyset$.
%\marginnote{inconsistencies removed}
%(resp. convex) if $\alpha$ is convex (resp. concave).
%or $\alpha=\Gamma^{i\pm}_j $ and for convex $E_j^\pm$.  For a concave arc $E_j$ function $\alpha$ is convex. 

Our construction depends on the curvature of $\alpha$. First, we consider $\alpha$ which is not a line segment. Then, for any natural $n\ge1$,
%let $\eta_\alpha \geq 0$ be a strictly concave function such that $\eta_\alpha(\pm r_\alpha)=0$.
%Then,
we define $\alpha^n := G((1- \frac \kappa n) h_\alpha)$, where $\kappa =1$ when $\alpha$ is convex and $\kappa =-1$ when $\alpha$ is concave. When $\alpha$ is convex, then this construction implies that $\alpha^n$ is convex and  $\alpha^n\subset \conv(\Bar \alpha)\subset\bar\Omega.$ In case $\alpha= E_i^+$ (resp. $\alpha= E_i^-$) is concave, since the distance between $E^+_i$ and $E^-_i$ is positive we conclude that $E^+_{i,n}:=\alpha^n$ (resp. $E^-_{i,n}:=\alpha^n$) is concave and $\alpha^n\subset E_i$ for sufficiently large $n$.

When  $\alpha$ happens to be a line segment we proceed differently. We take a strictly convex function $\eta_\alpha:[-r_\alpha,r_\alpha] \mapsto (-\infty,0]$ such that $G(\eta_\alpha) \subset \Omega$ with $\eta_\alpha(\pm r_\alpha)=0$,  then we set $\alpha^n := G(\frac{\eta_\alpha}{n})$, $n\ge 1$.  
%According to the convention used above we have $\{(x_1, x_2)\in\bR^2:\ |x_1|<r_\alpha,\ x_2<0\}\cap \Omega {\color{red}{=}}\emptyset$.
Hence, we conclude that $\alpha^n \subset \Omega$ for large $n\in\bN.$
%$h_\alpha\equiv 0,$ hence
%$\alpha^n = \alpha$.
%We see that for a convex $\alpha$ all
%In particular, we see that if $\alpha$ is convex, then the curve $\alpha^n$ is contained in the interior of the convex hull of $\alpha$. In case $\alpha=E_i^\pm$ is concave (resp. $\alpha={\Gamma^i_j}^\pm$ is a line segment), then it is easy to check that for large $n$ the curve $\alpha^n$ is contained in $E_i$ (resp. in the convex hull of ${\Gamma^i_j}^+$ and ${\Gamma^i_j}^-$).}

%\PR{
After this preparation we will define the domain $\Omega_n.$
%The process depends whether or not we have to deal with line segments.
%} 
%The first case occurs when}, 
%\marginnote{\PR{the case when\\ both $\Gamma^+_i$, $\Gamma^-_i$ are\\ intervals is not treated!}}
 %for all $i \in I_C$ and $j \in I^i_\Gamma$, 
%at least one of \,${\Gamma^i_j}^+$ and ${\Gamma^i_j}^-$ is not a line segment. In this case
More precisely, the boundary $\partial \Omega_n$ is nothing else than $\partial\Omega$ but every arc $\,\alpha ={\chi_j^i}$
%, then we set $\chi^{i,n}_j = \alpha^n$;\\
%(b) \,if \,
(resp. $\alpha = {\Gamma^i_j}^\pm$
or $\alpha=E_{i}^\pm$) in $\partial\Omega$
%(resp. $\alpha={\Gamma^i_j}^\pm$) 
is replaced by the arc $\alpha^n$. We note that $\alpha^n$ and $\alpha$ have same endpoints and so $\partial \Omega_n$ is 
%\, and \,$E_{i,n}^-:=E_i^-$.\\
%\PR{(f) if  $E_{i}^+$, (resp. $\alpha=E_{i}^-$) is a line segment, then  $E_{i,n}^+:=E_{i}^+$ (resp. $E_{i,n}^-:=E_{i}^-$) and $E_{i,n}^-:= \alpha^n$ (resp. $E_{i,n}^+:=E_{i}^+$), where $\alpha = E_{i}^-$, (resp. $\alpha = E_{i}^+$.}
%\marginnote{\PR{MIND NEW (f) and (e)}}
%%${\Gamma_j^{i,n}}^\pm:=\alpha^n$), 
%where $\alpha^n$ is always the curve %defined in the way described above.
%On the other hand, for all $i \in I_E$, let $E_{i,n}^\pm$ be two arcs inside $E_i$ having the same endpoints as $E_i^\pm$ and close enough to $E_i^\pm$. 
%\PR{Let us assume that $E_{i,n}^+$, $E_{i,n}^-$
%are curves  
%Then, $\Omega_n$ is defined as the domain bounded by 
%\PR{we introduce a sequence of 
a closed curve (i.e. topologically a circle). In the same way, we define $\partial \tilde{C}_{i,n}$ as $\partial C_i$ but any arc $\,\alpha ={\chi_j^i}$
%, then we set $\chi^{i,n}_j = \alpha^n$;\\
%(b) \,if \,
(resp. $\alpha = {\Gamma^i_j}^\pm$) in $\partial C_i$ is replaced by $\alpha^n$. Moreover, $E_{i,n}$ is the region bounded by $E_{i,n}^\pm$ and $\partial E_i \cap \Omega$. The sets $\Gamma^i_{j,n}$ and $\chi_{j,n}$ are defined in a natural way. It is clear that $\partial\Omega_n \rightarrow \partial \Omega$, $\partial \tilde{C}_{i,n} \rightarrow \partial {C}_{i}$ and $\partial {E}_{i,n} \rightarrow \partial {E}_{i}$ in the Hausdorff distance as $n\to \infty$. 

{\it Step 2.} Now, we  define maps $S_n: \partial\Omega %\partial\Omega 
\mapsto \partial\Omega_n$  %\marginnote{$\bbT$ not defined at endpoints} 
as follows. Since $g$ is piecewise monotone, we first consider points belonging to $U^+\cup U^-$ (see Definition \ref{d-ps-mono}). For such points, $S_n(x)$ is set to be the point of intersection of the line segment $[x,{\bf{T}}(x)]$ (resp. $[x,{\bf{T}}^{-1}(x)]$) with the boundary  $\partial\Omega_n$, which is the closest to $x$. We note that $S_n$ is well defined since the line segment $[x, \bbT(x)]$ (resp. $[x,{\bf{T}}^{-1}(x)]$) intersects $\Omega$ transversally and $\alpha_{n}$ is either strictly convex or strictly concave. Moreover, we see  that the inverse map $S^{-1}_n$ exists, for each $n\in\bN^\star$.

%For all \,$x \in \spt(f^+)^{{\circ}}$ (resp. $x \in \spt(f^-)^{{\circ}}$); 
We extend $S_n$ to $\partial\Omega$
%\backslash \spt(f)$ 
by setting $S_n=Id$ on $U_0$, ($U_0$ is from the Definition \ref{d-ps-mono}). %\partial\Omega\backslash\spt(f)$. 
Hence, if $x \in \alpha=\chi_j^i$ (resp. $x \in \alpha={\Gamma_j^i}^\pm$ or $x \in  \alpha=E_i^\pm$), then it is obvious that $S_n(x) \in \alpha^n$ for sufficiently large $n\in \bN$. In particular, we have $S_n(\partial {C}_{i,n})=\partial \Tilde{C}_{i,n}$ and $S_n(E_{i}^\pm)=E_{i,n}^\pm$.

Moreover, due to continuity of $\bbT$ and  $\bbT^{-1}$  it is not difficult to check that the mapping $S_n$ is continuous as well on $\partial\Omega$.
%$\bigcup_{i\in I_C} (\Gamma^i\cup \chi^i)\cup \bigcup_{i\in I_E}  E_i^\pm$.
%, where $(\partial E_i\cap \partial \Omega)^\circ$ denotes the relative interior of the set with respect to $\partial \Omega$}.
  
%\, on $\partial\Omega\setminus(\cup_{i \in I_E} E_i^\pm)$. 
%where \,$\mathcal{C}$\, is the union of all concave arcs ${E_i}^\pm$ and $L_n:=\cup_{i,\,k} \,{L^{i,n}_{k}}^\pm$.
%It is clear that $\partial\Omega_n$ is connected. Now
Once we fix any $x_0\in U_0$
we define the boundary data $g_n$ on $\partial\Omega_n$ by $g_n(x) = f_n(\arc{x_0\,x})$, where $f_n$ is defined as follows:
%$$f_n=f \res [\partial\Omega\setminus(\cup_{i \in I_E} E_i^\pm)]+{S_n}_{\#}[f\res (\cup_{i \in I_E}\, E_i^\pm)],$$
%In this way, we get that 
$$f_n={S_n}_{\#}f.$$ 
It is clear that $f_n(\partial \Tilde{C}_{i,n})=f(S_n^{-1}(\partial \Tilde{C}_{i,n}))=f(\partial C_i)=0$ and $f_n(\partial {E}_{i,n})=f(S_n^{-1}(\partial {E}_{i,n}))=f(\partial E_i)=0$.

{\it Step 3.} Since $f_n$ is defined only on $\partial\Omega_n$,  we extend it 
%by zero 
on \,$\overline{\Omega} \setminus \partial\Omega_n$, 
by formula $\bar f_n(A) = f_n(A\cap \partial\Omega_n)$ for any Borel set $A$ and
for all $n$.  Let $\bar f_n^+$ and ${\bar f}_n^-$ be the positive and negative parts of $\bar f_n$. 
%Thanks to Proposition \ref{Prop. nonconvex case general}, 
Let $\gamma_n$ be any optimal transport plan between $\bar f_n^+$ and $\bar f_n^-$ on  \,$\overline{\Omega} \times  \overline{\Omega}$. Up to a subsequence, we know that $\gamma_n \rightharpoonup \gamma$ for some $\gamma \in \mathcal{M}^+(\overline{\Omega} \times\overline{\Omega})$. Yet, we see that $\bar f_n \rightharpoonup \bar f$, where $\bar f\LC \partial\Omega = f$ and $\spt\bar f \subset \partial\Omega$. Indeed, for any $\varphi \in C(\overline{\Omega})$, we have
%\marginnote{there are\\ 
%changes here}
%{\color{black}{NOT ANY LONGER, NOW $f_n$ IS A MEASURE OVER $\Omega$ WHICH IS NOT A PUSH FORWARD OF $f$}}
$$
\lb \bar f_n,\varphi\rb= 
\lb f_n,\varphi\rb= \lb {S_n}_{\#}f,\varphi\rb=  \int_{\partial\Omega} \varphi(S_n(x))\,\mathrm{d}f(x)  \rightarrow  \int_{\partial\Omega} \varphi(x)\,\mathrm{d}f(x)=\lb f,\varphi\rb,$$
because $S_n(x)$ converges to $x$, for all $x \in \partial\Omega$. This follows immediately from 
%the fact that
%
the definition of $\Omega_n$ which assures us that
%we choose the arcs ${\chi^{i,n}}$, ${\Gamma^{i,n}}^\pm$ and $E_{i,n}^\pm$ in such a way that we have the following estimate: 
\begin{equation}\label{r-Hd-co}
\lim_{n\to\infty}\max\bigg\{|x-S_n(x)|\,:\,x\in \spt(f)\bigg\} =0.
\end{equation} %\leq \frac{1}{n}
%(thanks to the fact that $|S_n(x)-x| \leq \frac{1}{n}$). 
Hence, $(\Pi_x)_{\#}\gamma=f^+$ and $(\Pi_y)_{\#}\gamma=f^-$.
%On the other hand, let $u_n$ be a Kantorovich potential such that $u_n(x_0)=0$, for a fixed point $x_0 \in \overline{\Omega}$. Then, $u_n$, up to a subsequence, converges uniformly to a function $u \in \Lip_1(\overline{\Omega_{n_0}})$. Yet, thanks to the duality $\min\eqref{Kantorovich}=\sup\eqref{dual}$, we have
%$$\int_{\overline{\Omega}_{n_0} \times \overline{\Omega}_{n_0}}|x-y|\,\mathrm{d}\gamma_n=\int_{\overline{\Omega}_{n_0}} u_n\,\mathrm{d}f_n.$$
%Passing to the limit when $n \to \infty$
Similarly, as in the proof of Proposition \ref{Prop. convex case with finite arcs}, we infer that $\gamma$ is an optimal transport plan between $f^+$ and $f^-$. %At the same time $u$ is the corresponding Kantorovich potential, because
%$$
%\min\eqref{Kantorovich} 
%\leq \int_{\overline{\Omega}_{n_0} \times \overline{\Omega}_{n_0}}|x-y|\,\mathrm{d}\gamma
%=\int_{\overline{\Omega}_{n_0}} u\,\mathrm{d}f \leq \sup \eqref{dual}.$$
%In addition, thanks to Proposition \ref{Prop. nonconvex case}, we have $\gamma_n=(Id,{\bf{T}}_n)_{\#}{f_n^+}$.
%On the other hand, for any $x \in \spt(f^+)$, we have $S_n(x) \to x$ and ${\bf{T}}_n(S_n(x)) \to  {\bf{T}}(x)$ since ${\bf{T}}_n(S_{n}(x))=S_{n}({\bf{T}}(x))$.
%But, 
%$$u_n(P_{n}^{-1}(x))-u_n({\bf{T}}_n(P_{n}^{-1}(x)))=|P_{n}^{-1}(x)-{\bf{T}}_n(P_{n}^{-1}(x))|.$$
%Passing to the limit when $n \to \infty$, we get that
%$$u(x)-u({\bf{T}}(x))=|x-{\bf{T}}(x)|.$$
%In other words, this means that
%Recalling again Proposition \ref{Prop. convex case with finite arcs}, this yields that $(Id,{\bf{T}})_{\#}{f^+}$ is an optimal transport plan between $f^+$ and $f^-$. 

{\it Step 4.}
Now, we show that $g_n$  satisfies (A3).
     %Let $P^{[-1]}_n$ be the inverse map of $P$ on $\partial\Omega_n$. 
%where the last equality follows from the definition of ${\bf{T}}$.
%which yields that $P_n^{[-1]}({\bf{T}}(P(e^+)))  \in \chi_{i,n}^-$ (resp. $\Gamma_{i,j,n}^-$).
%We will show it in two steps. First, l
Let $\{e_{k}^\pm\}_{1 \leq k \leq m}$ be two finite sequences of points such that 
${e_{k}}^\pm \in \partial \tilde{C}_{i_k,n} \cap \left(\bigcup_{j\in I^{i_k,n}_\Gamma} \Gamma^{i_k,n}_j\cup \bigcup_{j\in I^{i_k,n}_\chi} \chi^{i_k,n}_j\right)$ %\spt(f_n^\pm)$ 
(for some $i_k \in I_C$) or \,${e_{k}}^\pm \in E_{i_k,n}^\pm$ (for some $i_k \in I_E$) with $i_k \neq i_{k^\prime}$ for all $k \neq k^\prime$ such that $\{i_k,\,i_{k^\prime}\} \subset I_C$ or $\{i_k,\,i_{k^\prime}\} \subset I_E$.
%$e_{k}^+ \in \partial C_{i} \cup  E_{i,n}^+$, for some $i \in \cI_C \cup \cI_E$. 

For every $1 \leq k \leq m$, let ${e_k^\prime}^\pm \in \spt(f^\pm)$ be such that $[{e_k^\prime}^+,{e_k^\prime}^-] \cap \partial\Omega_n=\{e_k^+,e_k^-\}$.
%{This is similar to the construction of $S_n$.} 
Thus, %So, one can see easily that
%we have the following equality (we recall that $\tilde{S}_n=Id$ on $\partial\Omega\setminus(\cup_{i \in I_E} E_i)$):
$$
\sum_{k=1}^{m} |e_k^+ - e_k^-|
%=\sum_{i \in \cI_C} |e_i^+ - {\bf{T}}(e_i^+)| + \sum_{i \in \cI_E} |e_i^+ - {\bf{T}}_n(e_i^+)|$$
=
%\sum_{i \in \cI_C} |e_i^+ - {\bf{T}}(e_i^+)|+ 
\sum_{k=1}^m (|{e_k^\prime}^+ - {e_k^\prime}^-| - |e_k^+ - {e_k^\prime}^+| -|e_k^- - {e_k^\prime}^-|).$$\\
%$$\leq \sum_{k=1}^{m} |S^{-1}(e_k^+) - {\bf{T}}(S^{-1}(e_k^+))|+\varepsilon_n,$$
%where $\varepsilon_n \to 0$ when $n \to \infty$.
%2\, \mathrm{d}_H(\partial\Omega_n,\partial\Omega)
%\leq \sum_{i=1}^{m_\star} |P(e_i^+) - {\bf{T}}(P(e_i^+))|+\frac{2\,m_\star}{n}.$$
%For
At the same time, due to the triangle inequality, we have 
\begin{align}\label{r410p}
\sum_{k=1}^{m-1}
|{e_k^\prime}^+ - e_{k+1}^{\prime -}| + |e_m^{\prime +}- {e_1^\prime}^-|
&\le
\sum_{k=1}^{m-1}
|{e_k^\prime}^+ - e_{k}^+| + |e_k^+ - e_{k+1}^-| + |e_{k+1}^--e_{k+1}^{\prime-}|\\
&\qquad \qquad+|e_m^{\prime+} - e_{m}^+| + |e_m^+ - e_{1}^-| + |e_{1}^- - {e_1^\prime}^-|.\nonumber
\end{align}
The inequality in (\ref{r410p}) is
strict  as soon as
there is at least one integer $k_0$ such that the points 
$e_{k_0}^{\prime+},\,e_{k_0}^+,\, e_{k_0+1}^-,\,e_{k_0+1}^{\prime-}$
%the points $\{e_{k}^\pm\,:\,1 \leq k \leq m\}$ 
are not co-linear. We proceed while assuming that this is the case. Since $e^{\prime\pm}_k\in \partial\Omega$, then ($\widetilde{\hbox{A3}}$) implies that we also have 
 $$\sum_{k=1}^{m} |{e_k^\prime}^+  - {e_k^\prime}^-| \leq  \,\,
\sum_{k=1}^{m-1}
|{e_k^\prime}^+ - e_{k+1}^{\prime-}| + |e_m^{\prime+}-{e_1^\prime}^-|.$$
%But, we have
%$$\sum_{k=1}^{m} |e_k^+ - {\bf{T}}_n(e_k^+)| $$
%$$
%<\sum_{k=1}^{m-1}
% |e_k^+ - {\bf{T}}_n(e_{k+1}^+)|  + |e_m^+ - {\bf{T}}_n(e_{1}^+)|+2
% \sum_{i=1}^{m} [|e_i^+ - S_n^{-1}(e_i^+)|+|{\bf{T}}(S_n^{-1}(e_i^+)) - {\bf{T}}_n(e_i^+)|]%<\sum_{k=1}^{m-1}
 %|e_k^+ - {\bf{T}}_n(e_{k+1}^+)|  + |e_m^+ - {\bf{T}}_n(e_{1}^+)|+
% \leq \varepsilon_n,$$
 %for some $\varepsilon_n \to 0$ when $n \to \infty$. Then, for $n$ large enough, 
Consequently, we infer that 
\begin{equation} \label{no transport outside} 
 \sum_{k=1}^{m} |{e_k}^+  - {e_k}^-| <  \,\,
\sum_{k=1}^{m-1}
|{e_k}^+ - {e_{k+1}}^-| + |{e_m}^+-{e_1}^-|.
\end{equation}

It remains to consider the case when for all $k=1,\ldots, m$
%this is not the case, i.e.
%we have equality then it means that one has
%$$|{e_k^\prime}^+ - {e_{k+1}^\prime}^-|=|{e_k^\prime}^+ - e_{k}^+| + |e_k^+ - e_{k+1}^-| + |e_{k+1}^--{e_{k+1}^\prime}^-|,\,\,\,\,\mbox{for all}\,\,\,1 \leq k \leq m-1$$
%and 
%$$|{e_m^\prime}^+- {e_1^\prime}^-|=|{e_m^\prime}^+ - e_{m}^+| + |e_m^+ - e_{1}^-| + |e_{1}^- - {e_1^\prime}^-|.$$\\
%In other words, this yields that 
the points ${e_k^\prime}^+,\,e_{k}^+,\, e_{k+1}^-,\,e_{k+1}^{\prime-}$ are co-linear,  (we use the convention that $e^-_{m+1}\equiv e^-_1$ and ${e^\prime_{m+1}}^-\equiv {e^\prime}^-_1$). 
%and are exactly in this order; the same for the points ${e_m^\prime}^+,\,e_{m}^+,\, e_{1}^-,\,{e_{1}^\prime}^-$. 
By definition %Yet, we also know that 
${e_k^\prime}^+,\,e_{k}^+,\, e_{k}^-,\,{e_{k}^\prime}^-$ are co-linear, too. Consequently, all  points $\{e_k^\pm,\,{e_k^\prime}^\pm\}$ are co-linear. We claim that this observation combined with the fact that the sets $C_i$, $i \in I_C$, $E_j$, $j\in I_E$ are mutually disjoint, implies then the line segments $[e_k^+,e_k^-]$, $1 \leq k \leq m$, must be disjoint. Indeed, let us assume that there is a common point ${e_{k_1}^\prime}^+={e_{k_2}^\prime}^+ =:p \in \partial C_{i_{k_1}} \cap \partial C_{i_{k_2}}$ with $i_{k_1} \neq i_{k_2}$. We notice that ${e_{k_1}^\prime}^+={e_{k_2}^\prime}^-$ is impossible. 

We know that $]{e_{k_1}^\prime}^+,{e_{k_1}^\prime}^-[ \subset C_{i_{k_1}}$ and $]{e_{k_2}^\prime}^+,{e_{k_2}^\prime}^-[ \subset C_{i_{k_2}}$ while $C_{i_{k_1}}$ and $C_{i_{k_2}}$ are disjoint. Since point $p$ is by definition in the (relative) interior of both arcs $\partial C_{i_{k_1}}$ and $\partial C_{i_{k_2}}$, then due to the Lipschitz continuity of
$\Omega$ we see that $\partial C_{i_{k_1}}$ and $\partial C_{i_{k_2}}$ must coincide in a neighborhood of the common point $p$. %because this point is in the interior of the boundary
As   a result, there will be always a triple bifurcation point $a \in \partial C_{i_{k_1}} \cap \partial C_{i_{k_2}}$, which contradicts again the Lipschitz regularity of $\partial\Omega$.

%Now, for the sake of simplicity of the argument we assume that  are real numbers. 
Let us take $m_-$, $m_+$ among points $e^\pm_k$, $k=1,\ldots, m$, such that 
$$
|m_+ - m_- | = \diam\, \{ e^+_k, e^-_k: 1 \leq k \leq m\}.
$$
%\marginnote{S.: Remove: convex hull}
%$$m_- = \min\{,\qquad
%m_+ = \max{e^+_k, e^-_k: 1 \leq k \leq m}.
%$$
Since the intervals $[e_k^+,e_k^-]$, $1 \leq k \leq m$, are disjoint, then we see that we have the following inequality:
$$ 
 \sum_{k=1}^{m} |{e_k}^+  - {e_k}^-| <|m_+- m_- |.
$$
At the same time every $x\in ]m_+, m_-[$ belongs to at least one interval $[ e^+_{k}, e^-_{k+1}]$. %$\}, \max\{ e^+_{k+1}, e^-_k\}] $. 
Hence, we get that
$$
|m_+ - m_-| \le \sum_{k=1}^{m-1}
|{e_k}^+ - {e_{k+1}}^-| + |{e_m}^+-{e_1}^-|.
$$
As a result (\ref{no transport outside}) follows.

{\it Step 5.}
%Finally, assume that there are some $i \in I_C$ and $j \in I^i_\Gamma$ such that both ${\Gamma^i_j}^+$ and ${\Gamma^i_j}^-$ are line segments (see Example \ref{non uniqueness of the optimal transport plan}). The only difference now is that when we approximate one of these curves 
 %(say $\alpha$) by $\alpha^n$, we lose convexity of $C_{i,n}$. However, this will not pose any issue. Indeed, let $\gamma_n$ be an optimal transport plan between $f_n^+$ and $f_n^-$. 
 Thanks to \eqref{no transport outside}, one can show exactly as in the proof of Proposition \ref{Prop. nonconvex case}\,-\,{Step 1} that $\gamma_n$ transports all the mass on $\partial \tilde{C}_{i,n}$ (resp. $\partial E_{i,n}$) to itself. We note that here we are not interested in characterizing the restriction of $\gamma_n$ to $\partial \tilde{C}_{i,n} \times \partial \tilde{C}_{i,n}$, since we recall that $\tilde{C}_{i,n}$ is now no more convex and thus, (H1) is a priori not satisfied in $\tilde{C}_{i,n}$. From Step 3, we know that $\gamma_n \rightharpoonup \gamma$ and $\gamma$ is an optimal transport plan between $f^+$ and $f^-$. Hence, we infer that $\gamma$ also transports the mass from $\partial C_i$ (resp. $\partial E_i$) to itself. To see this, take two continuous functions $\varphi(x)$ and $\psi(y)$ over $\overline{\Omega}$ such that $\spt(\varphi) \subset \overline{{C}_{i}}$ and $\spt(\psi) \subset \overline{E_{j}}$. Yet, we have ${\tilde{C}_{i,n}} \subset {{C}_{i}}$ and $E_{j,n} \subset E_j$, for $n$ large enough.
 %in the Hausdorff distance, then for $n$ large enough we see that $\spt(\varphi) \subset\subset {\tilde{C}_{i,n}}$ and $\spt(\psi) \subset\subset {E_{j,n}}$. 
 Hence, we get
 $$\int_{\overline{\Omega} \times \overline{\Omega}}\varphi(x)\, \psi(y)\,\mathrm{d}\gamma_n(x,y)=\int_{\overline{C_i} \times \overline{E_j}}\varphi(x)\, \psi(y)\,\mathrm{d}\gamma_n(x,y)=\int_{\overline{\tilde{C}_{i,n}} \times \overline{E_{j,n}}}\varphi(x)\, \psi(y)\,\mathrm{d}\gamma_n(x,y)=0$$
  %\longrightarrow \int_{\overline{\Omega} \times \overline{\Omega}}\varphi(x)\, \psi(y)\,\mathrm{d}\gamma(x,y).$$
 Consequently, letting $n\to \infty$ we obtain
 $$\int_{\overline{\Omega} \times \overline{\Omega}}\varphi(x)\, \psi(y)\,\mathrm{d}\gamma(x,y)=0.$$
 \\
 But, $\varphi$ and $\psi$ are two arbitrary functions supported in $\overline{{C}_{i}}$ and $ \overline{E_{j}}$, respectively. Then, we infer that $\gamma(\partial C_i \times \partial E_j)=0$. In the same way, we show that $\gamma(\partial C_i \times \partial C_j)=0$ and $\gamma(\partial E_i \times \partial E_j)=0$, for all $i \neq j$. Recalling Proposition \ref{Prop. nonconvex case general}, we conclude the proof. $\qedhere$
\end{proof} 
Now, we show that the assumption (S) in Proposition \ref{General case} can be removed. We recall that we needed this condition in Section \ref{Section Convex case} because our approach was based on the approximation by strictly convex domains $\Omega_n$, where we used the projection map $P_n$ to the boundary $\partial\Omega$. However, one can use a more suitable map in our approximation, which does not take into account the presence of singular points on $\partial\Omega$ -- this is the map $S_n$ that we have just constructed in Proposition \ref{General case}. More precisely, we have:
%\marginnote{S.: Piecewise monotonicity is an initial assumption}
\begin{proposition}\label{General}
Assume that (A1), (A2) and ($\widetilde{A3}$) hold and $g\in W^{1,1}(\partial\Omega)$ is piecewise monotone. Then, ${\bf{T}}$ is an optimal transport map from $f^+$ to $f^-$. 
\end{proposition}

\begin{proof} %We will reduce the current problem to the setting of 
%Proposition  \ref{General case}.
%Recalling the proof of Proposition \ref{General case}, we know that there is an optimal transport plan $\gamma$ such that $\gamma$ transports any set $C_i$ (resp. $E_i$) to itself. {\color{black} we can't use Proposition \ref{General case} here.}
From Proposition \ref{General case}, we just need to show that $\gamma^\star\res [C_i \times C_i]=(Id,{\bf{T}})_{\#}(f^+\res \partial C_i)$, for all $i \in I_C$. For simplicity of exposition and without loss of generality, let us assume that $\Omega$ is convex or equivalently, that  $I_C$ is a singleton and $I_E=\emptyset$; so we have $\Omega=C_1$.

%$\beta_n$ as the following sum:
%the domain $\Omega_n$ as follows: %%After a repartition of the arcs $E_i^+$ and $E_i^-$, we assume that for all $i$, one between $E_i^+$ or $E_i^-$ is concave while the second is either convex or concave. If ${E_i}^\pm$ is concave, then we put $n$ points (say ${p^{i,n}_k}^\pm$, $1 \leq k \leq n$) on ${E_i}^\pm$.
%and we connect each others to the endpoints of ${E_i}^\pm$ by $(n+1)-$line segments (say ${L^{i,n}_{k}}^\pm$, $1 \leq k \leq n+1$). 
%If both ${E_i}^+$ and ${E_i}^-$ are concave, one can 
%%We choose the points $\{{p^{i,n}_{k}}^\pm\}$ on $E_i^\pm$ so that ${p^{i,n}_{k}}^-={\bf{T}}({p^{i,n}_{k}}^+)$. We note that these points $\{{p^{i,n}_{k}}^\pm\}$ divide ${E_i}^\pm$ into $(n+1)-$arcs ${E^{i,n}_{k}}^\pm$, $1 \leq k \leq n+1$. Let $E^{i,n}_k$ be the convex hull of ${E^{i,n}_k}^\pm$. Then, we define the domain $\Omega_n$ as follows:
%$$%\partial\Omega_n
%{\beta_n}:=
%$\partial\Omega\setminus \spt(f)$ and the arcs $\alpha^n$ (where $\alpha ={\chi_j^i}$,
%, then we set $\chi^{i,n}_j = \alpha^n$;\\
%(b) \,if \,
 %$\alpha = {\Gamma^i_j}^\pm$ or $\alpha=E_i^\pm$).
%(\cup_{i \in I_C} ({\Gamma^i}^\pm \cup{\chi^i}^\pm) \cup (\cup_{i \in I_E} E_i^\pm)]
%\cup [\cup_{i \in I_C} (\overline{{\Gamma^{i,n}}^\pm} \cup \overline{{\chi^{i,n}}}) \cup (\cup_{i \in I_E} \overline{E_{i,n}^\pm})].$$ \\
Now, we set 
$$I_S:=\{i \in I_\chi \cup I_\Gamma\,:\,{\color{black}|f|(\chi_i \cap \overline{\mathcal{S}})>0\,\,\,\,\mbox{or}\,\,\,\,|f|(\Gamma_i^\pm \cap \overline{\mathcal{S}})>0} %\,\,\,\mbox{is dense}
\}.$$
For any fixed $n \in \mathbb{N}$ and  every $i \in I_S$, we take $n$ points on $\chi_i$ (resp. $\Gamma_i^\pm$), including the endpoints, which are equidistanced in the sense of the arclength. We take the boundary of their convex hull and after removing the interval connecting endpoints of $\chi_i$  (resp. $\Gamma_i^\pm$) we call the resulting 
%and take the 
polygonal curve $\chi_{i,n}$ (resp. $\Gamma_{i,n}^\pm$).
%joining these points to the endpoints of $\chi_i$ (resp. $\Gamma_i^\pm$). 
Now, we define the regions $\Omega_n$, $n \in \bN$, as 
the open sets bounded by  curves 
$$\partial\Omega_n:=[\partial\Omega\backslash \bigcup_{i \in I_S}(\chi_i \cup \Gamma_i^\pm)] \cup \bigcup_{i \in I_S} (\chi_{i,n} \cup \Gamma_{i,n}^\pm).$$ 

It is clear that $\Omega_n$ is convex. %Recalling 
The map $S_n$ has been already  constructed; %previously,
however, we recall its definition for the sake of clarity of exposition. 
For points $x\in U^+\cup U^-$, we define
 $S_n(x)$ to be the point of intersection of the line segment $[x,{\bf{T}}(x)]$ (resp. $[x,{\bf{T}}^{-1}(x)]$) with the boundary  $\partial\Omega_n$, which is the closest to $x$. On $U_0$, we set $S_n=Id$.
%Namely,   %again maps 
%$$
%S_n: \partial\Omega %\partial\Omega 
%\mapsto \partial\Omega_n
%$$  %\marginnote{$\bbT$ not defined at endpoints} 
%is given as follows, for all $x \in \spt(f^+)$ (resp. $x \in \spt(f^-)$): $S_n(x)$ is the point, which is 
%the closest to $x$, belonging to the intersection of the line segment $[x,{\bf{T}}(x)]$ (resp. $[{\bf{T}}^{-1}(x),x]$) with $\partial\Omega_n$. We extend $S_n$ to $\partial\Omega$
%\backslash \spt(f)$ 
%by setting $S_n=Id$ on $\partial\Omega\backslash \spt(f)$. 
We recall also that the inverse map $S^{-1}_n$ is well defined on $\partial\Omega_n$ %Since $\bbT$ is  continuous and the interval $[x, \bbT(x)]$ intersects $\Omega$ transversally, then it is not difficult to check 
and that the map $S_n$ is continuous. 
%Moreover, \eqref{r-Hd-co} holds. %{the definition of $\Omega_n$ assures us that
%we choose the arcs ${\chi^{i,n}}$, ${\Gamma^{i,n}}^\pm$ and $E_{i,n}^\pm$ in such a way that we have the following estimate: 
%\begin{equation}\label{r-Hd-co}
%\lim_{n\to\infty}\max\bigg\{|x-S_n(x)|\,:\,x\in \spt(f)\bigg\} =0,
%\end{equation}\\ %\leq \frac{1}{n},
Now, we define the Borel measure $f_n$ on $\partial\Omega_n$ as follows:
$$f_n={S_n}_{\#}f.$$ 
%, where $(\partial E_i\cap \partial \Omega)^\circ$ denotes the relative interior of the set with respect to $\partial \Omega$}.
%\, on $\partial\Omega\setminus(\cup_{i \in I_E} E_i^\pm)$. 
%where \,$\mathcal{C}$\, is the union of all concave arcs ${E_i}^\pm$ and $L_n:=\cup_{i,\,k} \,{L^{i,n}_{k}}^\pm$.
%It is clear that $\partial\Omega_n$ is connected. Now
Let us fix any $x_0\in U_0$, we note that $x_0 \in \partial\Omega_n$, for all $n$. Then, we set the boundary data $g_n$ on $\partial\Omega_n$ by $g_n(x) = f_n(\arc{x_0\,x})$.
%$$f_n=f \res [\partial\Omega\setminus(\cup_{i \in I_E} E_i^\pm)]+{S_n}_{\#}[f\res (\cup_{i \in I_E}\, E_i^\pm)],$$
%In this way, we get that 
%To fix the definitions of the arcs $E_{i,n}^\pm$
%which implies  that $\partial\Omega_n$ converges to $\partial\Omega$ in the Hausdorff distance.
%In particular, we see that $S_n({p^{i,n}_{k}}^\pm)={p^{i,n}_{k}}^\pm$ and $S_n({E^{i,n}_{k}}^\pm)={E^{i,n}_{k}}^\pm$ if ${E^{i,n}_{k}}^\pm$ is convex.
%I${E^{i,n}_{k}}^\pm$ is convexf $\Omega$ is convex, then due to \cite[Theorem 20.4]{Rockafellar} we can  always 
%The proof is by approximation of a convex set with strictly convex domains. This ispossible since if \,\,  we can
%find a sequence of decreasing polygons $\Delta_n$ containing \,$\Omega$ and
%contained in $\Omega_{n-1}$ 
%\todo{possible\\ problems}
%converging to $\Omega$ in the Hausdorff distance. After that we use the argument in \cite{rs1} to construct $\tilde\Omega_n$ strictly convex containing $\Delta_n$ and with arbitrarily small Hausdorff distance to $\Delta_n$. We can then obtain a decreasing sequence of strictly convex domains $\Omega_n$ such that $\overline{\Omega}_n \to \overline{\Omega}$ in the Hausdorff sense. Indeed, if we set $\Omega_n = \bigcap_{i=1}^n\tilde\Omega$, then we see that
%$$
%\Omega \subset \Omega_{n+1} = \tilde\Omega_{n+1}\cap \Omega_n \subset \Omega_n.
%$$
%Hence the sequence $\Omega_n $ has the desired properties. 

%{\it Step 2.} 
We shall prove that $\partial\Omega_n$ can be decomposed into arcs $\chi_{i,n}$ ($i \in I_\chi$) and $\Gamma_{i,n}^\pm$ ($i \in I_\Gamma$) satisfying conditions (H1), (H2) and (H3). %\marginnote{\PR{This doesn't\\ seem possible if $E^{n,\pm}_i$ are concave}}
%where for every $i$, $\partial C^{i,n} \cap \partial\Omega_n$ can be decomposed into arcs $\chi^{i,n}_j$ and $\Gamma^{i,n}_j$ 
%satisfying conditions (A1), (A2) and (A3). 
%However, Lemma \ref{Lemma 4.7} tells us that $(P|_{\partial\Omega_n})^{-1}$ will not preserve the structure of arcs $\chi_i$ or $\Gamma_j$ due to possible presence of singular points. They are elements of $S$,
%$$
%S=\{x\in \chi^\pm \cup \Gamma^\pm: \hbox{ there is no tangent line to }\partial\Omega \hbox{ at } x\}.
%$$
%We write 
%$S:=\bigcup_{i\in I_\chi\cup I_\Gamma} S_i$, where 
%$$
%S_i:= S^+_i \cup S^-_i := (S \cap \chi_i^+) \cup (S^-_i\cap \chi^-_i),\qquad
%(\hbox{resp. }S^\pm_i = S \cap \Gamma^\pm_i).
%$$
%Since the set $S$ is finite we may write
%$$
%S^+_i =\{s_{i,1}^+, \ldots, s^+_{i, n_i^+} \},\qquad
%S^-_i =\{s_{i,1}^-, \ldots, s^-_{i, n_i^-} \}. 
%$$
%in \chi_i^\pm\,\,(\mbox{resp.}\,\,\,s_{i,k}^\pm \in \Gamma_i^\pm) \,:\,1 \leq k \leq n_i^\pm\}$, where $n_i^\pm$ denotes the number of singularity points in $\chi_i^\pm$ (resp. $\Gamma_i^\pm$). 
%For every $i\in I_\chi$ (resp. $i\in I_\Gamma$) and
%$1 \leq k \leq n_i^\pm$, s
For every $i \in I_\chi\backslash I_S$ (resp. $i \in I_\Gamma\backslash I_S$), we have $\chi_{i,n}=\chi_i$ (resp. $\Gamma_{i,n}^\pm=\Gamma_i^\pm$). Since  $f_n(\chi_{i,n}) =0$ for $i \in I_S$, 
%If for all $j \in I_\Gamma^i$ the arcs ${\Gamma^i_j}^\pm$ are not line segments, then 
%we start by showing that (A1) and (A2) are satisfied. 
%we set ${\beta}_{i,n}:=S_n(\partial C_i \cap \partial\Omega) \cup (\partial C_i \cap \Omega)$
%{and we define $C_{i,n}$ to be the region bounded by ${\beta}_{i,n}$}. Moreover, we define the region $E_{i,n} \subset E_i$ as the open set
%bounded by $E_{i,n}^+\cup E_{i,n}^-\cup(\partial E_i \cap \Omega)$. It is not difficult to check that $C_{i,n}$ is convex.
%And, it is clear that $\partial C^{i,n} \cap \partial\Omega_n$ can be decomposed into arcs of the forms $\chi^{i,n}_j$ and $\Gamma^{i,n}_j$ 
%f $J^i:=\{j \in I_\Gamma^i : {\Gamma^i_j}^\pm \,\,\mbox{is a line segment}\} \neq \emptyset$, then it is not difficult to see that the set with boundary $S_n(\partial C_i \cap \partial \Omega) \cup (\partial C_i \cap \Omega)$ can be decomposed, up to removing the sets $E^{i,n}_j$ ($j \in J^i$), into sets $C_{i,n}^j$ ($j \in I^i_C \subset \mathbb{N}$).
%for all $i \in I_C$, so there is nothing to prove for $C_{i,n}$.
one can see that $\chi_{i,n}$ can be decomposed into two open arcs $\chi_{i,n}^+$,\ $\chi_{i,n}^-$  and a singleton $\{c_{i,n}\}$ such that $g_n$ is strictly increasing on ${\chi_{i,n}^+}$ and strictly decreasing on ${\chi_{i,n}^-}$ with $TV(g_{|{\chi_{i,n}^+}}) =  TV(g_{|{\chi_{i,n}^-}})$. In the same way, we can check that $g_n$ is strictly increasing on ${\Gamma_{i,n}^+}$ (resp. strictly decreasing on ${\Gamma_{i,n}^-}$) with $TV(g_{|{\Gamma_{i,n}^+}}) =  TV(g_{|{\Gamma_{i}^+}})$ (resp. $TV(g_{|{\Gamma_{i,n}^-}}) =  TV(g_{|{\Gamma_{i}^-}})$). 
Indeed, if $x_1,\,x_2 \in \partial\Omega_n$
%then we may compare these points. Hence,  we may assume that 
and $x_1<x_2$, %\marginnote{\PR{sorry,\\ writing $x<y$\\ is not right, {\color{red}{with respect to the orientation, we have already introduced this notion in Section 1.}}}}
then we see that $S_n^{-1}(x_1),\,S_n^{-1}(x_2) \in \partial\Omega$ satisfy $S_n^{-1}(x_1)<S_n^{-1}(x_2)$ and we have
$$f_n(\arc{x_1x_2})=f(\arc{S_n^{-1}(x_1)\,S_n^{-1}(x_2)}).$$
%We keep in mind the notation convention we introduced before 
%in the Introduction 
%regarding the definition of arcs. 
%Then, we define the trace function $g_n$ on $\partial \Omega_n$ such that $\partial_\tau g_n=f_n$
%as follows. For a fixed $x_0 \in \partial\Omega\setminus S$, we define $x_n := P_n^{-1}(x_0)\cap \partial\Omega_n$, then we set
 %   $$g_n(x):=f_n(\arc{x_{n}x}),\,\,\mbox{for all}\,\,\,x \in \partial\Omega_n.$$
  %  Thanks to the fact that $f$ is atomless, it is easy to see that $g_n \in C(\partial\Omega_n)$.
%On the other hand, we see that 
%if \,$\chi_i$\, is $C^1$ then the arc \,$\mathcal{C}_{i,n}=\mathcal{C}_{i,n}^+ \cup \mathcal{C}_{i,n}^-$, where
%$\mathcal{C}_{i,n}^\pm=P_n^{-1}(\chi_i^\pm)$ are two arcs of the forms $\chi_{i,n}^\pm$. First of all, we notice that \,$\overline{\mathcal{C}_{i,n}^+}\cap \overline{ \mathcal{C}_{i,n}^-} =P_n^{-1}(c_i)$ and, we have
%$$f_n(\mathcal{C}_{i,n}^+ \cup \mathcal{C}_{i,n}^-)=f(\chi_i)=0.$$
%Then, we see  that $g_n$ is strictly increasing (resp. decreasing) on $E_{i,n}^+$ (resp. $E_{i,n}^-$) since
%Moreover, if 
%\,$x_1,\,x_2 \in E_{i,n}^\pm$ such that 
%$x_1<x_2$ then 
%$S_n^{-1}(x_1),\,S_n^{-1}(x_2) \in E_{i}^\pm$ with \,
%$S_n^{-1}(x_1)<S_n^{-1}(x_2)$.
Hence, 
$$g_n(x_2)-g_n(x_1)=
 %f_n(\arc{x_1x_2})=f(\arc{S_n^{-1}(x_1)S_n^{-1}(x_2)})=
g(S_n^{-1}(x_2)) - g(S_n^{-1}(x_1)).$$\\
Let  \,$T_{i,n}$\, be the convex hull of \,$\Gamma_{i,n}^+$ and $\Gamma_{i,n}^-$ and $D_{i,n}$ be the convex hull of $\chi_{i,n}$. Then, we  have $D_{i,n} \subset D_i$ and $T_{i,n} \subset T_i$, as a result the sets $T_{j,n},$ $j\in I_\Gamma$,\  $D_{i,n}$, $i\in I_\chi$ are mutually disjoint.
%\marginnote{Not needed}
In addition, we have $|f_n|(\partial \Omega\backslash(\bigcup_{i\in I_\Gamma} (\Gamma_{i,n}^+\cup\Gamma_{i,n}^-)\cup \bigcup_{i\in I_\chi} \chi_{i,n}))=0$. Hence, condition (H1) is satisfied. 

Let ${\bf{T}}_n$ be the transport map defined by (\ref{df-bartildeT}) corresponding to $f_n$ and $\partial \Omega_n$. %$(\spt(f_n^+))^\circ$.
%\marginnote{$\bbT$ is NOT\\ defined on\\ $\spt f^+$!}
We recall that if \,$e^+$  is in the domain of ${\bf{T}}_n$ % \in (\spt(f_n^+))^\circ$ 
then one has %we see that 
$${\bf{T}}_n(e^+)=S_n({\bf{T}}(S_n^{-1}(e^+))).$$
Thanks to the fact that $]S_n^{-1}(x),{\bf{T}}(S_n^{-1}(x))[ \subset \Omega$\, for all $x \in \spt(f^+_n)$, it is clear that $]x,{\bf T}_n(x)[  \subset \Omega_n$ and so, (H2) is also satisfied. 
 
Now, we need to check  that assumption (H3) holds on $\partial\Omega_n$. Take a finite sequence $\{e_{k}^+\}_{1 \leq k \leq m}$ ($m \in \mathbb{N}^\star$) such that 
$e_k^+ \in \chi_{i_k,n}^+ \cup \Gamma_{i_k,n}^+$ for some $i_k \in I_\chi \cup I_\Gamma$.
%or \,${e_{k}}^\pm \in {E_{i_k,n}}^\pm$ (for some $i_k \in I_E$), with $i_k \neq i_{k^\prime}$ for all $k \neq k^\prime$ such that $\{i_k,\,i_{k^\prime}\} \subset I_C$ or $\{i_k,\,i_{k^\prime}\} \subset I_E$.
%$e_{k}^+ \in \partial C_{i} \cup  E_{i,n}^+$, for some $i \in \cI_C \cup \cI_E$. 
Since $g$ satisfies (H3) on $\partial\Omega$, we know that
$$\sum_{k=1}^{m} |S_n^{-1}(e_k^+) - {\bf{T}}(S_n^{-1}(e_k^+))| $$
\begin{align}\label{eq.4.8.reff}
&< 
\sum_{k=1}^{m-1}
|S_n^{-1}(e_k^+) - {\bf{T}}(S_n^{-1}(e_{k+1}^+))| + |S_n^{-1}(e_m^+ )-{\bf{T}}(S_n^{-1}(e_{1}^+))|
\\ \notag
 % \nonumber
&\leq\sum_{k=1}^{m-1}
|S_n^{-1}(e_k^+) - e_{k}^+|+|e_k^+ - {\bf{T}}_n(e_{k+1}^+)|+|{\bf{T}}_n(e_{k+1}^+)-{\bf{T}}(S_n^{-1}(e_{k+1}^+))|
\end{align}
$$\,\,\,\,\,\,\,+\,\,\,\,\,|S_n^{-1}(e_m^+) - e_{m}^+| 
+|e_m^+ - {\bf{T}}_n(e_{1}^+)| + |{\bf{T}}_n(e_{1}^+)-{\bf{T}}(S_n^{-1}(e_{1}^+))|.$$
\\

%\end{align}\\
Due to the definition of $S_n$, we have the following equality:\\
\begin{equation} \label{eq.4.6.ref}
|e_k^+ - {\bf{T}}_n(e_k^+)|=
|S_n^{-1}(e_k^+) - {\bf{T}}(S_n^{-1}(e_k^+))| - |e_k^+ - S_n^{-1}(e_k^+)| -|{\bf{T}}(S_n^{-1}(e_k^+)) - {\bf{T}}_n(e_k^+)|.
\end{equation}\\
%On the other hand,
%This observation leads to
%the strict triangle inequality below, so that 
%we see that }  
%\begin{align} \label{eq.4.7.reff}
%\sum_{k=1}^{m-1}
%&|S_n^{-1}(e_k^+) - {\bf{T}}(S_n^{-1}(e_{k+1}^+))| + |S_n^{-1}(e_m^+ )-{\bf{T}}(S_n^{-1}(e_{1}^+))|\leq 
%\sum_{k=1}^{m-1}
%|S_n^{-1}(e_k^+) - e_{k}^+| + |e_k^+ - {\bf{T}}_n(e_{k+1}^+)| \\
%&+ |{\bf{T}}_n(e_{k+1}^+)-{\bf{T}}(S_n^{-1}(e_{k+1}^+))|+|S_n^{-1}(e_m^+) - e_{m}^+| +|e_m^+ - {\bf{T}}_n(e_{1}^+)|
 %+ |{\bf{T}}_n(e_{1}^+)-{\bf{T}}(S_n^{-1}(e_{1}^+))|.\nonumber
%\end{align}%\\
After summing over $k$ in \eqref{eq.4.6.ref} and taking into account \eqref{eq.4.8.reff}, %with  \eqref{eq.4.7.reff}}, 
we infer that $g_n$ satisfies (H3) on $\partial\Omega_n$, i.e.  $$\sum_{k=1}^{m} |e_k^+ - {\bf{T}}_n(e_k^+)| 
 < \sum_{k=1}^{m-1}
 |e_k^+ - {\bf{T}}_n(e_{k+1}^+)|  + |e_m^+ - {\bf{T}}_n(e_{1}^+)|.$$\\
Condition (S) is satisfied on $\partial\Omega_n$ since it is a polygon with finitely many sides. Hence, thanks to Propositions \ref{Prop. convex case} and \ref{Uniqueness in the convex case}, $\gamma_n:=(Id,{\bf{T}}_n)_{\#}f_n^+$ is the unique optimal transport plan between $f_n^+$ and $f_n^-$.  We know that $\gamma_n \rightharpoonup \gamma$ for some $\gamma \in \mathcal{M}^+(\overline{\Omega} \times\overline{\Omega})$, up to a subsequence. Due to \eqref{r-Hd-co},  we deduce that $f_n \rightharpoonup f$. 
%Indeed, for any $\varphi \in C(\overline{\Omega})$, we have
%$$
%\lb f_n,\varphi\rb= \lb {S_n}_{\#}f,\varphi\rb=  \int_{\partial\Omega} \varphi(S_n(x))\,\mathrm{d}f(x)  \rightarrow  \int_{\partial\Omega} \varphi(x)\,\mathrm{d}f(x)=\lb f,\varphi\rb,$$
%{because due to (\ref{r-Hd-co})}, $S_n(x)$  converges to $x$, for all $x \in \partial\Omega$. 
%(thanks to the fact that $|S_n(x)-x| \leq \frac{1}{n}$). 
Thus, we also have  $(\Pi_x)_{\#}\gamma=f^+$ and $(\Pi_y)_{\#}\gamma=f^-$.
%On the other hand, let $u_n$ be a Kantorovich potential such that $u_n(x_0)=0$, for a fixed point $x_0 \in \overline{\Omega}$. Then, $u_n$, up to a subsequence, converges uniformly to a function $u \in \Lip_1(\overline{\Omega_{n_0}})$. Yet, thanks to the duality $\min\eqref{Kantorovich}=\sup\eqref{dual}$, we have
%$$\int_{\overline{\Omega}_{n_0} \times \overline{\Omega}_{n_0}}|x-y|\,\mathrm{d}\gamma_n=\int_{\overline{\Omega}_{n_0}} u_n\,\mathrm{d}f_n.$$
%Passing to the limit when $n \to \infty$
By Lemma \ref{l-conv},
%Similarly to {\it Step 4} in the proof of Proposition \ref{Prop. convex case with finite arcs}, 
we infer that $\gamma$ is an optimal transport plan between $f^+$ and $f^-$. %At the same time $u$ is the corresponding Kantorovich potential, because
%$$
%\min\eqref{Kantorovich} 
%\leq \int_{\overline{\Omega}_{n_0} \times \overline{\Omega}_{n_0}}|x-y|\,\mathrm{d}\gamma
%=\int_{\overline{\Omega}_{n_0}} u\,\mathrm{d}f \leq \sup \eqref{dual}.$$
%In addition, thanks to Proposition \ref{Prop. nonconvex case}, we have $\gamma_n=(Id,{\bf{T}}_n)_{\#}{f_n^+}$.
In addition, for any $x \in U^+\cup U^-$,
%{\color{black}NOT SURE, fixed (?)}
we have $S_n(x) \to x$ and ${\bf{T}}_n(S_n(x)) \to  {\bf{T}}(x)$, because ${\bf{T}}_n(S_{n}(x))=S_{n}({\bf{T}}(x))$.
%But, 
%$$u_n(P_{n}^{-1}(x))-u_n({\bf{T}}_n(P_{n}^{-1}(x)))=|P_{n}^{-1}(x)-{\bf{T}}_n(P_{n}^{-1}(x))|.$$
%Passing to the limit when $n \to \infty$, we get that
%$$u(x)-u({\bf{T}}(x))=|x-{\bf{T}}(x)|.$$
%In other words, this means that
Thus, this yields that $\gamma=(Id,{\bf{T}})_{\#}{f^+}$ is an optimal transport plan between $f^+$ and $f^-$. Since $C_1$ is convex, then recalling the proof of  Proposition \ref{Uniqueness in the convex case}, we infer that $\gamma^\star\res [C_1 \times C_1]=(Id,{\bf{T}})_{\#}{f^+}$.  
\end{proof}

Finally, we show the existence of a minimal vector field $v$ for Problem \eqref{Beckmann}. We note that the optimal transport plan $\gamma$ in Problem \eqref{Kantorovich} is not necessarily unique if condition ($\widetilde{A3}$) holds instead of (A3) (see Example \ref{non uniqueness of the optimal transport plan} below). Despite that, we will be able to prove  uniqueness of the minimal vector field $v$ in Problem \eqref{Beckmann} anyway.
%solution $u$ in the case when $\Omega$ is not convex.
%More precisely, we have the following: %\todo{comments}
\begin{proposition}  \label{Uniqueness in the nonconvex case} 
Under the assumptions (A1), (A2) $\&$ ($\widetilde{A3}$), Problem \eqref{Beckmann} has a unique minimizer provided that $g$ is piecewise monotone.
%$\&$ (H4) hold. %and 
%Suppose that $\Omega^+$ is  strictly convex, $f^+$ is non-atomic and, $f^+$ and $f^-$ have no common mass. 
%Moreover, assume that (H1), (H2), (H3) $\&$ (H4) hold.
%$\gamma:=(Id,{\bf{T}})_{\#}f^+$ is the unique optimal transport plan.
%, between $f^+$ and $f^-$, which will be induced by a transport map $S$,
%provided that $f^+$ is nonatomic.
\end{proposition}
\begin{proof}
%In order to prove that the solution $u$ of Problem \eqref{LGP} is unique, we just need to show that the optimal flow $v$ of Problem \eqref{Beckmann} is unique. 
Due to Proposition \ref{General}, we know that $\gamma^\star:=(Id,{\bf{T}})_{\#}f^+$ is an optimal transport plan between $f^+$ and $f^-$. Since $]x,{\bf{T}}(x)[ \subset \Omega$, for all \,$x \in \spt(f^+)$, %\marginnote{$\bbT$ is NOT\\ defined at\\ the bdry of\\ its spt!}
then the vector measure $v_{\gamma^\star}$ is well defined and it turns out to be a solution for Problem \eqref{Beckmann}. In particular, we have  $\eqref{Kantorovich}=\eqref{Beckmann}$ and so, we recall, see for instance \cite[Chapter 4]{Santambrogio}, that in this case any minimizer $v$ of Problem \eqref{Beckmann} will be of the form $v=v_\gamma$ for some optimal transport plan $\gamma$ and for all $(x,y) \in \spt(\gamma)$, we must have $[x,y] \subset \overline{\Omega}$.
Let $v=v_\gamma$ be such a minimizer. Then, we claim that
%Let $\gamma$ be an optimal transport plan in Problem \eqref{Kantorovich}. Assume that 
$\gamma = (Id,{\bf{T}})_{\#}f^+$. 
%Assume that this is not the case, 
Indeed, let us set
$$
A = \bigg\{ x\in \partial \Omega: \exists\,\,y\neq \bbT(x),\,(x,y)\in \spt(\gamma)\bigg\}
$$
and assume that $f^+(A)>0$.
%has a positive measure.
Then, one can see exactly as in the proof of Proposition \ref{Uniqueness in the convex case} that for $f^+-$a.e. $x \in A$, there are two different transport rays starting at $x$, 
%Hence, there will be a set $A\subset\partial\Omega$ with $f^+(A)>0$ and such that for $\gamma-$a.e. $(x,y) \in A \times \partial\Omega$, we have $y \neq {\bf{T}}(x)$. 
%or every $i \in \mathcal{I}_\chi$ (resp. $i \in \mathcal{I}_\Gamma$), there will be a set $A_i \subset \chi_i^+$ (resp.  $A_i \subset \Gamma_i^+$) such that $f^+(A_i)>0$ and for $\gamma-$a.e. $(x,y) \in A_i \times \partial\Omega$, we have $y \neq {\bf{T}}(x)$. 
%by Proposition \ref{Prop. convex case}, 
$[x,{\bf{T}}(x)]$ and $[x,y]$ which are contained in $\overline{\Omega}$, thus 
%is a transport ray for $f^+-$a.e. $x$.  %and thanks to the fact 
%We know that two transport rays cannot intersect at an interior point of $\Omega$, as a result
%one of them, then 
%we see that 
$A$ must be contained in the set of endpoints of arcs of type $\Gamma_i^\pm$, $\chi_i$, or $E_i^\pm$. %They belong to a null set $U_0$. %\spt(f^+)$.
%{for this to be true we must assume that $\spt(f^+)$ is at most a countable sum of intervals}. 
Hence, $A$ is at most countable and $f^+(A)=0$. $\qedhere$
%\PR{I'M LOST}
%Thanks to Proposition \ref{Prop. nonconvex case}, we see that if \,$\gamma$\, is an optimal transport plan then 
%This implies that $\gamma=(Id,{\bf{T}})_{\#}f^+$
%, where ${\bf{T}}(x)=\bbT^i(x)$ for $f^+-$a.e. $x \in \partial\Omega_i$. But, this implies that $\gamma$
%is the unique optimal transport plan in Problem \eqref{Kantorovich}.
%since, if $\gamma^\prime$ is another optimal transport plan then \,$\gamma^{\prime\prime}=(\gamma + \gamma^\prime)/2$\, is also optimal, which is not possible as $\gamma^{\prime\prime}$ must be induced by a transport map.  
%On the other hand, we recall that, under the  (H1), (H2), (H3) \& (H4), all the transport rays lie inside $\Omega$ 
\end{proof}
Consequently, we get the following extension of Theorem \ref{Theorem 4.7}:
\begin{theorem}\label{thm: weak A3}
Under the assumptions (A1), (A2) $\&$ ($\widetilde{A3}$), Problem \eqref{LGP} has a unique solution provided that $g$ is piecewise monotone. 
\end{theorem}

\begin{proof}
The argument is similar to %will be done as in 
the proof of Theorem \ref{Theorem 4.7} and it is left to the interested reader. 
\end{proof}

Finally, we present an example where (A3) is violated but ($\widetilde{A3}$) is satisfied and a solution to Problem \eqref{LGP} exists.
\begin{example} \label{non uniqueness of the optimal transport plan} \rm
 Fix \,$0<a<b$. Then, the domain $\Omega$ as shown in Figure \ref{non uniqueness} is formed from the vertices of the squares $[-a,a]^2$ and $[-b,b]^2$. We define the  boundary data as follows:
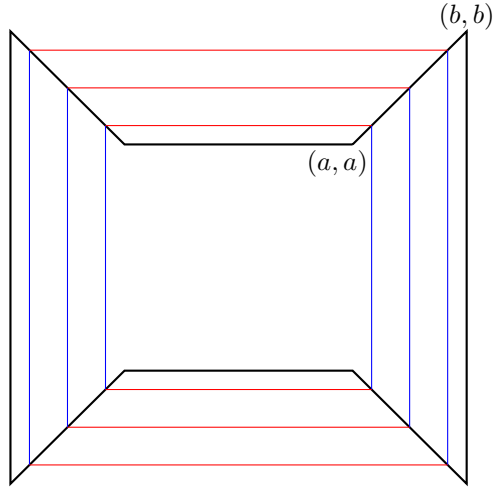
\begin{figure}[h]
\begin{tikzpicture}[scale=0.50]
\draw[thick] (3,3)--(6,6)--(6,-6)--(3,-3)--(-3,-3)--(-6,-6)--(-6,6)--(-3,3)--(3,3);
\draw[blue] (3.5,-3.5)--(3.5,3.5);
\draw[blue] (4.5,-4.5)--(4.5,4.5);
\draw[blue] (5.5,-5.5)--(5.5,5.5);
\draw[blue] (-3.5,-3.5)--(-3.5,3.5);
\draw[blue] (-4.5,-4.5)--(-4.5,4.5);
\draw[blue] (-5.5,-5.5)--(-5.5,5.5);
\draw[red] (3.5,3.5)--(-3.5,3.5);
\draw[red] (4.5,4.5)--(-4.5,4.5);
\draw[red] (5.5,5.5)--(-5.5,5.5);
\draw[red] (-3.5,-3.5)--(3.5,-3.5);
\draw[red] (-4.5,-4.5)--(4.5,-4.5);
\draw[red] (-5.5,-5.5)--(5.5,-5.5);

 \node at (2.6,2.5) {$(a,a)$};
 \node at (6,6.4)  {$(b,b)$};
\end{tikzpicture}
    \caption{Example of non uniquenss}
        \label{non uniqueness}
\end{figure}

$$g(x_1,x_2)=\begin{cases} x_1 &\quad x_2=\pm x_1,\,\,\,\,a\leq x_1\leq b, \\
-x_1 &\quad x_2=\pm x_1,\,\,-b\leq x_1\leq-a,\\
a &\quad x_2=\pm a,\,\,\,-a\leq x_1\leq a,\\ 
%b &\quad x_2=\pm b,\,\,\,-b\leq x_1\leq b\\
b &\quad x_1=\pm b,\,\,\,-b\leq x_2\leq b.\\ 
\end{cases}$$\\
Notice that (A3) is violated in this case and the blue segments and the red ones correspond to all possible transportation rays between $f^+$ and $f^-$. In particular, we can construct two optimal transport plans $\gamma_b$ and $\gamma_r$, where $\gamma_b$ is supported on the blue segments while $\gamma_r$ is supported on the red ones. This means that the optimal transport plan is not unique. However,
only $\gamma_b$ corresponds to a solution to the least gradient problems since its transport rays are included in $\Omega$.
\end{example}

We close the paper with an observation on the continuity of the solution $u$ in Problem \ref{LGP}.
\begin{theorem} 
Assume that (A1), (A2) $\&$ ($\widetilde{A3}$) hold and $g\in C(\partial\Omega)$ is piecewise monotone. Then, the unique solution $u$ of Problem \eqref{LGP} is continuous in $\overline{\Omega}$.  
%\PR{Is (A3$^{\,\prime}$)= ($\widetilde {A3}$)?}
\end{theorem}
\begin{proof}
%First of all, let us denote by $\tilde{g}$ a continuous extension of $g$ on $\mathbb{R}^2\setminus\Omega$.
Assume that there is a point $x_0 \in \overline{\Omega}$ such that $u$ is discontinuous at $x_0$. Then, there exist two numbers $t_1$ and $t_2$ such that 
$$\lim_{n \to \infty}\mbox{ess\,inf}_{B(x_0,\frac{1}{n})} \,u< t_1 <t_2<\lim_{n \to \infty}\mbox{ess\,sup}_{B(x_0,\frac{1}{n})}\, u.$$
%In particular, this means that there are two sequences $x_n \to x_0$ with $u(x_n)>t_2$ and $x_n^\prime \to x_0$ with $u(x_n^\prime)<t_1$.
Now, consider the super-level sets $E_{t_1}:=\{u \geq t_1\}$ and $E_{t_2}:=\{u \geq t_2\}$. Then, we see that $x_0 \in \overline{E_{t_2}} \cap  \overline{\mathbb{R}^2 \setminus E_{t_1}}$. However, we have $E_{t_2} \subset E_{t_1}$. Hence, we infer that $x_0 \in \partial  E_{t_1} \cap \partial  E_{t_2}$. So, the only possibility is to have $x_0 \in \partial\Omega$. Let $E$ be the set bounded by $\partial  E_{t_1}$, $\partial  E_{t_2}$ and $\partial\Omega$. It is clear that for $\varepsilon>0$ small enough, the transport density $\sigma=0$ on $E \cap B(x_0,\varepsilon)$. 
%\marginnote{\PR{what is $\sigma$ here?}}
Then, $u$ is constant on $E$ and since $t_1 <t_2$ then this means that $u$ has a jump on $\partial  E_{t_1}$ or $ \partial  E_{t_2}$. But, this yields  a contradiction thanks to the fact that $f$ is atomless and so, $\sigma$ gives zero mass to any transport ray.   
\end{proof}

\section*{Acknowledgement}
The project was initiated when S. Dweik was employed by the Faculty of Mathematics, Informatics and Mechanics of the University of Warsaw. P. Rybka was partially supported by the National Science Centre, Poland through the grant 2017/26/M/ST1/00700. 
A. Sabra would like to thank the National Science Centre, Poland, for partially supporting  his visit to the University of Warsaw through the grant 2017/26/M/ST1/00700, and the Center for Advanced Mathematical Sciences (\url{https://orcid.org/0009-0004-5763-5004}) for partially funding the research on this project.


\begin{thebibliography}{99}
\bibitem{Aliprantis}{\sc
C. Aliprantis and K. Border}, Infinite Dimensional Analysis. A Hitchhiker’s Guide, {\it Springer,} Berlin, {\bf $3^{rd}$ edition} (2006).

\bibitem{czle}
{\sc S. Czarnecki and t. Lewi\'nski}, A stress-based formulation of the free material design problem with the trace constraint and multiple load conditions, {\it Structural and Multidisciplinary Optimization,} {\bf 49} (2014), no. 5, 707--731
%

\bibitem{SD}{\sc S. Dweik},
The least gradient problem with Dirichlet and Neumann boundary conditions, {\it{preprint}}, 2023.



%
\bibitem{DG22}{\sc S. Dweik and W. G\'orny},
Least gradient problem on annuli, 
{\it Anal. PDE} {\bf 15} (2022), no. 3, 699–725.
%
\bibitem{DweGor}{\sc S. Dweik and W. G\'orny},
Optimal transport approach to Sobolev regularity of solutions to the weighted least gradient problem, {\it{SIAM J. Math. Anal.}}, {\bf 55} (2023), no. 3, 1916–1948.

\bibitem{DweSan}{\sc S. Dweik and F. Santambrogio}, $L^p$ bounds for boundary-to-boundary transport densities, and $W^{1,p}$ bounds
for the BV least gradient problem in 2D, {\it{Calc. Var. Partial Differential Equations}}, {\bf  58}, no. 1, 2019.


\bibitem{SD}{\sc S. Dweik},
The least gradient problem with Dirichlet and Neumann boundary conditions, {\it{preprint}}, 2023.
%

\bibitem{WG1}{\sc W. G\'orny}, Existence of minimisers in the least gradient problem for general boundary data, {\it{Indiana Univ. Math. J}}, {\bf  70}, no. 3 (2021), pp. 1003-1037.




\bibitem{goma}
{\sc W. Górny and J.M. Mazón,} Functions of Least Gradient, {\it Birkh\"auser}, Cham, 2024.

\bibitem{grs}
{\sc W. G\'orny, P. Rybka and A. Sabra,}
Special cases of the planar least gradient problem, {\it Nonlinear Analysis}, {\bf 151} (2017), 66-95.
%
 \bibitem{mazon}
{\sc J.M. Maz\'on, J.D. Rossi and S. Segura de Le\'on,}
Functions of least gradient and 1-harmonic functions,
{\it Indiana Univ. Math. J.} {\bf  63} (2014), no. 4, 1067--1084.
%
\bibitem{Rockafellar}
{\sc R.T. Rockafellar}, Convex analysis, {\it Princeton Review Press} {\bf vol 11}, Princeton, 1997.


%

\bibitem{rs1}
{\sc P. Rybka and A. Sabra,}   The planar Least Gradient problem in convex domains, the case of continuous datum, {\it Nonlinear Analysis,} {\bf 214} (2022), 112595.
%
\bibitem{rs2}
{\sc P. Rybka and A. Sabra,} The planar Least Gradient problem in convex domains: the discontinuous case, {\it Nonlinear Differ. Equ. Appl.}, {\bf 28} (2021), 15.
%
\bibitem{Santambrogio}
{\sc F. Santambrogio,} {\it{Optimal Transport for Applied Mathematicians}}, in {\it{Progress in Nonlinear Differential Equations and Their Applications}} 87, {\it Birkhäuser}, Basel, 2015.
%
\bibitem{tamasan}
{\sc G. Spradlin and A. Tamasan}, Not all traces on the circle come from functions of least gradient in
the disk,  {\it Indiana Univ. Math. J.} {\bf 63} (2014), 1819–1837.
%
\bibitem{sternberg}
{\sc P. Sternberg, G. Williams and W.P. Ziemer,} Existence, uniqueness, and regularity for functions of least gradient, {\it J. Reine Angew. Math.} {\bf 430} (1992), 35--60. 
%
%\bibitem{valadier} C.Castaing, M.Valadier,  Convex analysis and measurable multifunctions. Lecture Notes in Mathematics, Vol. 580. Springer-Verlag, BerlinNew York, 1977.
%
\end{thebibliography}
\end{document}